\documentclass[11pt]{article}
\usepackage{graphicx}
\usepackage{amsmath}
\usepackage{amssymb}
\usepackage{amsthm}
\usepackage{epsfig}
\usepackage{mathrsfs}
\usepackage{multirow}
\usepackage[usenames,dvipsnames]{color}
\usepackage{setspace}
\usepackage{algorithm, algorithmic} 
\usepackage{multicol}
\usepackage{subcaption}
\usepackage[margin=2.54cm]{geometry}
\usepackage{lipsum}
\usepackage{bigints}
\usepackage{framed}
\usepackage{colortbl}
\usepackage{calligra}
\usepackage{mathrsfs}
\usepackage{enumerate}
\usepackage{hyperref}
\usepackage{enumitem}
\usepackage{natbib}
\usepackage{cleveref}
\usepackage{xcolor}
\usepackage{mathtools}

\usepackage{appendix}

\DeclareMathOperator{\st}{s.t.}

\DeclareMathOperator{\Diag}{Diag}

\DeclareMathAlphabet{\mathcalligra}{T1}{calligra}{m}{n}

\DeclareMathOperator{\fr}{F}
\DeclareMathOperator{\proj}{proj}
\DeclareMathOperator{\vspan}{span}

\newtheorem{theorem}{Theorem}

\newtheorem{corollary}{Corollary}
\newtheorem{lemma}{Lemma}

\newtheorem{proposition}{Proposition}

\newtheorem{remark}{Remark}
\newtheorem{example}{Example}

\newcommand{\PP}{\mathbb P}

\newcommand{\SOC}{\mathcal{SOC}}

\newcommand{\bmath}{\mathbf}
\newcommand{\bm}{\boldsymbol}
\newcommand{\bs}{\boldsymbol}

\def\RR{ {\mathbb{R}}}

\newcommand{\X}{{\mathcal X}}

\newcommand{\mc}[1]{\mathcal{#1}}
\newcommand{\mb}[1]{\mathbb{#1}}

\newcommand{\xit}{\tilde{\bm{\xi}}}
\newcommand{\gammat}{\tilde{\bm{\gamma}}}

\newcommand{\x}{{\bm{x}}}

\newcommand{\indep}{\perp \!\!\! \perp}
\newcommand{\gm}{\gamma}
\newcommand{\bxi}{\bm \xi}

\newcommand{\VVrho}{{\mathbb{V}_{\bm\gamma}}}
\newcommand{\VVrhohat}{{\hat{\mathbb{V}}_{\bm\gamma}}}
\newcommand{\EE}{{\mathbb{E}}}
\newcommand{\EErho}{{\mathbb{E}_{\bm\gamma}}}
\newcommand{\EErhohat}{{\hat{\mathbb{E}}_{\bm\gamma}}}

\title{
On Data-Driven Prescriptive Analytics with Side Information:\\ A Regularized Nadaraya-Watson Approach
}

\author{
Prateek R. Srivastava\thanks{Graduate Program in Operations Research and Industrial Engineering, 
The University of Texas at Austin, Austin, TX, 78712-1591, USA. Email: {\tt \{prateekrs,yijie-wang,grani.hanasusanto\}@utexas.edu}.}
\ \ \ \ \ \ \
Yijie Wang\footnotemark[1]
\ \ \ \ \ \ \
Grani A. Hanasusanto\footnotemark[1]
\ \ \ \ \ \ \
Chin Pang Ho\thanks{School of Data Science, City University of Hong Kong, Hong Kong. Email: {\tt clint.ho@cityu.edu.hk}.}%
}

\date{}

\begin{document}

\maketitle
\begin{onehalfspace}
\begin{abstract}
We consider generic stochastic optimization problems in the presence of side information which enables a more insightful decision. The side information constitutes observable exogenous covariates that alter the conditional probability distribution of the random problem parameters. A decision maker who adapts her decisions according to the observed side information solves an optimization problem where the objective function is specified by the conditional expectation of the random cost. If the joint probability distribution is unknown, then the conditional expectation can be approximated in a data-driven manner using the Nadaraya-Watson (NW) kernel regression. While the emerging approximation scheme has found successful applications in diverse decision problems under uncertainty, it is largely unknown whether the scheme can provide any reasonable out-of-sample performance guarantees. In this paper, we establish guarantees for the generic problems by leveraging techniques from moderate deviations theory. Our analysis motivates the use of a variance-based regularization scheme which, in general, leads to a non-convex optimization problem. We adopt ideas from distributionally robust optimization to obtain tractable formulations. We present numerical experiments for newsvendor and wind energy commitment problems to highlight the effectiveness of our regularization scheme. 
\noindent  \\ \\
 \noindent Keywords: stochastic optimization; side information; Nadaraya-Watson estimator; moderate deviation principles; large deviation principles; distributionally robust optimization
%The new theoretical result motivates the design of an effective regularization scheme via empirical conditional standard deviation. We highlight the performance of the regularized NW approximation through an example in portfolio management. 
%{\color{blue} mention about the DRO formulation}
%\mbox{}

%However, in general, a variance-based regularization scheme leads to a non-convex formulation, and therefore, is intractable. We develop convex reformulations for the problem based on piecewise linear approximation schemes and ideas from distributionally robust optimization. We highlight the effectiveness of our regularized NW approximation scheme for example problems in different domains, which include portfolio optimization and inventory management problems. 

%\textbf{INFORMS ABSTRACT:} We consider the stochastic optimization problem with side information, which aims to minimize  the conditional expected loss in the presence of observable exogenous covariates. In general, the joint distribution of the side information and the loss function is unknown. Instead, only historical data is available. We propose an approximation based on the Nadaraya Watson estimator and derive out-of-sample performance guarantees based on moderate deviations theory.  Our analysis leads to a variance-based regularization scheme, which is generally non-convex. We adopt ideas from distributionally robust optimization to obtain equivalent tractable formulations. We present numerical experiments for portfolio optimization and newsvendor problems.

\end{abstract}

%\begin{keyword}
 
%\end{keyword}

\section{Introduction} \label{sec:intro}
In the presence of uncertainty, decisions can often be improved by taking into account the side information, such as weather conditions, interest rates, exchange rates,  past prices and demands, volatility indices, etc., that provides a more accurate description of the uncertain problem parameters. In the stochastic optimization setting, the side information corresponds to observable exogenous covariates~$(\gamma_1,\ldots,\gamma_p)$ that may reshape the conditional probability distribution of the random problem parameters $(\tilde\xi_1,\ldots,\tilde\xi_q)$. A decision maker prescribed with full knowledge about  the joint distribution of the random vectors $\gammat\coloneqq(\tilde\gamma_1,\ldots,\tilde\gamma_p)$ and $\xit\coloneqq(\tilde\xi_1,\ldots,\tilde\xi_q)$ endeavors to solve the following stochastic optimization problem with side information:
\begin{equation}\label{eq:conditional_expectation_problem}
\min_{\bm x\in\mathcal X}\;\left\{ \EErho[\ell(\bm x,\xit)]\coloneqq\EE[\ell(\bm x,\xit)\,|\,\gammat=\bm\gamma]\right\}.
\tag{$\mathcal{SO}$}
\end{equation}
Here, the vector $\bm x\in\RR^d$ comprises all decision variables, while the objective function is specified through the conditional expectation of the random cost $\ell(\bm x,\xit)$ given the side information $\bm\gamma$. 

For instance, in the context of portfolio optimization---which aims to maximize the expected portfolio return---the loss function is defined as  $\ell(\bm x,\xit)\coloneqq - \xit^\top\bm x$, where $\xit\in\RR^q$ ($d=q$ in this case) and $\bm x$ correspond respectively to the vectors of random asset returns and allocated investments. If short selling is prohibited, then the feasible set $\mathcal X$ of the allocation vector $\bm x$ is described by the unit $q$-simplex. In this problem, the exogenous covariate vector $\gammat$ may additionally comprise the firms' market capitalizations, book-to-market ratios, past returns, and also include other market indicators such as the volatility indices and financial news indicators \citep{brandt2009parametric,baziergeneralization}. 
We now illustrate the importance of side information through the following example. 

\begin{example}[A Three-Asset Portfolio]\label{ex:finance}
Consider a stylized three-asset portfolio optimization problem, in which the decision maker allocates a total wealth of $\$1$. The return of asset $i$ is
\begin{equation*}
\tilde\xi_i(\tilde\gamma) = \left\{
\begin{array}{ll}
0.5 - \tilde\gamma^2 + 0.1 \cdot \tilde\epsilon_i  & \forall i = 1,2, \\
 0    & i=3,
\end{array}
\right.
\end{equation*}
where the side information/covariate $\tilde\gamma \in \mathbb{R}$ is governed by a uniform distribution on the interval $[-1,1]$. The random variables $\tilde\epsilon_1$ and $\tilde\epsilon_2$ are assumed to be bivariate normally distributed with zero means and unit variances and are perfectly negatively correlated.

%Given the setting in Example~\ref{ex:finance}, 
Under this setting, the unconditional expected returns of the risky assets (i.e., $i=1,2$) are equal to 
%\begin{equation*}
$\EE\left[0.5 - \tilde\gamma^2 + 0.1\cdot\tilde\epsilon_i \right] = 1/6$. %, $i=1,2$.
%\end{equation*} 
Thus, in the absence of any side information, the optimal expected portfolio return is $1/6$, which can be obtained by allocating the entire wealth into any convex combination of the risky assets.

However, suppose that the value of the side information $\gamma$ is revealed before the decision is made. In this case, the conditional expected return of each risky asset is  $ \frac{1}{2} - \gamma^2$. Hence, when $\gamma^2 < 1/2$, it is optimal to allocate the entire wealth into any convex combination of the risky assets; otherwise, it is optimal to allocate the entire wealth into the risk free asset. Since $\tilde\gamma$ follows a uniform distribution on $[-1,1]$, the optimal expected return of this strategy is given by
\begin{equation*}
\int_{-\sqrt{\frac{1}{2}}}^{\sqrt{\frac{1}{2}}} \; \frac{1}{2} \left( \frac{1}{2} - \gamma^2 \right) \; \text{d} \gamma = \frac{2}{3} \left( \frac{1}{2}\right)^{3/2} = \frac{1}{3\sqrt{2}}.
\end{equation*}
The above calculations show that the expected return deteriorates by $(1/(3\sqrt{2}) - 1/6)/(1/(3\sqrt{2})) \approx 29 \%$ if the portfolio manager ignores the side information. 
\end{example}
\noindent The example highlights the critical benefits of exploiting the side information in our decision making processes, when such information is available.

%\subsection{Related Work}
\subsection{Literature Review}
In the ideal case, solving \eqref{eq:conditional_expectation_problem} exactly allows us to make optimal decision with side information. However, in most situations of practical interest, the joint distribution of $(\gammat,\xit)$ is unknown, and only past historical data $\{(\bm\gamma^1,\bm\xi^1),\ldots,(\bm\gamma^n,\bm\xi^n)\}$ is available to infer the conditional distribution of $\xit$ and to estimate the conditional expectation in \eqref{eq:conditional_expectation_problem}. In recent years, there has been a focus on developing integrated learning and optimization frameworks to approximate the optimal solution for \eqref{eq:conditional_expectation_problem} with statistical guarantees on their performances. \cite{bertsimas2014predictive} consider different machine learning approaches to construct empirical conditional expectations that well approximate the conditional expectation in \eqref{eq:conditional_expectation_problem}. They further establish that the resulting approximations are asymptotically consistent, meaning that the approximations converge to the true conditional expectation as the sample size grows. \cite{bertsimas2019predictions} extend the result of \cite{bertsimas2014predictive} to the multistage setting under the assumption that covariates evolve according to a Markov process. The resulting data-driven decision is shown to be consistent and asymptotically optimal, and finite-sample guarantees are developed for k-nearest neighbors (KNN)-based approaches. Solutions to their proposed formulations, however, exhibit an optimistic bias if the sample size is small. 

To mitigate this overfitting effect,~\cite{hanasusanto2013robust} propose a robust version that minimizes a worst-case empirical conditional expectation in view of the most adverse weight vector that is close to the nominal one generated by the Nadaraya-Watson (NW) estimator. ~\cite{bertsimas2019dynamic} incorporate side information into robust dynamic programming problems and establish that the solution is asymptotically optimal for multi-period stochastic programs.
~\cite{bertsimas2017bootstrap} propose an alternative robust scheme whose solutions enjoy a limited disappointment on the bootstrap data.~\cite{esteban2020distributionally} construct a  framework %to mitigate intrinsic error in the process of inferring conditional information from limited joint data 
using trimmings of probability distributions, which they prove to be connected with the partial mass transportation problem and show that the approach naturally produces distributionally robust optimization  (DRO) extensions of formulations with some nonparametric regression techniques. % They connect trimmings of probability distributions with the partial mass transportation problem and show that the approach naturally produces DRO extensions of formulations with some nonparametric regression techniques.
% Their framework produces DRO extensions of formulations and allows for contaminated data. They also show that the solution is asymptotically optimal and  provides finite-sample guarantees.

There exist other powerful and interesting approaches that solve \eqref{eq:conditional_expectation_problem} under more specific settings. For example, \citet{sen2018learning} and \cite{ban2019dynamic} first consider regression models with additive residual terms to model and generate scenarios for~$\xit$ given side information $\bm\gamma$. Inspired by their work and the sample average approximation scheme for classical stochastic optimization problems, \cite{kannan2020data} propose a formulation based on a regression model that assumes~$\xit$ to be modeled in terms of $\bm \gamma$ as  $\xit = f(\tilde{\bm\gamma})+\tilde{\bm\epsilon}$, where $f(\bm\gamma)=\mb E[\xit\lvert \tilde{\bm\gamma} = \bm\gamma ]$ and $\tilde{\bm\epsilon}$ are mean zero errors. This formulation, however, relies on the crucial assumption that the distribution of the errors $\tilde{\bm\epsilon}$ is independent of the covariates $\tilde{\bm \gamma}$, which allows them to formulate the problem as a sample average approximation problem that assigns an equal weight of $1/n$ to each observation. With the idea of obtaining better out-of-sample performances on problems with limited data, the authors incorporate their residual-based formulation into a DRO framework \citep{kannan2020residuals} and also consider extensions where they relax the homoscedasticity assumption on the residuals \citep{kannan2021heteroscedasticity}. In a similar spirit,~\cite{elmachtoub2021smart} propose a smart ``Predict, then Optimize" framework for contextual optimization problems with an unknown linear objective. Building upon the above ideas, \cite{sim2021} propose a robustness optimization counterpart for the robust satisficing framework. In this paper, we focus on the setting without assuming the regression models for $\xit$.%, so they are beyond the scope of this paper.
 %{\color{blue} Do you know which paper is first posted online? \cite{elmachtoub2021smart} first or \cite{kannan2020data} first? Perhaps we can add "In parallel to this line of work, xxx..."}

Despite the practical significance of the stochastic optimization problem \eqref{eq:conditional_expectation_problem}, there is an incomplete picture of the  properties of the existing solution schemes. Although the NW approximation is shown to be asymptotically consistent \citep{bertsimas2014predictive}, it is unknown whether the scheme could provide out-of-sample performance guarantees for solutions to the generic problems. An alternative method that optimizes over parametric \emph{decision rules}, such as linear or quadratic functions in $\bm\gamma$, can generate finite-sample performance bounds \citep{bertsimas2014predictive,ban2018big,baziergeneralization}. In \cite{brandt2009parametric}, the portfolio optimization with side information model is solved in view of linear decision rules (LDR) where one seeks for the best linear policy in the exogenous covariates that maximizes the empirical return. An $\ell_2$-regularized version of the linear decision rule approximation is studied in \cite{baziergeneralization}. The decision rules scheme, however, is less attractive because it is not asymptotically consistent, meaning that we cannot  produce results that would parallel those of sample-average approximation in the classical setting of  stochastic optimization  without side information~\citep{kleywegt2002sample,Shapiro09}. In~\cite{ban2018big}, the authors apply both the NW and decision rule approximations to the single-item newsvendor problem and derive finite-sample performance guarantees for the solutions. Unfortunately, the bound for the NW approximation inconveniently relies on an optimal solution to the corresponding linear decision rule problem. An alternative bound derived in \cite{bertsimas2017bootstrap} holds only for the bootstrap data, that is generated via resampling from the empirical distribution. Although encouraging, their bound does not provide a complete understanding on its out-of-sample performance. %might be deceiving because it does not depend on the dimension $p$ of the exogenous covariate vector~$\gammat$.

\subsection{Our Contributions}
This paper focuses on the approximation scheme using the popular NW kernel regression estimator \citep{nadaraya:64,watson:64}. %, which is asymptotically consistent. 
By leveraging techniques from \emph{large and moderate deviations theory}, we derive for the first time out-of-sample performance guarantees for the empirical conditional expectation minimization model. Our result indicates that the out-of-sample errors of the approximation scale with $O(\sqrt{1/(nh^p)})$, where $h > 0$ is the bandwidth parameter that is used for the kernel function in our proposed model. In contrast to the result in \cite{ban2018big} for a single-item newsvendor problem, our guarantees hold independently of optimal solutions to the corresponding linear decision rule problems and   conform with the best bandwidth parameter scaling $h=O(1/n^{1/(p+4)})$ suggested in the literature. As a byproduct of our new theoretical result, we identify a suitable regularization term in   \emph{empirical conditional standard deviation}. If this term is small, then our guarantees imply that the out-of-sample errors are of the lower rate  $\sim O({1/(nh^p)})$. Thus, the regularization term will encourage an optimal  solution that  yields small generalization errors. We devise a solution scheme for this variance regularized formulation based on a distributionally robust optimization (DRO) problem. Numerical results in the context of newsvendor and wind energy commitment problems demonstrate the superiority of our new regularized NW approximation over the linear decision rule scheme and a state-of-the-art DRO framework proposed by~\cite{kannan2020residuals}.

%  \\
% Duchi paper -\\
% {\color{red}
% \cite{ban2018big} solution learning - learn a mapping from the covariate to the solution $\bm x$ by assuming a functional form (linear decision rules, etc.) for the newsvendor problem
% learning step is embedded in the optimization step \\
% Also, suvrajeet sen paper \cite{sen2018learning} ? multistage setting?
% }

We summarize below the main contributions of the paper: 
\begin{enumerate}
\item Leveraging techniques from large and moderate deviations theory, we derive generalization bounds for the NW estimator. Unfortunately, typical for settings where kernel functions are used, the bound suffers from the \emph{curse of dimensionality}, which becomes prominent when the side information vector~$\bm\gamma$ is high-dimensional. 
%The bound suffers from the \emph{curse of dimensionality}, which is typical for settings where kernel functions are used and the side information vector~$\bm\gamma$ is high-dimensional
We propose a dimensionality reduction scheme based on principal component analysis (PCA) that strengthens our obtained bounds for the case where the intrinsic dimensionality of $\bm\gamma$ is small, even though the dimensionality of the ambient space may be large.
\item Our generalization bound motivates the use of a variance-based regularization scheme, where in addition to the empirical conditional expectation %of the random cost function 
specified by the NW estimator, we minimize a penalty term that corresponds to empirical conditional standard deviation of the random cost function in the objective. 
%Althoug a similar bound The obtained bound assumes the same form as the \citep{maurer2009empirical} in the empirical risk minimization literature. 
Furthermore, we derive suboptimality bounds for the optimal solution $\bm x^\star$ obtained for this variance-regularized formulation. 
\item  In general, a variance-based regularization scheme leads to a non-convex formulation and, therefore, is intractable. We derive an exact mixed-integer second-order cone programming (MISOCP) reformulation  for the case when the loss function $\ell(\bm x, \bm\xi)$ is piecewise linear in~$\bm x$ for all $\bm \xi \in \Xi$, which can be solved using off-the-shelf optimization solvers. Furthermore, we show that the problem reduces to an efficiently solvable second-order cone program (SOCP) if $\ell(\bm x, \bm\xi)$ is linear in $\bm x$ for all $\bm \xi \in \Xi$ and the solution set $\mathcal{X}$ is second-order conic representable. 
\item Adapting ideas from \citet{duchi2019variance} proposed in the context of empirical risk minimization problems, we develop a DRO formulation for the case when the loss function is a general convex function of $\bm x$ for all $\bm \xi \in \Xi$. Furthermore, we establish the equivalence of our variance regularized formulation and the DRO formulation for large sample sizes. For a convex loss function that is quadratic or piecewise linear in $\bm x$, the DRO formulation reduces to a SOCP if $\mc X$ is second-order conic representable.
\end{enumerate}

%\subsection{Paper Organization}
% The remainder of the paper is structured as follows. In Section~\ref{sec:generalization_bound}, we derive generalization bounds for the NW approximation using moderate deviations theory and present the PCA-based dimensionality reduction scheme. Section~\ref{sec:regularization_scheme} proposes the variance regularized scheme that is motivated by the generalization bound and derives the suboptimality bound of the regularized approximation. We present an application in portfolio management in Section~\ref{sec:application_portfolio}, and provide concluding remarks in Section~\ref{sec:conclusion}. For clarity of exposition, lengthy and technical proofs are deferred to the appendix.
The remainder of the paper is organized as follows. In Section~\ref{sec:background}, we provide a background on the Nadaraya-Watson kernel regression estimator as well as introduce the large and moderate deviations theory on which our main results are based. In Section~\ref{sec:generalization_bound}, we derive the generalization bound for the NW approximation using results from moderate deviations theory and present the PCA-based dimensionality reduction scheme. Section~\ref{sec:regularization_scheme} develops a regularization scheme that is motivated by the generalization bound and derives the suboptimality bound for the proposed method. The section also develops an exact reformulation for the regularized problem based on piecewise linear convex loss functions and presents an application from portfolio management. In Section~\ref{sec:DRO_formulation}, we propose a distributionally robust optimization formulation for general convex loss functions. In Section~\ref{sec:numerical_exp}, we provide computational results for inventory management and wind energy commitment problems. Finally, we provide concluding remarks in Section~\ref{sec:conclusion}. For clarity of exposition, lengthy and technical proofs are deferred to the appendix.

\paragraph{Notation and terminology}
We use bold letters for vectors, while scalars are printed in regular font. We denote by $\mathbf e$  the vector of all ones. Random variables are designated by tilde signs (e.g., $\xit$), while their realizations are represented by the same symbols without tildes (e.g., $\bm\xi$). For any $n\in\mathbb N$, we define $[n]$ as the  index set $\{1,\ldots,n\}$. For any matrix $\bmath A$, the operator norm $\lVert \bmath A \rVert_2$ represents the largest singular value of $\bmath A$ and its Frobenius norm is defined as $\lVert \bmath A\rVert_{\fr}=\big(\sum_{ij}A_{ij}^2\big)^{1/2}$. We define by ${\SOC}(n+1) \subseteq \RR^{n+1}$ the standard second-order cone:
$\bm{v} \in {\SOC}(n+1) \Longleftrightarrow \|(v_1, \ldots, v_n)^\top\| \leq v_{n+1}$. The probability simplex in $\mathbb{R}^n_+$ is denoted as $\Delta^n = \left\{ \bm{w} \in \mathbb{R}^n_+ : \mathbf{e}^\top\bm{w}=1 \right\}$ and the Dirac distribution which assigns unit mass on $\bm \xi$ is denoted by $\delta_{\bm \xi}$. For any $x \in \RR$, we define $(x)_+=\max(x,0)$.

For  asymptotic analysis, we use standard notations like $o$ and $O$ to represent rates of convergence. 
%We use $\theta(1)$ as a shorthand for the class of functions that have the asymptotic growth rate of $1+o(1)$ {\bf\color{purple}(are we still using that?)}. Formally, we say that the function $f:\mathbb N\rightarrow \RR$ satisfies the inclusion $f(n)\in\theta(1)$ if for every $\tau>0$ there exists $n_\tau\in\mathbb N$ such that for all $n \geq n_\tau$ we have $1-\tau<f(n)<1+\tau$, i.e.,  $\lim_{n\rightarrow\infty}f(n)=1$.  
We use $\tilde{O}$ notation to denote the $O$ notation that suppresses multiplicative terms with logarithmic dependence on $n$.

\section{Background}\label{sec:background}

In this section, we provide the preliminaries of Nadaraya-Watson (NW) approximation and large and moderate deviations theory that are necessary for the development of the main results in this paper. 

\subsection{Nadaraya-Watson Kernel Regression}
To approximate \eqref{eq:conditional_expectation_problem}, we apply the NW kernel regression which estimates the conditional expectation with
\begin{equation}
\tag{$\mathcal {NW}_{\textup{est}}$}
\label{eq:nadaraya-watson}
\displaystyle\hat\EE_{\bm\gamma}[\ell(\bm x,\xit)]=\frac{\sum_{i=1}^n\mathcal K\left(\frac{\bm\gamma-\bm\gamma^i}{h}\right)\ell(\bm x,\bm\xi^i)}{\sum_{i=1}^n\mathcal K\left(\frac{\bm\gamma-\bm\gamma^i}{h}\right)},
\end{equation}
where $\mathcal K$ is a prescribed kernel function and $h > 0$ is the bandwidth parameter. In this paper, we consider the exponential kernel function given by \citep{Genton2001kernel}
\begin{equation}
\label{eq:Gaussian_kernel}
\mathcal K(\bm \theta)=\frac{1}{Z} \exp\left(-\|\bm \theta\|_2\right),
\end{equation}
with $Z = \int_{\mathbb{R}^p} \exp\left(-\|\bm \theta\|_2\right)\mathrm d\bm\theta$  a normalization constant. 

The estimator \eqref{eq:nadaraya-watson} encapsulates a popular model in data-driven analytics. Indeed, an extremely large value of the bandwidth parameter $h$ means that the approximation \eqref{eq:nadaraya-watson} reduces to the unconditional \emph{sample-average approximation} $\frac{1}{n}\sum_{i=1}^n\ell(\bm x,\bm\xi^i)$. 
On the other hand, a very small bandwidth implies that most of the probability mass is assigned to the sample point closest to $\bm\gamma$.  The choice $h=O(1/n^{1/(p+4)})$ provides the best balance between bias and variance that yields the minimum  expected error  \citep{gyorfi2006distribution}. 

Using the estimator~\eqref{eq:nadaraya-watson}, we arrive at the following approximation to the stochastic optimization problem~\eqref{eq:conditional_expectation_problem}:
\begin{equation}
\tag{$\mathcal {NW}$}
\label{eq:nadaraya-watson_problem}
\min_{\bm x\in\mathcal X}\;\EErhohat[\ell(\bm x,\xit)]. 
\end{equation}
This approximation is first developed by~\cite{Hannah:10}.

\subsection{Large and Moderate Deviations Theory}
Large deviations theory studies the tail behavior of sequences of random variables.  
It characterizes the exponential decay rate of the probability that a random variable in the sequence realizes on  any particular rare event. Formally, we say that the sequence of random variables $\{\tilde z_n\}_{n\in\mathbb N}$ satisfies a large deviation principle with speed $\nu_n$ and rate function $I:\RR\rightarrow[0,+\infty]$ if
\begin{equation}
\label{eq:LDP}
\begin{array}{l}
\displaystyle	\liminf_{n\rightarrow \infty} \frac{1}{\nu_n}\log \mathbb P\left(\tilde z_n\in\mathcal O\right)\geq-\inf _{y\in\mathcal O} I(y)\quad\textup{and}\quad
\displaystyle	\limsup_{n\rightarrow \infty} \frac{1}{\nu_n}\log \mathbb P\left(\tilde z_n\in\mathcal C\right)\leq-\inf _{y\in\mathcal C} I(y),
	\end{array}
	\end{equation}
for every open subset $\mathcal O$ and closed subset $\mathcal C$ of $\mathbb R$, respectively. If the random variable is defined as the  average $\tilde z_n=\frac{1}{n}\sum_{i=1}^n\tilde r_i$ of i.i.d.~random variables  $\tilde r_i$, $i\in\mathbb N$, with a finite logarithmic moment generating function $\Lambda(t)=\EE[\exp(t\tilde r_1)]<+\infty$, then we obtain  the Cramer's theorem which states that the  sequence $\{\tilde z_n\}_{n\in\mathbb N}$ obeys a large deviation principle with speed $\nu_n=n$ and rate $I(y)=\sup_{t\geq 0}(t y-\Lambda(t))$. The inequalities in \eqref{eq:LDP} thus imply that for large enough $n$ the probability that $\tilde z_n$ takes value within the rare event set $\{z:z\geq y\}$, with $y>\EE[\tilde r_1]$, is roughly equal to 
$\exp(-n I(y))$. That is, it decays exponentially fast in $n$ at the rate $I(y)$.
Note that the rate function depends on the particular distribution of the random variable $\tilde r_1$. From the central limit theorem, however, we know that the distribution of the renormalized average $\sqrt{n}\tilde z_n=\frac{1}{\sqrt{n}}\sum_{i=1}^n\tilde r_i$ is asymptotically  normal, which admits a succinct description through the first and second-order moments of  $\tilde r_1$.

Moderate deviations theory delineates the intermediate cases between the two extremes of large deviations theory and central limit theorem. The theory  studies situations where  the sequence $\{a_n\tilde z_n\}_{n\in\mathbb N}$ obeys a large deviation principle with the \emph{same} rate function for a certain range of renormalization parameters $a_n\rightarrow\infty$. The theory often provides a result that combines both large deviations theory and central limit theorem. Analogous to the central limit behavior, the rate function in a moderate deviation principle is typically \emph{analytical}, requiring only limited information about the distribution, such as the variance. However, we also observe an exponential decay rate characteristic of results in large deviations theory. In the case of  i.i.d.~random variables, we find that if $a_n^2/n\rightarrow 0$ as $n\rightarrow \infty$ then the sequence $\{a_n\tilde z_n\}_{n\in\mathbb N}$ obeys a large deviation principle with speed $n/a_n^2$ and analytical rate function $I(y)=\frac{1}{2}y^2/\sigma^2$, where $\sigma^2$ is the variance of the random variable $\tilde r_1$ \cite[Theorem 3.7.1]{dembo38large}. We refer the  reader to the references \citep{dembo38large,eichelsbacher2003moderate} for  a more detailed account on large and moderate deviations theory.
\section{Generalization Bounds via Moderate Deviation Principles}\label{sec:generalization_bound}

In this section, we first derive generalization bounds on the approximation \eqref{eq:nadaraya-watson_problem} for a fixed decision~$\bm x$. The result leverages the following moderate deviations theory of the NW estimator by~\cite{mokkadem:08}  in the setting of exponential kernel functions. To apply this theorem, in this paper we assume the following mild regularity conditions:
\begin{enumerate}[label=(\textbf{A\arabic*})]
\item The support $\Xi$ of the random vector $\xit$ is compact and the loss function $\ell(\bm x,\bm\xi)$ %is measurable and  
takes value in the interval $[0,1]$ for all $\bm x\in\mathcal X$ and $\bm\xi\in\Xi$. \label{as1} 
\item The density function $f(\bm\gamma,\bm\xi)$ is twice differentiable with continuous and bounded partial derivatives. In addition, the marginal density ${f}_{\tilde{\bm\gm}}(\bm{\gamma})$ is non-zero at the given side information vector $\bm{\gamma}$. \label{as2} 
\item The bandwidth parameter $h$ for the kernel function $\mathcal K_h$ is scaled such that $\lim_{n\rightarrow \infty} h_n=0$ and $\lim_{n\rightarrow \infty}n h^p_n=\infty$. 
\label{as3} 
\end{enumerate}
The assumptions about the support set  and the loss function in \ref{as1} are typical in the literature. Here, we do not impose any restriction on the size and structure of the support set  other than its compactness. %The assumption about the loss function is also non-restrictive. 
If the loss function is bounded, then one can simply apply scaling and translation so that it takes value in the interval $[0,1]$. The assumptions about the density function in \ref{as2} are standard  regularity conditions  in kernel density and kernel  regression estimations. They ensure that the conditional distribution of $\xit$  given the side information $\bm\gamma$ can be inferred reasonably well using the historical observations. The assumption about the bandwidth parameter $h$ in \ref{as3} ensures that the estimator \eqref{eq:nadaraya-watson} is asymptotically consistent \citep{gyorfi2006distribution,silverman:86}. 

\begin{theorem}\label{thm:MDP}[Moderate Deviation Principles]
Let the density function $f:\RR^p\times\RR^q\rightarrow\RR$  satisfy assumption \ref{as2}. Consider a function $L:\RR^q\rightarrow\RR$ that satisfies the following conditions:
\begin{enumerate}
\item The function $\bm t\rightarrow\int_{\RR^q}L(\bm\xi)^2f(\bm t,\bm\xi)\mathrm d\bm\xi$ is continuous at $\bm t=\bm\gamma$.  
\item For every $u\in\RR$, the function  $\bm t\rightarrow\int_{\RR^q}\exp(uL(\bm\xi))f(\bm t,\bm\xi)\mathrm d\bm\xi$ is bounded and continuous at $\bm t=\bm\gamma$. 
\item The function $\bm{t}\rightarrow\int_{\RR^q}L(\bm\xi)f(\bm t,\bm\xi)\mathrm d\bm\xi$ is twice  differentiable on $\mb R^p$, with continuous and bounded partial derivatives at  $\bm t=\bm\gamma$. 
\end{enumerate}
	
Then, for any positive sequence $\{a_n\}_{n\in\mathbb N}$ such that
\begin{equation*}
\lim_{n\rightarrow \infty} a_n=\infty,\quad \lim_{n\rightarrow \infty} \frac{a_n^2}{n h_n^p}=0,\quad\textup{and}\quad\lim_{n\rightarrow \infty}a_n h_n^2=0,
\end{equation*}
the sequence $\{a_n(\EErho[L(\xit)]-\EErhohat[L(\xit)])\}_{n\in\mathbb N}$ satisfies a large deviation principle with speed $\nu_n=n h_n^p/a_n^2$ and rate function 	
\begin{equation}\label{eq:rate_function}
I_{\bm\gamma}(y)=\frac{y^2 g(\bm\gamma)}{\VVrho[L(\xit)]}
\end{equation}
where $g(\bm\gamma) = f_{\tilde{\bm\gamma}}(\bm\gamma)/\left( 2\int_{\RR^p}\mathcal K^2(\bm \theta)\mathrm d\bm \theta \right)$ is the scaled marginal density of $\gammat$ and $\VVrho[L(\xit)]=\mathbb V[L(\xit)|\gammat=\bm\gamma]$ is the conditional variance of $L(\xit)$ given the side information $\bm\gamma$. 
	That is, we have 
		\begin{equation}
		\label{eq:MDP}
		\begin{array}{l}
\displaystyle	\liminf_{n\rightarrow \infty} \frac{1}{\nu_n}\log \mathbb P\left(a_n\left(\EErho[L(\xit)]-\EErhohat[L(\xit)]\right)\in\mathcal O\right)\geq-\inf _{y\in\mathcal O} I_{\bm\gamma}(y)\quad\textup{and}\\[3mm]
\displaystyle	\limsup_{n\rightarrow \infty} \frac{1}{\nu_n}\log \mathbb P\left(a_n\left(\EErho[L(\xit)]-\EErhohat[L(\xit)]\right)\in\mathcal C\right)\leq-\inf _{y\in\mathcal C} I_{\bm\gamma}(y),
	\end{array}
	\end{equation}
	for every open subset $\mathcal O$ and closed subset $\mathcal C$ of $\mathbb R$, respectively.
\end{theorem}

\subsection{Generalization Bounds}
Using \Cref{thm:MDP}, we arrive at our first main result whose proof can be found in Appendix~\ref{app:generalization_bound}.

\begin{proposition}
\label{thm:generalization_bound}
For any fixed $\bm x\in\mathcal X$, we have
\begin{equation}\label{eq:generalization_1}
\mathbb P\left(\left|\EErho[\ell(\bm x,\xit)]-\EErhohat[\ell(\bm x,\xit)]\right|\geq\epsilon\right)= \exp\left(-nh_n^p\frac{ \epsilon^2 g(\bm\gamma)(1+o(1))}{\VVrho[\ell(\bm x,\xit)]}\right).
\end{equation}
\end{proposition}

\noindent \Cref{thm:generalization_bound} asserts that, as the sample size grows, the probability that the NW approximation deviates by at least $\epsilon$ from the true conditional expectation decays exponentially fast in $nh_n^p$. Setting the right-hand side of \eqref{eq:generalization_1} to $\delta$, we arrive at the following guarantee on the out-of-sample errors. 
\begin{corollary}[Generalization Bound]
\label{cor:generalization_bound}
For any fixed $\bm x\in\mathcal X$, we have 
\begin{equation}
\label{eq:error_bound}
\left|\EErho[\ell(\bm x,\xit)]-\EErhohat[\ell(\bm x,\xit)]\right|\leq \sqrt{\frac{\VVrho[\ell(\bm x,\xit)]}{ n h_n^pg(\bm\gamma)(1+o(1))}\log\left(\frac{1}{\delta}\right)} = O\left( \sqrt{\frac{1}{nh_n^p}} \right),
\end{equation}
with probability at least $1-\delta$.
\end{corollary}

\begin{remark}
%It is easy to show that a similar generalization bound can be obtained from the application of  by assuming a Gaussian kernel instead of an exponential kernel. 
With minor modifications, it is possible to derive a similar generalization bound when the popular Gaussian kernel is used in~\ref{eq:nadaraya-watson} instead of the exponential kernel.
\end{remark}

\noindent The bound in \eqref{eq:error_bound} degrades if the scaled density $g(\bm\gamma)$ is small or if the conditional variance $\VVrho[\ell(\bm x,\xit)]$ is large. In the limit where $g(\bm\gamma)\downarrow 0$, there are fewer historical samples close to the given side information, implying that the NW estimator constitutes a poor  approximation of the true conditional expectation. On the other hand, a smaller variance indicates that  few data points are sufficient to accurately describe the conditional distribution of $\xit$ given $\bm\gamma$. 

Using the best bandwidth parameter scaling $h_n= O(1/n^{1/(p+4)})$ for the multivariate NW estimator \cite[Chapter 5.2]{gyorfi2006distribution}, we find that the error bound in \eqref{eq:error_bound} diminishes at the rate of $ O(1/n^{2/(p+4)})$. Note that we have a dependence on the dimension~$p$, which suggests that the estimator suffers from the \emph{curse of dimensionality}. %{\color{red} It may appear here that the dependence on $d$ can be avoided by using the bandwidth scaling $h_n= O(1/n^{1/(cp)})$,  where $c$ is a constant greater than~$1$. In this case, the error decreases at the rate $O(1/n^{(c-1)/(2c)})$ independently of the dimension $p$. However,  as the bandwidth is suboptimal the magnitude of the term $o(1)$ inflates, implying that the approximation becomes poor when the  sample size is small. The result in Theorem \ref{thm:generalization_bound}  indicates that as the  sample size grows one may employ  the bandwidth scaling  $h_n= O(1/n^{1/(cp)})$ to eliminate the dependence on the dimension $p$. [should we delete that?]}
In general, such a result is quite typical for settings where kernels are used; it has also been observed in other works such as \cite{kannan2020data}. In Section~\ref{sec:high_dim}, we propose a dimensionality reduction scheme based on principal component analysis that allows us to obtain tighter bounds when the intrinsic dimensionality of $\bm \gamma$ is considerably smaller than the dimensionality $p$ of the ambient space.

So far, we have obtained the generalization bound for a fixed $\bm x \in \mc X$. In the following theorem, we extend the result in \Cref{cor:generalization_bound} to obtain uniform generalization bounds for all $\bm x \in \mc X$, under the assumption that the feasible set $\mc X$ consists of finitely many points. 
\begin{theorem}[Generalization Bound for a Finite Set $\mc X$]
\label{thm:generalization_bound_finiteX}
Suppose that $\mathcal X$ is a finite set. Then, we have 
\begin{equation*}
\label{eq:generalization_bound_finiteX}
\mb E_{\bm\gm}[\ell(\bm x,\tilde{\bxi})]\leq  \hat{\mb E}_{\bm\gm}[\ell(\bm x,\tilde{\bxi})]+ \sqrt{\frac{\mb V_{\bm\gm}[\ell(\bm x,\tilde{\bxi})]}{ n h_n^pg(\bm\gamma)(1+o(1))}\log\left(\frac{\lvert \mathcal X \rvert}{\delta}\right)}\qquad\forall \bm x\in\mathcal X
\end{equation*}
with probability at least $1-\delta$.
\end{theorem}
\begin{proof}
The proof follows from a straightforward application of the union bound to the result obtained in \Cref{cor:generalization_bound}.
\end{proof}

\noindent Note that the bound  \eqref{eq:generalization_bound_finiteX} grows only logarithmically in the cardinality of the feasible set $\mathcal X$ and, hence, at most linearly in the dimension of the decision vector~$\bm x$. %{\color{red} Improved bounds with similar guarantees can be derived for other classes of loss function  and feasible set by taking into account the class' Rademacher complexity and VC dimension \citep{bartlett2002rademacher,vapnik1998statistical,shalev2014understanding}.} {\color{blue}[Prateek, are you sure?]}

In our analysis for Theorem \ref{thm:generalization_bound_finiteX}, we assumed that the feasible set $\mathcal{X}$ is finite. In what follows, we show that under additional mild assumptions on the loss function, the result can be extended to the setting where $\mathcal{X}$ is a continuous and bounded set. 

\begin{theorem}[Generalization Bound for a Continuous and Bounded Set $\mc X$]
\label{thm:generalization_bound_continuousX}
Suppose $\mathcal{X}\subset \mathbb{R}^d$ is a bounded subset with finite diameter $D = \sup_{\bm{x},\bm{x}' \in \mathcal{X}} \Vert \bm{x} - \bm{x}' \Vert$. Assume that the loss function $\ell(\bm x,\tilde{\bxi})$  is M-Lipschitz continuous in $\bm{x}$, i.e., there exists a constant $M>0$ such that 
\begin{equation}\label{eq:obj_fun_lips_assum} 
\left| \ell(\bm{x},\bm{\xi}) - \ell(\bm{x}',\bm{\xi}) \right| \leq M \Vert \bm{x} - \bm{x}' \Vert \quad \forall \bm{x},\bm{x}' \in \mathcal{X}\;\forall \bm{\xi}\in \Xi.
\end{equation}
Fix a tolerance level $\eta>0$. Then, with probability at least $1-\delta$, we have 
\begin{equation*}
\mb E_{\bm\gm}[\ell(\bm x,\tilde{\bxi})]\leq \hat{\mb E}_{\bm\gm}[\ell(\bm x,\tilde{\bxi})]+\sqrt{\frac{\mb V_{\bm\gm}[\ell(\bm x,\tilde{\bxi})]}{ n h_n^pg(\bm\gamma)(1+o(1))}\log\left(\frac{\lvert \mathcal X_\eta \rvert}{\delta}\right)}  + M\eta \left(1 + \sqrt{\frac{\log\left(\frac{\lvert \mathcal X_\eta \rvert}{\delta}\right)}{ n h_n^pg(\bs\gamma)(1+o(1))}} \right) \;\forall\bm x\in\mathcal X,
\end{equation*}
where $\lvert \mathcal X_\eta \rvert=O(1) (D/\eta)^d$.
\end{theorem}
\noindent We defer the proof of the above theorem to Appendix~\ref{sec:generalization_bound_continuousX}. 

\begin{remark}
An alternative way to construct an empirical estimator for conditional expectation is by using the k-nearest neighbors regression (KNN), which assigns equal weight $1/k$ to the $k$ nearest points of $\bm\gamma$. \cite{bertsimas2019predictions} derive a generalization bound  of  the scheme. They prove that under more restrictive assumptions, such as $\gammat$ is supported on a subset $\Gamma$ of $[0,1]^p$ and there exists a constant $g>0$ such that $\PP(\|\gammat-\bm\gamma\|\leq\epsilon)>g\epsilon^p$ for all $\bm\gamma\in\Gamma$,  the generalization bound of the scheme decays at the rate of $\tilde O(1/n^{1/(2p)})$. However, unlike our bounds in Theorems \ref{thm:generalization_bound_finiteX} and \ref{thm:generalization_bound_continuousX},  their bound is independent of the variance (or risk) of the decisions. Therefore, designing an effective regularization scheme for the KNN-based approach is not immediately obvious. 

%does not provide any useful insights for designing an effective regularization scheme. 
\end{remark}

\subsection{Extension to High-Dimensional $\bm \gamma$ }\label{sec:high_dim}

\noindent In this section, we extend our analysis to the setting where the side information $\bm \gamma \in \mb R^p$ is high-dimensional, i.e., where $p$ is large. From the result obtained in~\Cref{cor:generalization_bound}, we observe that the generalization bound decays at the rate $O\big(n^{-\frac{2}{(p+4)}}\big)$, which is slow for decision-making problems with large $p$. 

In real-world settings, however, data often lies on a low-dimensional subspace or manifold. In other words, the intrinsic dimensionality of the data is much smaller than the dimensionality of the ambient space. To take this into consideration, we consider the setting where the side information vector $\bm \gamma$ is drawn from a sub-gaussian\footnote{We refer the reader to \cite{vershynin2010introduction,wainwright2019high} for more details about sub-gaussian random vectors.} probability distribution with sub-gaussian parameter $\sigma$ and lies approximately in a low-dimensional linear subspace $\mc S$ where $\text{dim}(\mc S)= p'\ll p$. Here, we make the assumption that $ \gamma^{\mc S}$---the component of $\bm \gamma$ that lies in the subspace $\mc S$---corresponds to the signal and influences the random cost parameter vector $ \xit$, while its orthogonal component $\bm \gamma^{\mc S^\perp}$ corresponds to the noise term, which, given~$\bm \gamma^{\mc S}$, does not provide any information about~$\xit$. In other words, the random vector $\xit$ is conditionally independent of $\bm \gamma^{\mc S^\perp}$ given $\bm \gamma^{\mc S}$, i.e., $(\bm \xi \indep \bm \gamma^{\mc S^\perp}) \lvert \bm \gamma^{\mc S}$. Thus, the conditional distribution satisfies
\begin{equation*}
f(\bm \xi\lvert \bm \gamma) := f(\bm \xi\lvert \bm \gamma^{\mc S}, \bm \gamma^{\mc S^\perp}) = f(\bm \xi\lvert \bm \gamma^{\mc S}).
\end{equation*}
%$\bm \gamma=\bm \rho + \bm \epsilon$.
The sub-gaussian  assumption on $\bm \gamma$ is also non-restrictive and encompasses a wide class of probability distributions, including all multivariate Gaussian distributions and distributions with bounded support. We mention here that a setup similar to ours has been considered in \cite{xu2016statistical} for robust optimization problems in high-dimensions. 
%Here, $\bm \gamma^{\mc S}$ and $\bm \gamma^{\mc S^\perp}$ denote the projections of $\gamma$ on to the subspaces $\mc S$ and $\mc S^\perp$ respectively. We interpret $\bm\gamma^{\mc S}$ to be the component of the side information vector $\bm \gamma$ that influences the random cost parameter $\bm \xi$ and~$\bm\gamma^{\mc S^\perp}$ as the random noise component, which given~$\bm \gamma^{\mc S}$, does not provide any information about~$\bm \xit$.

Under the conditional independence assumption, the optimization problem \eqref{eq:conditional_expectation_problem} is equivalent to 
\begin{equation}
\label{eq:conditional_expectation_problem_reduced}
\tag{$\mc{SO}_{\textup{reduced}}$}
\min_{\bm x\in\mathcal X}\;\left\{  \EE_{\bm \rho}[\ell(\bm x,\xit)]:=\EE[\ell(\bm x,\xit)\,|\,\tilde{\bm \rho}=\bm\rho]\right\},
\end{equation}
where $\tilde{\bm \rho}=\proj_{\mc S}(\tilde{\bm\gamma})$ is the projection of the random vector~$\bm \gamma$ onto the subspace $\mc S$. As discussed in Section~\ref{sec:intro}, the exact conditional distribution $f(\bm \xi\lvert \bm \rho)$ is usually not known. If the exact subspace~$\mc S$ is known, the historical data $\{(\bm\rho^1,\bm \xi^1),\ldots,(\bm\rho^n,\bm \xi^n)\}$ can be obtained by projecting the realizations~$\bm \gamma^i$ onto the subspace~$\mc S$. Similar to the~\eqref{eq:nadaraya-watson_problem} formulation developed before for the stochastic optimization problem~\eqref{eq:conditional_expectation_problem}, we propose to approximate the problem~\eqref{eq:conditional_expectation_problem_reduced} using the Nadaraya-Watson kernel regression estimator, as follows:
\begin{equation}
\label{eq:NW_reduced}
\min_{\bm x\in\mathcal X}\;\left\{ \hat{\EE}_{\bm \rho}[\ell(\bm x,\xit)]:=\frac{\sum_{i=1}^n\mathcal K\left(\frac{\bm\rho-\bm\rho^i}{h}\right)\ell(\bm x,\bm\xi^i)}{\sum_{i=1}^n\mathcal K\left(\frac{\bm\rho-\bm\rho^i}{h}\right)}\right\}.
\end{equation}

In general, however, the exact subspace $\mc S$ may also be unknown. Therefore, we develop a dimensionality reduction procedure based on principal component analysis (PCA) that allows us to construct an estimate $\hat{\mc S}$ of the true subspace $\mc S$ in a data-driven manner. 
%by computing the principal components from the sample covariance matrix. 
Our approach is based on the idea of sample splitting, which has been previously proposed in the literature to obtain tighter bounds for high-dimensional problems in other contexts \citep{chaudhuri2009multi,srivastava2019robust,yan2020covariate}. The main idea is to randomly split the observations in the data matrix~$\bm \Gamma=(\bm \gamma^1, \ldots,\bm \gamma^n )^\top$ into two disjoint parts, $\bm \Gamma_1$ and $\bm \Gamma_2$, with the observations in the corresponding parts indexed by sets $\mc I_1$ and $\mc I_2$ with cardinalities $\lvert \mc I_1 \rvert=n_1$ and $\lvert \mc I_2 \rvert=n_2$, respectively. Using the observations in $\bm \Gamma_2$, we construct the sample covariance matrix $ \hat{\bm \Sigma}_2=\tfrac{1}{n_2}\sum_{i \in \mc I_2} (\bm\gamma^i - \bar{{\bm \gamma}}) (\bm\gamma^i- \bar{{\bm \gamma}})^\top $ where $\bar{{\bm \gamma}}=\tfrac{1}{n_2} \sum_{i \in \mc I_2} \bm\gamma^i$ and compute its top~$p'$ eigenvectors $\hat{\bmath U}=[\hat{\bm u}_1,\ldots,\hat{\bm u}_{p'}]^\top \in \mb R^{p'\times p}$, which form a basis  for the estimated subspace~$\hat{\mc S}:=\vspan(\hat{\bmath U})$. Once $\hat{\mc S}$ is determined, the observations in $\bm \Gamma_1$ are projected on to the subspace  to obtain their projections $\hat{\bm \Pi}_1=\bmath \Gamma_1 \hat{\bmath U}^\top$. Sample splitting ensures that the projected points are independent of each other, which is required for the application of moderate deviations theory to obtain the theoretical guarantees. In practice, however, this step can be usually skipped and the subspace $\hat{\mc S}$ can be estimated from the entire data matrix. Next, we let 
\begin{equation}\label{eq:NW_reduced_main}
\tag{$\mc{NW}^{\textup{red}}_{\textup{est}}$}
\hat{\EE}_{\hat{\bm \rho}}[\ell(\bm x,\xit)]:=\frac{\displaystyle\sum_{i\in \mc I_1}\mathcal K\left(\frac{\hat{\bm\rho}-\hat{\bm\rho}^i}{h}\right)\ell(\bm x,\bm\xi^i)}{\displaystyle \sum_{i\in \mc I_1}\mathcal K\left(\frac{\hat{\bm\rho}-\hat{\bm\rho}^i}{h}\right)},
\end{equation}
to denote the NW estimator defined in \eqref{eq:NW_reduced} based on the dimensionality reduction procedure detailed above. We delineate the generalization bound for the reduced estimator in the following proposition whose proof is deferred to  Appendix~\ref{sec:proof_generalizationbound-highdim}.  
% For this section, we assume that the side information $\bm \gamma\in \mb R^D$ is drawn from a sub-gaussian probability distribution with sub-gaussian parameter $\sigma^2$. Furthermore, we assume that $\bm\gamma$ lies in a low-dimensional subspace with intrinsic dimensionality $d \ll D$, i.e., there exist $\bmath W\in \mb R^{d\times D}$ and $v\in \mb R^d$ such that $\rho = \bmath W\bm\gamma+v$.

% \begin{proposition}[Generalization Bound for \ref{eq:NW_reduced_main}]\label{prop:generalizationbound-highdim}
% Suppose $n_1$ and $n_2$ are sufficiently large and $n_2^{-1/2} / h_{n_1} < 1$. Then for all $\bm x \in \mc X$, we have
% \begin{multline*}
% \lvert \EE_{\bm \gamma}[\ell(\bm x,\xit)]-\hat{\EE}_{\hat{\bm \rho}}[\ell(\bm x,\xit)] \leq \\ \sqrt{ \frac{ \VVrho[\ell(\bs x,\xit)] }{ n_1 h_{n_1}^{p'}g(\bs\gamma)(1+o(1))}\log \left(\frac{\lvert \mc X \rvert}{\delta} \right)}  + \frac{8}{h} \frac{4\sigma C}{\lambda_{p'} - \lambda_{p'+1}} \sqrt{\frac{p'}{n_2} \log \left(\frac{2\lvert \mc X \rvert}{\delta}\right)} \left(\sqrt{p}+\sqrt{\frac{1}{2}\log \left(\frac{\lvert \mc X \rvert}{\delta}\right)}\right)
% \end{multline*}
% with probability at least $1-5 n_1\delta$. 
% \end{proposition}
\begin{proposition}[Generalization Bound for \ref{eq:NW_reduced_main} with Finite Set $\mc X$]\label{prop:generalizationbound-highdim}
Suppose $\mc X$ is a finite set, $n_1$ and $n_2$ are sufficiently large and $n_2^{-1/2} / h_{n_1} < 1$. Then, we have
\begin{multline*}
\lvert \EE_{\bm \gamma}[\ell(\bm x,\xit)]-\hat{\EE}_{\hat{\bm \rho}}[\ell(\bm x,\xit)] \leq \sqrt{ \frac{ \VVrho[\ell(\bs x,\xit)] }{ n_1 h_{n_1}^{p'}g(\bs\rho)(1+o(1))}\log \left(\frac{5 n_1\lvert \mc X \rvert}{\delta} \right)} \\  \qquad \quad \qquad\qquad\qquad + \frac{8}{h} \frac{4\sigma C}{\lambda_{p'} - \lambda_{p'+1}} \sqrt{\frac{p'}{n_2} \log \left(\frac{10 n_1\lvert \mc X \rvert}{\delta}\right)} \left(\sqrt{p}+\sqrt{\frac{1}{2}\log \left(\frac{5 n_1\lvert \mc X \rvert}{\delta}\right)}\right)\quad\forall \bm x\in\mc X
\end{multline*}
with probability at least $1-\delta$. Here, $C>0$ is a constant that depends on the sub-gaussian parameter $\sigma$, $\lambda_{p'}$ is the $p'$-th largest eigenvalue of the true covariance matrix $\bs \Sigma$ of ${\gammat}$.
\end{proposition}
%\noindent We defer the proof of the proposition to. 
%\begin{remark}
\noindent From the proposition, we see that if we choose $h_{n_1}=O(1/n_1^{1/(p'+4)})$ and the sizes of $\bs \Gamma_1$ and $\bs \Gamma_2$ such that $n_1=\alpha n$ and $n_2=(1-\alpha)n$ for some $0<\alpha<1$, then the  requirement $n_2^{-1/2} / h_{n_1} < 1$ holds for sufficiently large $n$, and the generalization bound decays at the rate ${O}\big(n^{-\frac{2}{(p'+4)}}\log(n_1)\big)$. % as $n$ goes to $\infty$. %{\color{red}[Here we have the problem that the probability depends on $n_1$ and so it is no good when $n_1$ is large. Shall we use change of variable so that it beceomes independent? but that would affect the rate...]}
%\end{remark}
Thus, by adopting the proposed dimensionality reduction procedure, the generalization bound no longer depends on the original dimension $p$ of the ambient space. Instead, it is a function of only the intrinsic dimensionality $p'$ of the side information vector $\bs \gamma$. Hence, the adverse impact on the generalization bound associated with the curse of dimensionality is mitigated.    

When $\gammat$ is bounded, i.e., $\Vert \gammat \Vert \leq \gamma_{\max}$ almost surely, we obtain a sharper bound without the $\log(n_1)$ factor. In this case, the error decays at a faster rate ${O}\big(n^{-\frac{2}{(p'+4)}}\big)$. % as $n$ goes to $\infty$.
\begin{corollary}[Generalization Bound for \ref{eq:NW_reduced_main} with bounded covariates]\label{prop:generalizationbound-highdim2}
Consider the same setting as in \Cref{prop:generalizationbound-highdim} and assume that $\gamma$ is a bounded random variable where $\Vert \gammat \Vert \leq \gamma_{\max}$ almost surely. Then, we have
%Suppose  $n_1$ and $n_2$ are sufficiently large, and $n_2^{-1/2} / h_{n_1} < 1$. 
\begin{multline*}
\lvert \EE_{\bs\gamma}[\ell(\bs x,\xit)]-\hat{\EE}_{\hat{\bs\rho}}[\ell(\bs x,\xit)] \rvert\leq  \sqrt{ \frac{ \VVrho[\ell(\bs x,\xit)] }{ n_1 h_{n_1}^{p'}g(\bs\rho)(1+o(1))}\log \left(\frac{2\lvert \mc X \rvert}{\delta} \right)}\\
\qquad+ \frac{8}{h} \frac{C}{\lambda_{p'} - \lambda_{p'+1}} \sqrt{\frac{p'}{n_2} \log \left(\frac{4\lvert \mc X \rvert}{\delta}\right)} \gamma_{\max}\quad\forall \bm x\in\mc X
\end{multline*}
with probability at least $1-\delta$. Here, $C>0$ is a constant that depends on the sub-gaussian parameter $\sigma$, $\lambda_{p'}$ is the $p'$-th largest eigenvalue of the true covariance matrix $\bs \Sigma$ of ${\gammat}$.
\end{corollary}

\section{A Conditional Standard Deviation Regularization Scheme}
\label{sec:regularization_scheme}
The generalization bounds obtained in Theorems~\ref{thm:generalization_bound_finiteX} and \ref{thm:generalization_bound_continuousX} imply that the out-of-sample errors are negligible if the conditional standard deviation  $\sqrt{\VVrho[\ell(\bm x,\xit)]}$ is small. This suggests that a regularization scheme involving the term would ensure a solution with a strong generalization bound.
However, as we do not have access to the {true} conditional variance, we propose to utilize the \emph{empirical} conditional variance as a surrogate
\begin{equation}
\label{eq:empirical_variance}
\VVrhohat[\ell(\bm x,\xit)]:=\EErhohat[(\ell(\bm x,\xit)-\EErhohat[\ell(\bm x,\xit)])^2]=\EErhohat[\ell(\bm x,\xit)^2]-\EErhohat[\ell(\bm x,\xit)]^2. 
\end{equation}
This setting gives rise to the regularized NW approximation
\begin{equation}\label{eq:regularized_problem}
\tag{$\mathcal R$$\mathcal{NW}$}
\min_{\bm x\in\mathcal X}\;\EErhohat[\ell(\bm x,\xit)]+\lambda\sqrt{\VVrhohat[\ell(\bm x,\xit)]},
\end{equation}
where $\lambda\geq 0$ is a tuning parameter that controls the degree of regularization. We  point out here that a similar formulation on the variance-based regularization scheme has been previously proposed and analyzed in the empirical risk minimization literature~\citep{maurer2009empirical,duchi2019variance} for the unconditional setting, where the true (unconditional) probability distribution is approximated by the empirical distribution. % that assigns an equal weight of $1/n$ to each observation.
\subsection{Suboptimality Bounds}
In this section, we aim to establish the properties of the optimal solutions to problem~\eqref{eq:regularized_problem}. We first show that replacing the true conditional variance with its empirical estimate \eqref{eq:empirical_variance} does not significantly weaken the generalization bound derived in Section \ref{sec:generalization_bound}. 
\begin{proposition}\label{prop:bound_sample_stddev}
Fix a tolerance level $\tau>0$. For any fixed $\bm x \in \mc X$, we have
\begin{equation}\label{eq:bound_sample_stddev}
\displaystyle\left|\sqrt{\VVrho[\ell(\bm x,\xit)]}-\sqrt{\VVrhohat[\ell(\bm x,\xit)]}\right|\leq \tau 
+\sqrt{ \frac{\log\left(\frac{1+2/\tau}{\delta}\right)}{ n h_n^pg(\bm\gamma)(1+o(1))}},
\end{equation}
with probability at least $1-\delta$.
\end{proposition}
	
\noindent The proof of Proposition~\ref{prop:bound_sample_stddev} can be found in Appendix~\ref{app:bound_sample_stddev}. We remark that the tolerance level $\tau$ can be made small without significantly increasing the square root term on the right-hand side of~\eqref{eq:bound_sample_stddev} as the latter displays merely a logarithmic dependence in~$\tau$.

We next analyze the suboptimality bound resulting from solving the regularized  problem \eqref{eq:regularized_problem}. We first assume that the feasible set $\mathcal X$ is finite even though its cardinality can be exponential in the problem dimensions. Let $\hat{\bm x}$ be a minimizer of the regularized problem and $\bm x^\star$ be a minimizer of the true stochastic optimization problem \eqref{eq:conditional_expectation_problem}. 
\begin{theorem}[Suboptimality Bound for a Finite Set $\mc X$]\label{thm:optimality_bound}
Fix a tolerance level $\tau>0$. Then, for some scaling of the regularization parameter  $\lambda= O\left( 1/ \sqrt{ n h_n^pg(\bm\gamma) }\right)$, we have
\begin{equation}\label{eq:suboptimality_bound}
\displaystyle\EErho[\ell(\hat{\bm  x},\xit)]
\displaystyle\leq\displaystyle\EErho[\ell({\bm  x}^\star,\xit)]+\left(\sqrt{\VVrho[\ell({\bm  x}^\star,\xit)]}+\tau\right) \sqrt{\frac{4\log\left(\frac{6|\mathcal X|}{\delta}\right)}{ n h_n^pg(\bm\gamma) (1+o(1))}}+{ \frac{2\log\left(\frac{6|\mathcal X|(1+2/\tau)}{\delta}\right)}{ n h_n^pg(\bm\gamma)(1+o(1))}},
\end{equation}
with probability at least $1-\delta$.
\end{theorem}

\noindent The proof of the theorem is deferred to Appendix~\ref{sec:thm_optimality_bound}. Theorem~\ref{thm:optimality_bound} asserts that if there is an optimal solution~$\bs x^\star$ of the stochastic  problem~\eqref{eq:conditional_expectation_problem} that yields a cost with negligible  conditional variance, then the regularized solution $\hat{\bs  x}$ will converge to this optimal solution at a rate of $O({1}/{({n h_n^p})})$. 

In our analysis for Theorem \ref{thm:optimality_bound}, we assumed that the feasible set $\mathcal{X}$ is finite. In the next theorem, under the assumption of a Lipschitz loss function, we extend the result to obtain a similar suboptimality bound for the case where the solution set $\mathcal{X}$ is continuous and bounded.

\begin{theorem}[Suboptimality Bound for a Continuous and Bounded Set $\mc X$]\label{cor:continuousX}
Suppose $\mathcal{X}$ is a bounded subset of $\mathbb{R}^d$ with finite diameter $D = \sup_{\bm{x},\bm{x}' \in \mathcal{X}} \Vert \bm{x} - \bm{x}' \Vert$ and the cost function $\ell$ is Lipschitz continuous in $\bm{x}$, i.e., it satisfies condition~\eqref{eq:obj_fun_lips_assum}. Then,  for some scaling of the regularization parameter  $\lambda= O\left( 1/ \sqrt{ n h_n^pg(\bm\gamma) }\right)$ and any $\tau, \eta > 0$, we have 
\begin{equation*}
\begin{aligned}
\displaystyle\EErho[\ell(\hat{\bm  x},\xit)]
\displaystyle&\leq\displaystyle\EErho[\ell(\x^\star,\xit)]+(2+\lambda)M\eta+\left(\sqrt{\VVrho[\ell(\x^\star,\xit)]}+\tau\right) \sqrt{\frac{4\log\left(\frac{ O(1) (D/\eta)^d}{\delta}\right)}{ n h_n^pg(\bm\gamma) (1+o(1))}}\\
&\ \ \ +{ \frac{2\log\left(\frac{(1+2/\tau)O(1) (D/\eta)^d}{\delta}\right)}{ n h_n^pg(\bm\gamma)(1+o(1))}},
\end{aligned}
\end{equation*}
with probability at least $1-\delta$.
\end{theorem}
\noindent The proof of the theorem is deferred to  Appendix~\ref{sec:thm_optimality_bound_continuosX}.

\subsection{A Mixed-Integer Second-Order Cone Programming Formulation}\label{sec:exact_formulation}

In general, the exact problem \eqref{eq:regularized_problem} is intractable because of the non-convexity of the regularization term in the objective function. In this section, we consider the case where the loss function is piecewise linear convex and $\mathcal{X}$ is second-order conic representable, and we derive a mixed-integer SOCP formulation for \eqref{eq:regularized_problem}. Although the problem remains hard to solve, reasonably large instances of the problem can be solved using off-the-shelf solvers such as Gurobi and CPLEX. Based on our derivation, we also show that, particularly for the case where the loss function is linear, the problem is efficiently solvable as a SOCP.   

\begin{proposition}
Suppose the loss function $\ell(\bm x,\xit) = \max_{j\in[m]} \bs{a}_j(\bs{x})^\top \xit + b_j$ is piecewise linear convex in $\x$ and the feasible set $\mc X$ is second-order conic representable. Let $\overline{w}_{i}={\mathcal K(\frac{\bm\gamma-\bm\gamma^i}{h})}/({\sum_{i=1}^n\mathcal K(\frac{\bm\gamma-\bm\gamma^i}{h})})$ denote the kernel weight associated with the $i$-th observation, then the problem \eqref{eq:regularized_problem} is solvable as the following mixed-integer second-order cone program:
\begin{equation}\label{eq:MISOCP}
\begin{array}{cll}
\textup{min}  &\displaystyle \overline{\bs{w}}^\top \bs{\nu} + \lambda\rho\\
\textup{s.t.} &\displaystyle \left(\sqrt{\overline{w}_{1}}(\nu_1-t),\ldots,\sqrt{\overline{w}_{n}}(\nu_n -t),\rho\right)\in\SOC(n+1), \\
&\displaystyle     \bs{a}_j(\bs{x})^\top \bs{\xi}^i  + b_j   \leq \nu_i  & \forall i\in[n]\; \forall j\in[m], \\
&\displaystyle   \bs{a}_j(\bs{x})^\top \bs{\xi}^i  + b_j + M (1-z_{ij})  \geq \nu_i  & \forall i\in[n]\; \forall  j\in[m], \\&\displaystyle   \sum_{j\in [m]} z_{ij} = 1  & \forall i\in[n],\\
&\displaystyle \bs x\in\mathcal X, \; \bs \nu\in\RR^n,\;  \rho\in\RR,\; t\in\RR, \; \bm  z\in \{0,1\}^{n\times m}. 
\end{array}
\end{equation}
where $M > 0$ is a sufficiently large constant. Under the assumption that $\ell(\bm{x},\bm{\xi})$ takes value in the interval $[0,1]$, it is sufficient to set $M=1$.
\end{proposition}

\begin{proof}
To obtain the formulation, we first introduce the epigraphical variable $\rho$ to \eqref{eq:regularized_problem} to bring the conditional standard deviation term into the constraint: 
\begin{equation*}
\begin{array}{cl}
\textup{min} &\displaystyle \EErhohat[\ell(\bs x,\xit)] + \lambda\rho\\
\textup{s.t.} & \sqrt{\VVrhohat[\ell(\bs x,\xit)]}\leq \rho, \\
 &\displaystyle \bs x\in\mathcal X, \; \rho\in\RR.
\end{array}
\end{equation*}
Then, we have that the above formulation is equivalent to   
\begin{equation*}
\begin{array}{cl}
\textup{min} &\displaystyle \left(\sum_{i=1}^n \overline{w}_{i} \cdot \ell(\bs x,\bs\xi^i) \right) + \lambda\rho \\
\textup{s.t.} &\displaystyle \sqrt{ \sum_{i=1}^n \overline{w}_{i} \cdot \left( \ell(\bs x,\bs\xi^i) -t\right)^2  } \leq \rho, \\
&\displaystyle \bs x\in\mathcal X, \; \rho\in\RR,\; t\in\RR,
\end{array}
\end{equation*}
where (as in the proof of Proposition \ref{prop:bound_sample_stddev}) we make use of the fact that for any random variable $\tilde\chi$, $\VVrhohat[\tilde\chi]=\arg \min_{t\in \mb R} \EErhohat[(\tilde\chi-t)^2]$. Next, we introduce the auxiliary variables $\nu_i= \ell(\bs x,\bs\xi^i)= \max_{j\in[m]} \bs{a}_j(\bs{x})^\top \bs{\xi}^i  + b_j $ for all $i\in [n]$. Using the Big-M notation, we can linearize the resulting non-convex constraints to obtain the final formulation 
\begin{equation*}
\begin{array}{cll}
\textup{min} &\displaystyle \overline{\bs{w}}^\top  \bs{\nu} + \lambda\rho\\
\textup{s.t.} &\displaystyle \sqrt{ \sum_{i=1}^n \overline{w}_{i} \cdot \left( \nu_i -t\right)^2  } \leq \rho, \\
&\displaystyle     \bs{a}_j(\bs{x})^\top \bs{\xi}^i  + b_j   \leq \nu_i  & \forall i\in[n]\;  j\in[m], \\
&\displaystyle   \bs{a}_j(\bs{x})^\top \bs{\xi}^i  + b_j + M (1-z_{ij})  \geq \nu_i  & \forall i\in[n]\;  j\in[m], \\&\displaystyle   \sum_{j\in [m]} z_{ij} = 1  & \forall i\in[n],\\
&\displaystyle \bs x\in\mathcal X, \; \bs \nu\in\RR^n,\;  \rho\in\RR,\; t\in\RR, \; \bm  z\in \{0,1\}^{n\times m}.
\end{array}
\end{equation*}
This completes the proof. 
\end{proof}

Due to the binary decision variables $\bm  z\in \{0,1\}^{n\times m}$, the above formulation is a mixed-integer second-order cone program (MISOCP), provided that $\X$ is second-order conic representable with binary/integer variables. If the loss function $\ell(\bm x,\bm\xi)$ is simply a linear function of $\x$, i.e.,  $m=1$, then the formulation reduces to a second-order cone program (SOCP), which is efficiently solvable in polynomial time using interior-point methods. We state this result formally in the following corollary. 
\begin{corollary}\label{coro:SOCP}
Suppose the loss function $\ell(\bm x,\bm\xi)$ is a linear function of $\x$ and the feasible set $\mc X$ is second-order conic representable, then the problem \eqref{eq:regularized_problem} can  equivalently be reformulated as the second-order cone program
\begin{equation*}
\begin{array}{cll}
\textup{min} &\displaystyle \overline{\bm w}^\top \bs{\nu} + \lambda\rho\\
\textup{s.t.} &\displaystyle \left(\sqrt{\overline{w}_{1}}(\nu_1-t),\ldots,\sqrt{\overline{w}_{n}}(\nu_n -t),\rho\right)\in\SOC(n+1) , \\
&\displaystyle  \bs{a}(\bs{x})^\top \bs\xi^i + b =\nu_i  & \forall i\in[n], \\
&\displaystyle \bs x\in\mathcal X, \; \bs \nu\in\RR^n,\; \rho\in\RR,\; t\in\RR. \;
\end{array}
\end{equation*}
\end{corollary}
Next, based on our discussion above, we obtain the SOCP formulation for the generic portfolio optimization problem with side information and present the results of a small example.

\subsection{A Portfolio Optimization Example}\label{sec:application_portfolio}

In this section, we investigate the performance of our proposed regularized NW approximation on the portfolio optimization problem described in \Cref{ex:finance}. We compare the performances of the LDR approach and our regularization scheme. As a direct application of \Cref{coro:SOCP}, our regularization scheme can be reformulated as a SOCP. For both the proposed regularization scheme and the LDR approach, the details of the formulations are provided in \Cref{app:portOptExample}.

\begin{example}\label{ex:portOpt_compare}[A Three-Asset Portfolio]
Consider the portfolio optimization problem in \Cref{ex:finance}. We compare our regularized NW approximation from \Cref{coro:SOCP} with the state-of-the-art linear decision rule (LDR) formulation for the problem proposed by \cite{brandt2009parametric}~and~\cite{baziergeneralization}. We first empirically test the proposed regularized NW approximation and the LDR formulation, and see how they perform against these optimal returns. 

\begin{figure}[t]
  \centering
\begin{subfigure}[t]{.45\textwidth}
\centering
      \includegraphics[width=1.0\textwidth]{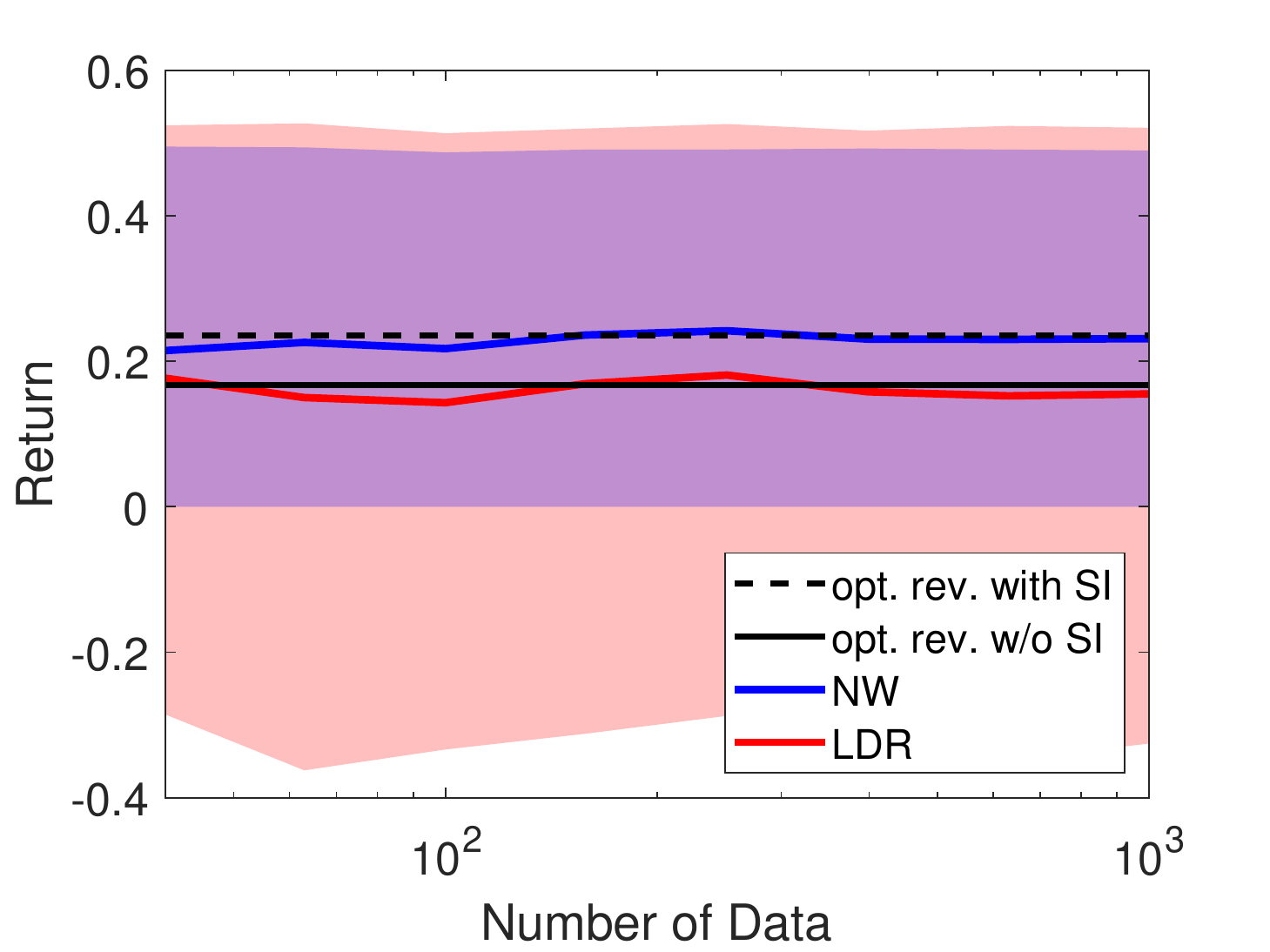}
  \caption{$\gamma$'s are sampled from uniform distribution within $[-1,1]$}
  \label{fig:experiment1}
\end{subfigure}
\begin{subfigure}[t]{.45\textwidth}
\centering
      \includegraphics[width=1.0\textwidth]{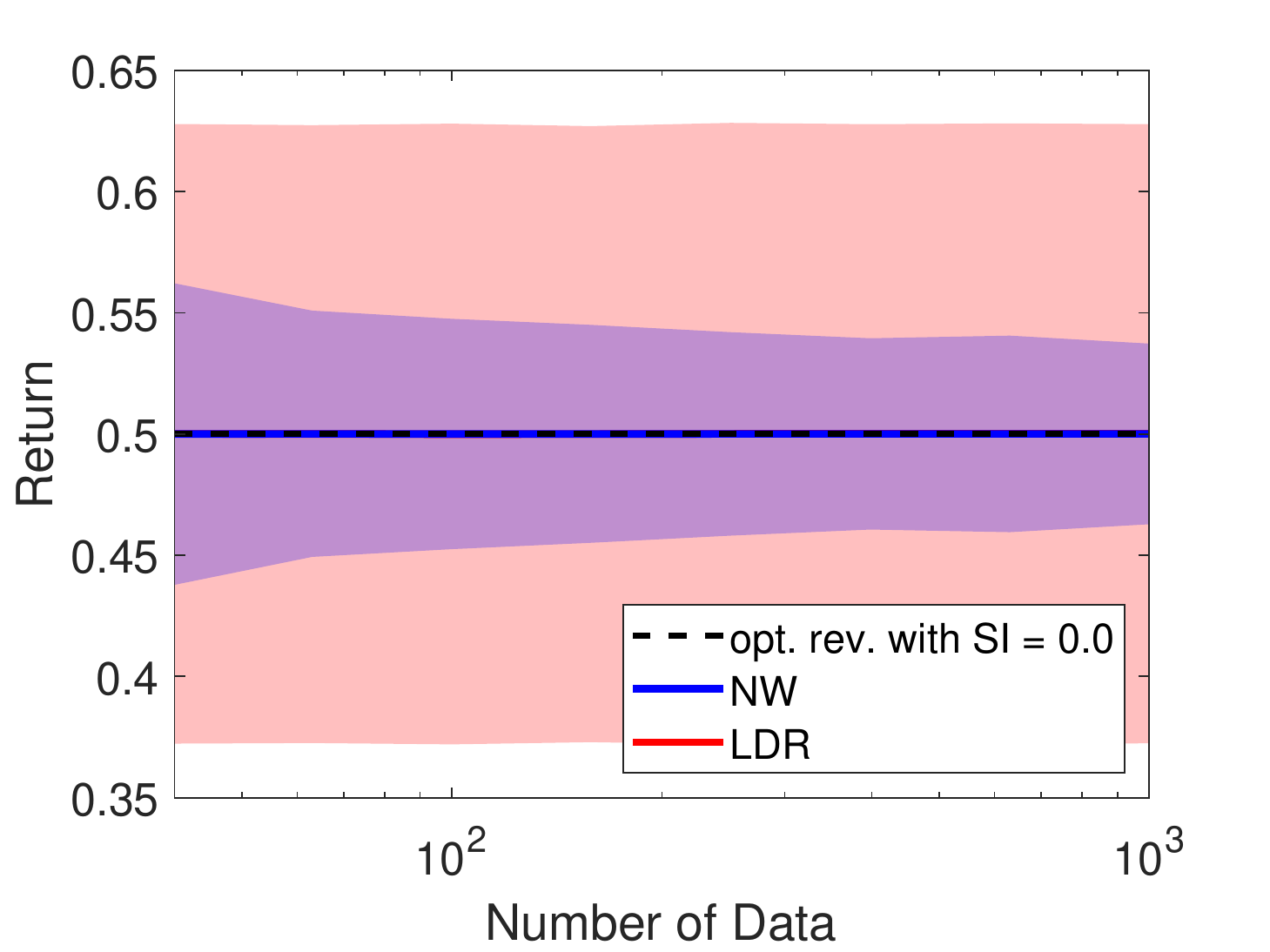}
  \caption{$\gamma = 0.0$}
  \label{fig:experiment2}
\end{subfigure}
\caption{Out-of-sample portfolio returns of different approaches over $300 \; \gamma$'s for each $n$. The black dot line is the optimal expected return with the side information $\gamma$ given. The black solid line is the optimal expected return without considering side information. The blue and red soild lines are the average returns of our proposed model and the LDR formulation, respectively. The shaded region for each color records the returns between the $10$th and $90$th percentile of the returns of the corresponding approach.}
\label{fig:experiment}
\end{figure}

Figure \ref{fig:experiment1} shows the out-of-sample returns of the two approaches, as well as the  optimal expected portfolio returns with and without consideration of the side information, respectively. We find that our proposed approach substantially outperforms LDR in terms of both return and risk. Even though the two approaches attempt to exploit the side information when generating their portfolios, the NW approach is more effective in  capitalizing  the information as it consistently generates higher expected returns. %, indicating that it can capitalize on the side information more effectively. 
We also observe that the NW returns have significantly lower variability. This is not entirely surprising because the regularization term encourages a portfolio with lower standard deviation. %{\color{magenta} Since Asset 1 and Asset 2 have perfect negative correlation, the NW approach tends to allocate an equal weight to both assets so that the individual noise terms $\tilde\epsilon_1$ and  $\tilde\epsilon_2$ are neutralized in the resulting portfolio; see  Figure \ref{fig:experiment2} }{\color{blue} 
Figure \ref{fig:experiment2} depicts the out-of-sample returns for a fixed covariate $\gamma = 0$. In this case, the conditional expected return of each risky asset is $0.5$ and investing in any convex combination of the two risky assets yields the optimal expected portfolio return. Since Asset 1 and Asset 2 have perfect negative correlation, the NW approach tends to allocate an equal weight to both assets so that the individual noise terms $\tilde\epsilon_1$ and $\tilde\epsilon_2$ are neutralized in the resulting portfolio. %To minimize the variation, however, investing equally in both risky assets would be the best option, which is approximately implemented by our proposed regularized NW approach.

As expected, the returns of the NW approximation converge fast to the best expected portfolio return as the data size grows. On the other hand, we observe that LDR disappointingly performs as if it were oblivious to the side information, even with large data size. This phenomenon can be explained analytically as follows. For any fixed parameters $x_1$, $x_2$, and $y$, the expected portfolio return is given by
\begin{equation*}
\begin{array}{rl}
\displaystyle \EE\left[\sum_{i=1}^2\tilde\xi_i(\tilde\gamma)(x_i + \tilde\gamma \cdot y) \right]    &=\displaystyle \;\; \EE\left[ \left(\frac{1}{2} - \tilde\gamma^2 + 0.1\cdot \tilde\epsilon_1 \right) (x_1 + \tilde\gamma \cdot y) + \left(\frac{1}{2} - \tilde\gamma^2 + 0.1\cdot \tilde\epsilon_2 \right) (x_2 + \tilde\gamma \cdot y) \right]  \\
%&=\displaystyle \;\;  \EE\left[ \left(\frac{1}{2} - \tilde\gamma^2 \right) (x_1 + \tilde\gamma \cdot y) + \left(\frac{1}{2} - \tilde\gamma^2  \right) (x_2 + \tilde\gamma \cdot y) \right]  \\
&=\displaystyle \;\; \EE\left[ \left(\frac{1}{2} - \tilde\gamma^2 \right) (x_1 + x_2) + 2 \left(\frac{1}{2}\tilde\gamma \cdot y - \tilde\gamma^3 \cdot y  \right)  \right]  \\
%&=\displaystyle \;\; \left(\frac{1}{2} - \EE[\gamma^2] \right) (x_1 + x_2) + 2 \left(\frac{1}{2}\EE[\tilde\gamma] \cdot y - \EE[\tilde\gamma^3] \cdot y  \right)  \\
&=\displaystyle \;\; \left(\frac{1}{2} - \frac{1}{3} \right) (x_1 + x_2) + 0 \;\; = \;\;  \dfrac{x_1 + x_2}{6},
\end{array}
\end{equation*}
where the second equality holds because the random variables $\tilde\epsilon_1$ and $\tilde\epsilon_2$ are independent of $\tilde\gamma$ and have mean zero, while the penultimate equality follows from the identities $\EE[\tilde\gamma^2]=1/3$ and $\EE[\tilde\gamma]=\EE[\tilde\gamma^3]=0$. Since the constraint $x_1+x_2\leq 1$ is imposed in the formulation, the LDR approach will never generate an expected portfolio return greater than $1/6$. This result affirms our observation that LDR indeed performs as poorly as the model that disregards the side information. 
\end{example}

From the above example, we demonstrate that the LDR approach could fail miserably at exploiting the side information, even on a simple setting. On the other hand, the proposed regularized NW approximation is highly effective at leveraging the side information and can generate a remarkably higher average return with minimal risks.

\section{Connections with Distributionally Robust Optimization}\label{sec:DRO_formulation}

In this section, we consider the setting where the loss function $\ell(\x,\bxi)$ is a general (not necessarily a  piecewise linear) convex function of $\x$ for all $\bm \xi \in \Xi$. Leveraging ideas from~\cite{duchi2019variance}, we obtain a distributionally robust optimization (DRO) formulation, which is a tractable approximation for our proposed variance regularization scheme. 
%I we will show later, when {\bf blahblahblah}, the DRO model is equivalent to the proposed regularization scheme.
In the following proposition, we derive the DRO formulation and show that for large $n$, the  DRO formulation is equivalent to the proposed variance regularized formulation.

\begin{proposition}\label{prop:DRO}
%Assume $\lvert \ell(\bm x,\xit)\rvert$ is bounded. Then, 
Let $\overline{w}_{i} = \mathcal K_{h}(\bm\gamma-\bm\gamma^i)/\sum_{j=1}^n\mathcal K_{h}(\bm\gamma-\bm\gamma^j)$, $i\in[n]$, denote the empirical weights obtained from NW regression, and
$\hat{\mb P}_{\bm \gamma}=\sum_{i=1}^n \overline{w}_{i} \delta_{\bm \xi^i}$ be the empirical conditional distribution.
For any $\bm{x}\in\mathcal{X}$, we have
\begin{multline*}
\EErhohat[\ell(\bm x,\xit)]+ \left( \lambda\sqrt{\VVrhohat[\ell(\bm x,\xit)]} - \lambda^2 \right)_+ \leq 
\max_{\mb P\in\mathcal P_\lambda(\hat{\mb P}_{\bm \gamma})}  \mb E_{\mb P} [\ell(\bm x,\xit)] \leq \EErhohat[\ell(\bm x,\xit)] + \lambda\sqrt{\VVrhohat[\ell(\bm x,\xit)]} ,
\end{multline*}
where 
%\begin{equation}\label{eq:modified_chi}
%  \mathcal P_\lambda(\hat{\mb P}_{\bm \gamma})= \left\{\mb P :\begin{array}{l}
%  \bm w \in \Delta^n,\\
%        \sum_{i=1}^n \frac{(w_{i}-\overline{w}_i)^2}{\overline{w}_i} \leq \frac{\lambda^2}{2},  \\
%        \mb P = \sum_{i=1}^n w_i \delta_{\bm \xi^i}
%   \end{array}  \right\}
%\end{equation}
\begin{equation}\label{eq:modified_chi}
\mathcal P_\lambda(\hat{\mb P}_{\bm \gamma})= \left\{\mb P = \sum_{i=1}^n w_i \delta_{\bm \xi^i} ~:~ \sum_{i=1}^n \dfrac{(w_{i}-\overline{w}_i)^2}{\overline{w}_i} \leq \frac{\lambda^2}{2}, \; \bm w \in \Delta^n 
\right\}
\end{equation}
is a modified $\chi^2$ ambiguity set constructed around the empirical conditional distribution.
In particular, if $\hat{\mathbb{V}}_\gamma [\ell(\bm x,\xit) ] \geq \lambda^2$,  then
\begin{equation*}
\max_{\mb P\in\mathcal P_\lambda(\hat{\mb P}_{\bm \gamma})} \mb E_{\mb P} [\ell(\bm x,\xit)] = \EErhohat[\ell(\bm x,\xit)] + \lambda\sqrt{\VVrhohat[\ell(\bm x,\xit)]} .
\end{equation*}
%Note that the above is the $\chi^2$-divergence with $\phi(t)= 0.5\cdot (t-1)^2$
%\begin{equation*}
%D_{\phi}(\bm{p} \Vert \bm{q}) = \sum_{i=1}^n q_i \cdot \phi\left(\frac{p_i}{q_i} \right) = \sum_{i=1}^n q_i \cdot \frac{1}{2} \left(\frac{p_i}{q_i} -1 \right)^2 =   \sum_{i=1}^n  \frac{1}{2q_i} \left(p_i- q_i \right)^2.
%\end{equation*}
\end{proposition}
% \begin{proposition}\label{prop:DRO}
% %Assume $\lvert \ell(\bm x,\xit)\rvert$ is bounded. Then, 
% For any $\bm{x}\in\mathcal{X}$, we have
% \begin{multline*}
% \EErhohat[\ell(\bm x,\xit)]+ \left( \lambda\sqrt{\VVrhohat[\ell(\bm x,\xit)]} - \lambda^2 \right)_+ \leq \\ 
% \max_{\bm{w}\in\Delta^n} \left\{\sum_{i=1}^n w_i \cdot \ell(\bm x,\bm{\xi}^i) : \sum_{i=1}^n  \frac{1}{2\overline{w}_i} \left(w_i - \overline{w}_i \right)^2 \leq \frac{\lambda^2}{2} \right\} \leq \EErhohat[\ell(\bm x,\xit)] + \lambda\sqrt{\VVrhohat[\ell(\bm x,\xit)]} ,
% \end{multline*}
% where 
% \begin{equation*}
% \overline{w}_i = \frac{\mathcal K_{h}(\bm\gamma-\bm\gamma^i)}{\sum_{j=1}^n\mathcal K_{h}(\bm\gamma-\bm\gamma^j)}.
% \end{equation*}
% In particular, if $\hat{\mathbb{V}}_\gamma [\ell(\bm x,\xit) ] \geq \lambda^2$,   
% \begin{equation*}
% \max_{\bm{w}\in\Delta^n} \left\{\sum_{i=1}^n w_i \cdot \ell(\bm x,\bm{\xi}^i) : \sum_{i=1}^n  \frac{1}{2\overline{w}_i} \left(w_i - \overline{w}_i \right)^2 \leq \frac{\lambda^2}{2} \right\} = \EErhohat[\ell(\bm x,\xit)] + \lambda\sqrt{\VVrhohat[\ell(\bm x,\xit)]} .
% \end{equation*}
% %Note that the above is the $\chi^2$-divergence with $\phi(t)= 0.5\cdot (t-1)^2$
% %\begin{equation*}
% %D_{\phi}(\bm{p} \Vert \bm{q}) = \sum_{i=1}^n q_i \cdot \phi\left(\frac{p_i}{q_i} \right) = \sum_{i=1}^n q_i \cdot \frac{1}{2} \left(\frac{p_i}{q_i} -1 \right)^2 =   \sum_{i=1}^n  \frac{1}{2q_i} \left(p_i- q_i \right)^2.
% %\end{equation*}
% \end{proposition}

As stated in \Cref{prop:DRO}, if $\hat{\mathbb{V}}_\gamma [\ell(\bm x,\xit) ] \geq \lambda^2$ then the DRO model is equivalent to the proposed regularization scheme. Although $\hat{\mathbb{V}}_\gamma [\ell(\bm x,\xit) ] $ is a random quantity, it should be close to $\mathbb{V}_\gamma [\ell(\bm x,\xit) ]$ with high probability when $n$ is sufficiently large. In addition, \Cref{cor:continuousX} suggests the scaling $\lambda= O\left( 1/ \sqrt{ n h_n^pg(\bm\gamma) }\right)$, which converges to $0$ as $n\rightarrow \infty$. Based on these observations, we derive the condition under which the two models are equivalent with high probability.

\begin{proposition}\label{prop:empirical_var_large}
Suppose 
\begin{equation}\label{eq:DROequivlent_var_cond}
\sqrt{\VVrho[\ell(\bm x,\xit)]} \geq  \frac{C_\lambda}{\sqrt{ n h_n^pg(\bm\gamma) }} + \tau 
+ \sqrt{ \frac{\log\left(\frac{\lvert \mc X_\eta \rvert}{\delta}\right)+\log(1+2/\tau)}{ n h_n^pg(\bm\gamma)(1+o(1))}}+2M\eta \qquad \forall\bm x \in \mc X,
\end{equation}
for some constants $C_\lambda, \tau, \eta \in \mathbb{R}_{++}$, $\delta \in (0,1)$, and $\lvert \mathcal X_\eta \rvert=O(1) (D/\eta)^d$. Then, with probability at least $1-\delta$, 
\begin{equation*}
\max_{\mb P\in\mathcal P_\lambda(\hat{\mb P}_{\bm \gamma})} \mb E_{\mb P} [\ell(\bm x,\xit)]= \EErhohat[\ell(\bm x,\xit)] + \lambda\sqrt{\VVrhohat[\ell(\bm x,\xit)]} ,
\end{equation*}
where  the ambiguity set $\mathcal P_\lambda(\hat{\mb P}_{\bm \gamma})$ is defined in  \eqref{eq:modified_chi} and the regularization parameter $\lambda$ is set to $C_\lambda / \sqrt{ n h_n^pg(\bm\gamma) }$.
\end{proposition}

\Cref{prop:empirical_var_large} provides a technical condition \eqref{eq:DROequivlent_var_cond}  for which, with high probability, the DRO model is equivalent to the proposed regularization scheme, which is in general intractable. We  emphasize that the condition \eqref{eq:DROequivlent_var_cond} should hold for sufficiently large $n$ if $\VVrho[\ell(\bm x,\xit)]  > 0$ for all $\bm x \in \mc X$. In particular, by carefully choosing the bandwidth $h_n$ (accordingly, $\tau$,  $\eta$, $\delta$), the right-hand side converges to $0$ as $n\rightarrow \infty$. For example, suppose the bandwidth $h_n = C_h/n^{1/(p+4)}$ is used with some constant $C_h > 0$. Then, one can show that for sufficiently large $n$, the right-hand side of \eqref{eq:DROequivlent_var_cond} becomes
\begin{equation*}
\sqrt{ \frac{2}{n^{2/(p+4)}g(\bm\gamma)(1+o(1))}} + \frac{2}{\exp\left( C_h^p n^{2/(p+4)} \right) - 1} + \frac{2MD}{\exp\left( \frac{C_h^p n^{2/(p+4)}}{2d}\right)} \rightarrow 0 \quad\text{as}\; n\rightarrow\infty,
\end{equation*}
and the DRO model is equivalent to the proposed regularization scheme with probability at least $1-C_{\mathcal X}\exp\left( - \frac{C_h^p n^{2/(p+4)}}{2}\right)$ for some constant $C_{\mathcal X}$. We provide the details and the associated corollary of \Cref{prop:empirical_var_large} in \Cref{app:empirical_var_large}.

\begin{remark}\label{rem:socp}
Assume that $\mathcal{X}$ is a convex set and $\ell(\bm x,\bxi)$ is convex in $\bm x$ for all $\bxi\in\Xi$. 
Then, the DRO problem
\begin{equation*}\label{eq:dro_formulation}
\tag{$\mathcal{DRO}$}
\min_{\bm x \in \mathcal{X}} \max_{\mb P\in\mathcal P_\lambda(\hat{\mb P}_{\bm \gamma})}\mb E_{\mb P} [\ell(\bm x,\xit)] 
\end{equation*}
can be formulated as the convex optimization problem given by
\begin{equation}\label{eq:dualform-conic}
\begin{aligned}
&\min && \alpha -  \sum_{i=1}^n  \sqrt{\overline{w}_i} \beta_i + \frac{\lambda}{\sqrt{2}} \nu   \\
&\text{s.t.} && \alpha \geq \ell(\bm x, \bm \xi^i) +\frac{\beta_i}{\sqrt{\overline{w}_i}} & \forall i\in [n],  \\
&&& \bm x \in \mc{\X} , \; \alpha\in\mb R, \; (\bm\beta,\nu) \in \SOC(n+1).
\end{aligned}
\end{equation}
    
\noindent Thus, the \ref{eq:dro_formulation} problem is efficiently solvable as a second-order cone program provided that $\mathcal{X}$ is second-order conic representable and $\ell(\bm x,\bxi)$ is either a convex quadratic or a piecewise linear convex function of $\bm x$ for all $\bxi$.
\end{remark}

\section{Numerical Experiments}
\label{sec:numerical_exp}
We evaluate the performance of the distributionally robust model \eqref{eq:dro_formulation} in the context of inventory management and wind energy commitment applications. All the experiments were run on a 2.2~GHz Intel Core i7 CPU laptop with 8 GB RAM and solved using MOSEK 9.2.  
%All optimization problems are implemented in Python 3.7 with package CVXPY 1.1.0 and solved by MOSEK 9.2. The experiments were run on a 2.2GHz Intel Core i7 CPU laptop with 8GB RAM

\subsection{Inventory Management}
We  first consider  the classical newsvendor problem with side information. Faced with an uncertain demand $\tilde{\xi}$, the vendor is interested in determining the order quantity $q$ that minimizes the overall cost. The vendor incurs a cost, which includes two components: holding cost and stock-out cost. Associated with order quantity $q$, the cost function  assumes the following form:
\begin{equation}
\ell(q, \xi) =  h (q-\xi)_+ + b (\xi-q)_+ ,
\end{equation}
where $b$ and $h$ denote respectively the per unit stock-out and holding costs. We assume that the random side information vector $\tilde{\bm \gm} = (\tilde{t},\tilde{p})$ consists of two components: $\tilde{t}\in [0,15]$, which represents the time of the day and $p \in [0,10]$, which is a measure of the popularity of the product at any given time. We assume that the demand varies according to the conditional distribution  %$\tilde{\xi}|\tilde{\bm\gamma} = \bm\gamma \sim \mc 
$\tilde{\xi}\sim U(\psi(\bm\gamma)-10,\psi(\bm\gamma)+10)$, %where the mean $\psi$ of the uniform distribution has the following functional form: 
which is uniform with mean
\begin{equation}
\psi(\bm\gamma)=50+20\cdot\sin\bigg(\frac{t}{\pi/3}\bigg) + 5p .
\end{equation}
In this equation, the first constant term represents a baseline demand for the product at any given time. The second term, which is a sinusoidal function of $t$, aims to capture the fluctuations in demand based on time, while the final term represents a linear relationship in the popularity $p$ of the product and its mean demand. Based on the derivation in \eqref{eq:dualform-conic}, we obtain the following DRO formulation for the newsvendor problem: 
\begin{equation}\label{eq:newsvendor}
\begin{aligned}
&\min && \alpha -  \sum_{i=1}^n  \sqrt{\overline{w}_i} \beta_i + \frac{\lambda}{\sqrt{2}} \nu   \\
&\text{s.t.} && \alpha \geq z_i +\frac{\beta_i}{\sqrt{\overline{w}_i}} & \forall i\in [n],  \\
&&& z_i\geq  h s^+_i + b s^-_i, \\
&&& s^+_i\geq  q-\xi^i  & \forall i\in [n],\\
&&& s^-_i\geq  \xi^i-q  & \forall i\in [n],\\
&&& (\bm\beta,\nu) \in \SOC(n+1), \\
&&& \bm s^+,  \bm s^- \in\RR_+^n, \ q\in \RR_+, \\
&&& \alpha\in\mb R, \; \bm z \in \mb R^n.  
\end{aligned}
\end{equation}
We measure the quality of the optimal solution $q^\star$ obtained by solving the formulation~\eqref{eq:newsvendor} in terms of the out-of-sample loss for  \eqref{eq:conditional_expectation_problem} formulation, which represents the true stochastic optimization problem with side information. Since we do not have access to the true conditional expectation of the loss function, we generate 500 samples of $\tilde{\xi}$ to approximate the conditional loss and solve the sample average approximation problem at each of the side information covariates $\bm \gamma$ of interest. %{\color{blue} (what does this mean? why do you need to solve the SAA problem? I thought we only use the data the estimate the loss with our output decisions?)}

In our problem setup, we set the parameters for the newsvendor problem to $b=10$ and $h=6$. We assume that the side information vector $\tilde{\bm \gamma}$ has a bivariate normal distribution $\mc N (\bm \mu, \bm \Sigma)$, with mean $\bm \mu=[7.5,5]^\top$ and covariance matrix $\bm \Sigma=\Diag([2,1])$. For our experiments, we conduct 10~simulation runs for each side information covariate ${\bm\gamma}$ and sample size $n$  of interest. In each simulation, we generate a training dataset $\{(\bm\gamma^i,\bm\xi^i)\}_{i=1}^n$ consisting of $n$ samples, solve the newsvendor DRO formulation, and evaluate the out-of-sample loss at each ${\bm\gamma}$ of interest.

\begin{figure}[h!]
  \centering
\begin{subfigure}[t]{.49\textwidth}
\centering
      \includegraphics[width=1.0\textwidth]{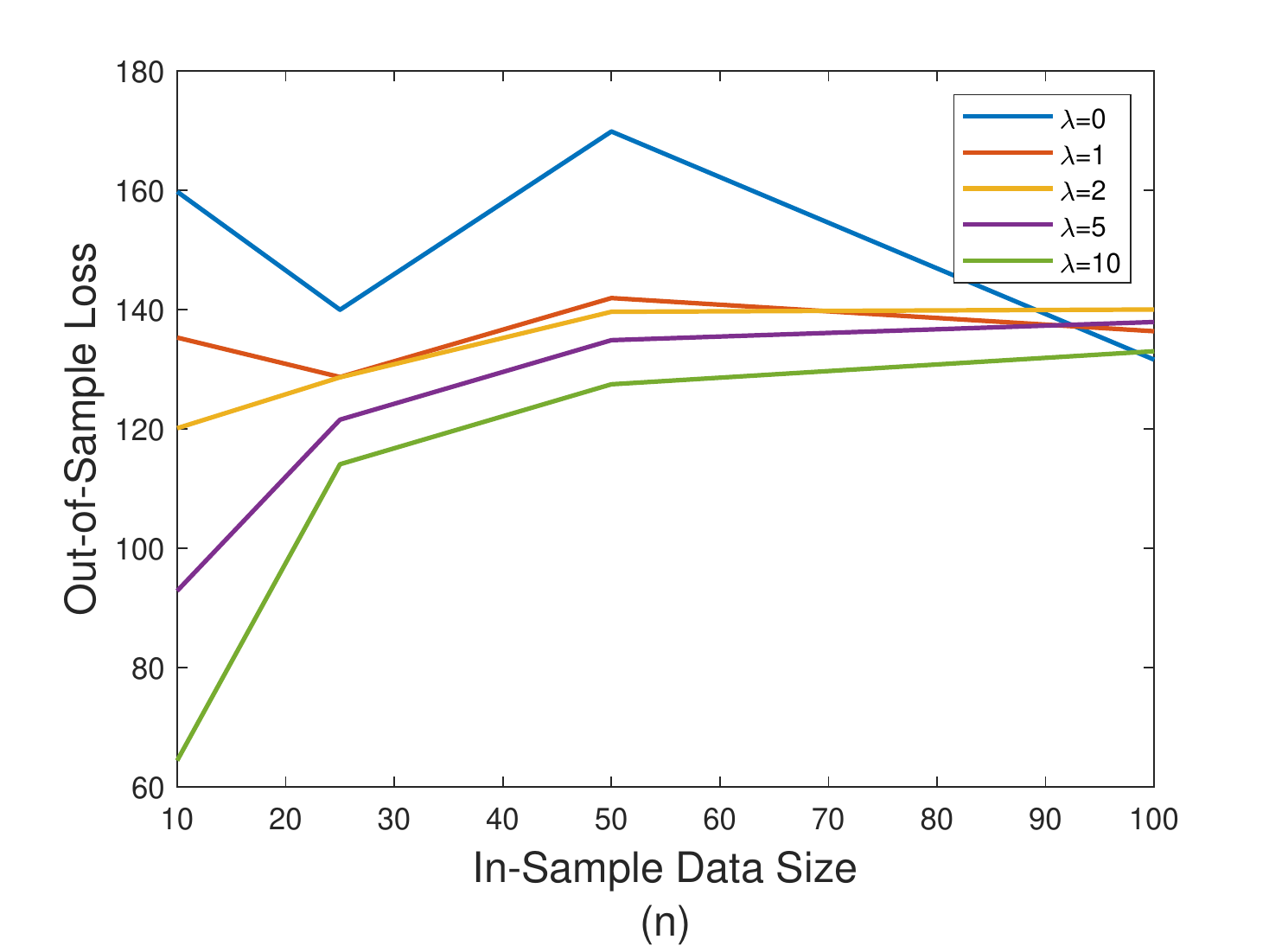}
  \caption{${\bm \gamma}=(1.5,2.5)$}
  \label{fig:newsvendor1}
\end{subfigure}
\begin{subfigure}[t]{.49\textwidth}
\centering
      \includegraphics[width=1.0\textwidth]{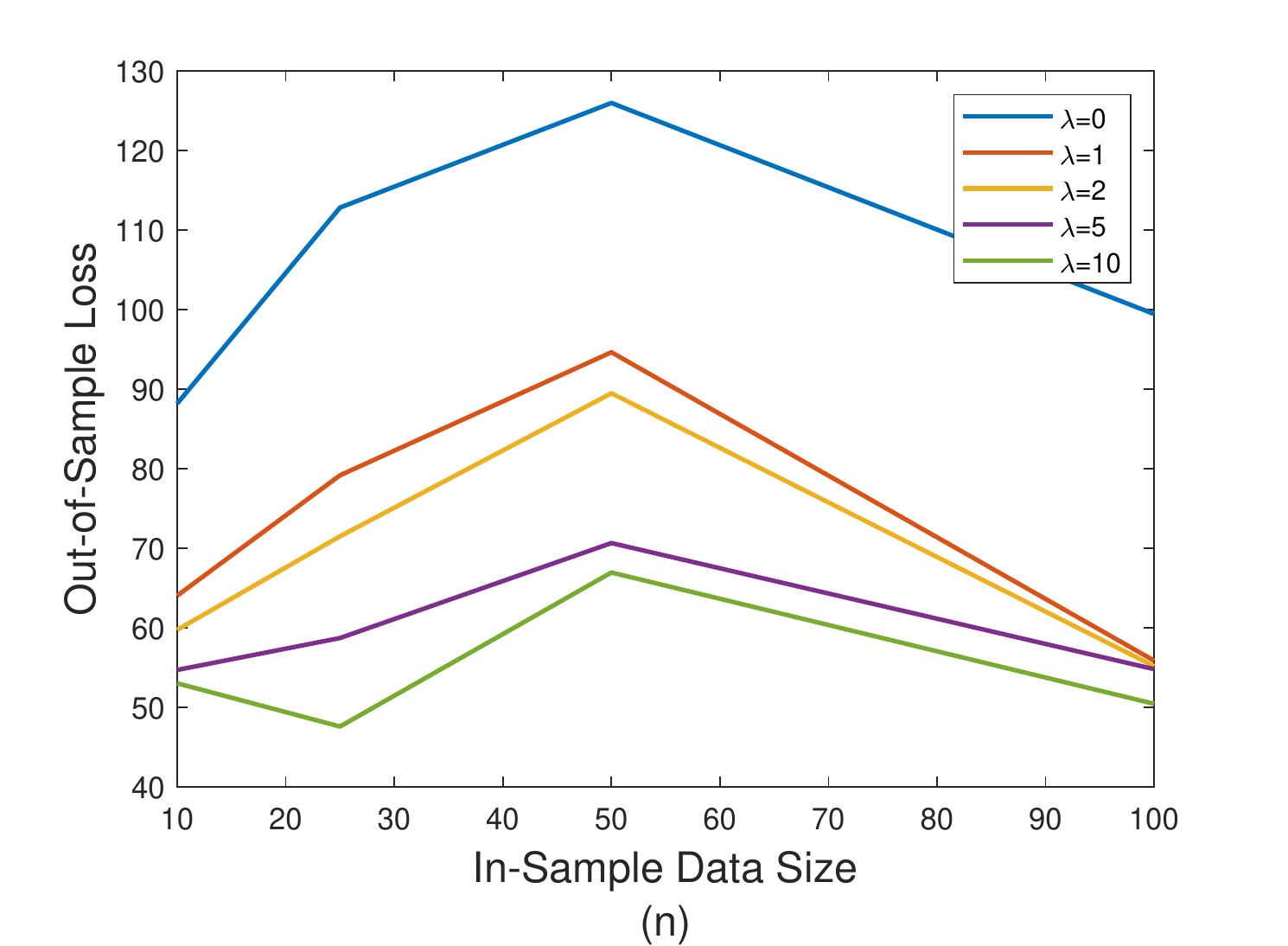}
  \caption{${\bm\gamma}=(3,2.5)$}
  \label{fig:newsvendor2}
\end{subfigure}
\begin{subfigure}[t]{.49\textwidth}
\centering
      \includegraphics[width=1.0\textwidth]{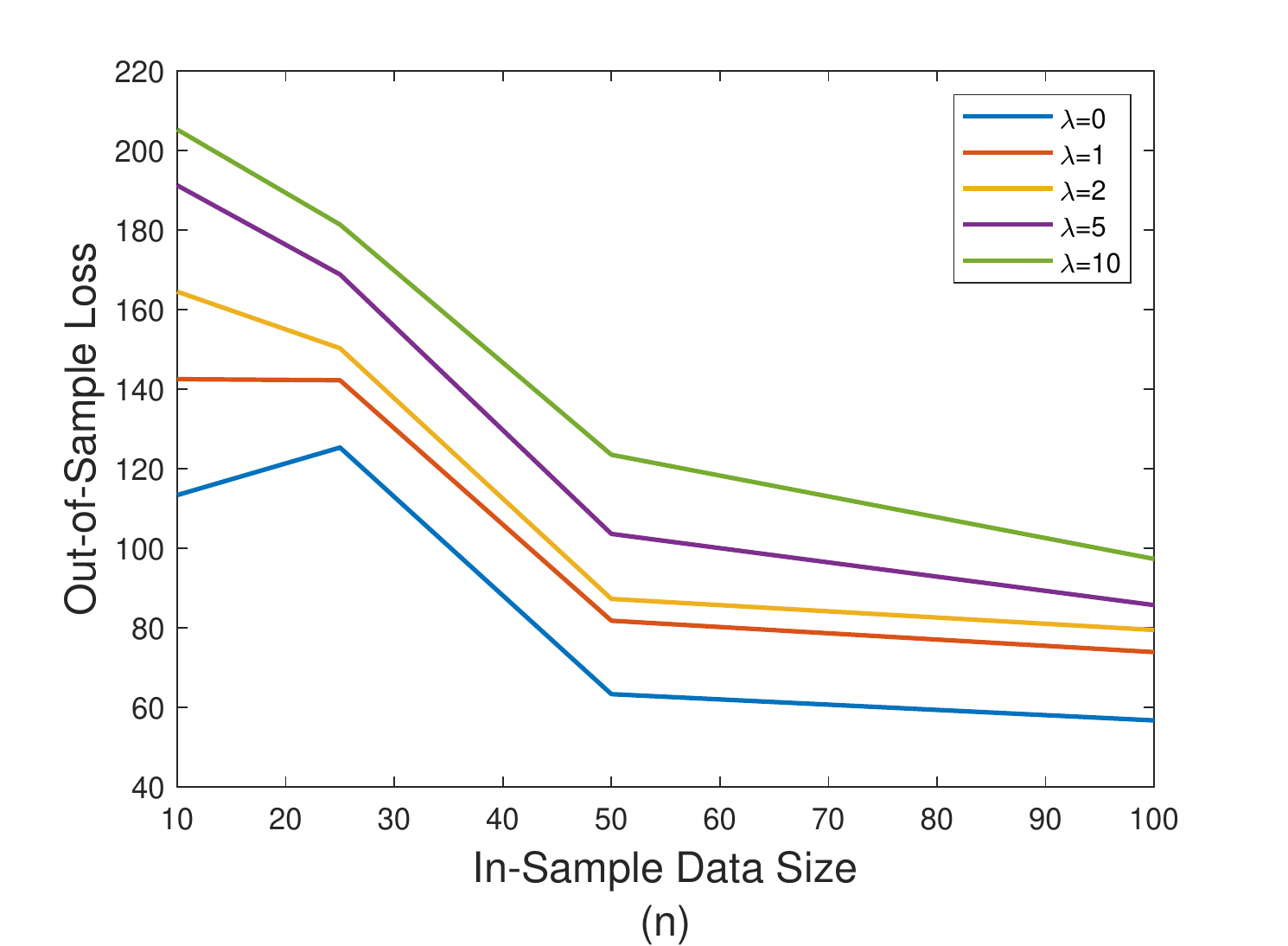}
  \caption{${\bm\gamma}=(4.5,2.5)$}
  \label{fig:newsvendor3}
\end{subfigure}
\begin{subfigure}[t]{.49\textwidth}
\centering
    \includegraphics[width=1.0\textwidth]{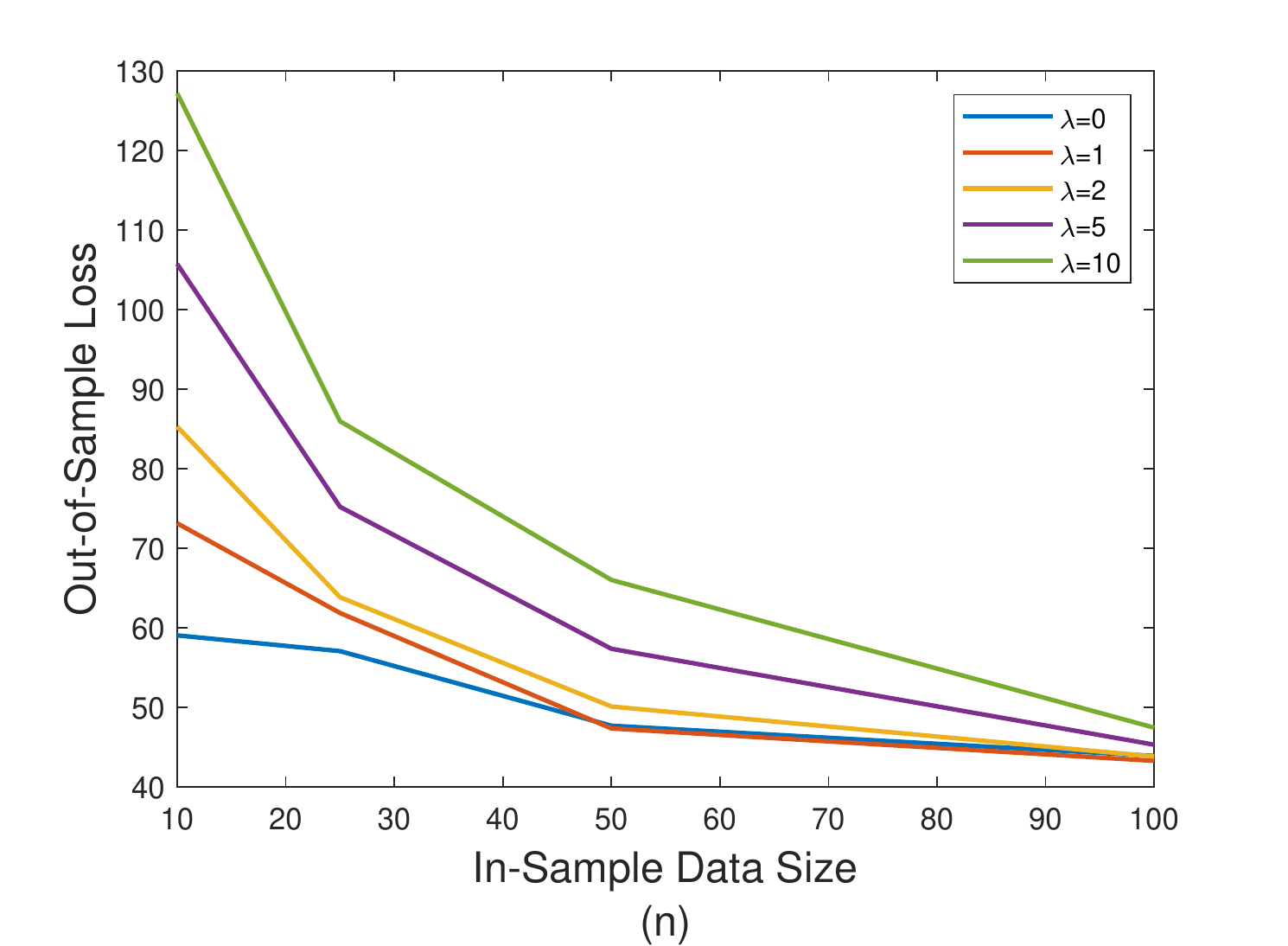}
  \caption{${\bm\gamma}=(4.5,5)$}
  \label{fig:newsvendor6}
\end{subfigure}
\begin{subfigure}[t]{.49\textwidth}
\centering 
      \includegraphics[width=1.0\textwidth]{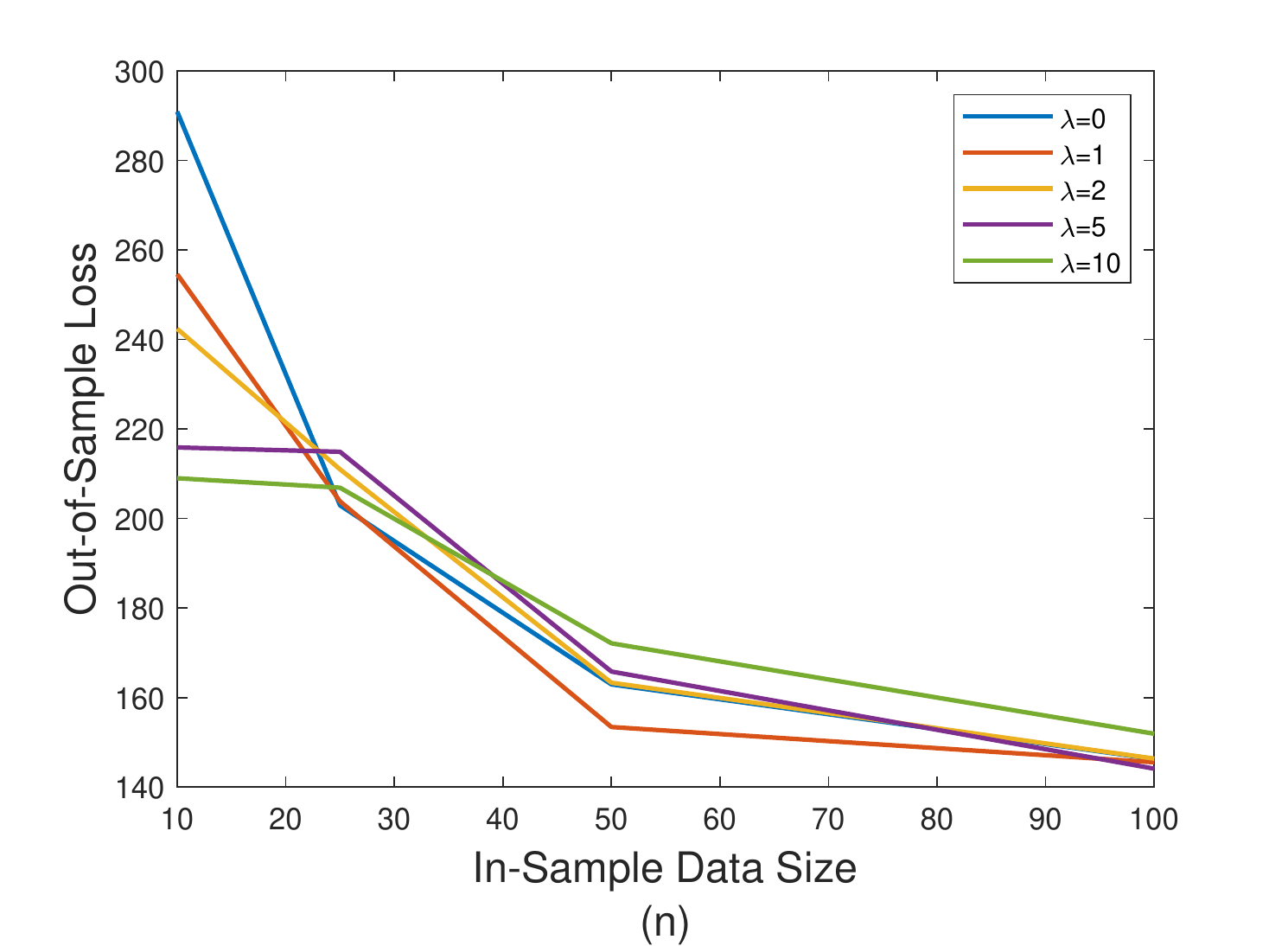}
  \caption{${\bm\gamma}=(1.5,5)$}
  \label{fig:newsvendor4}
\end{subfigure}
\begin{subfigure}[t]{.49\textwidth}
\centering
      \includegraphics[width=1.0\textwidth]{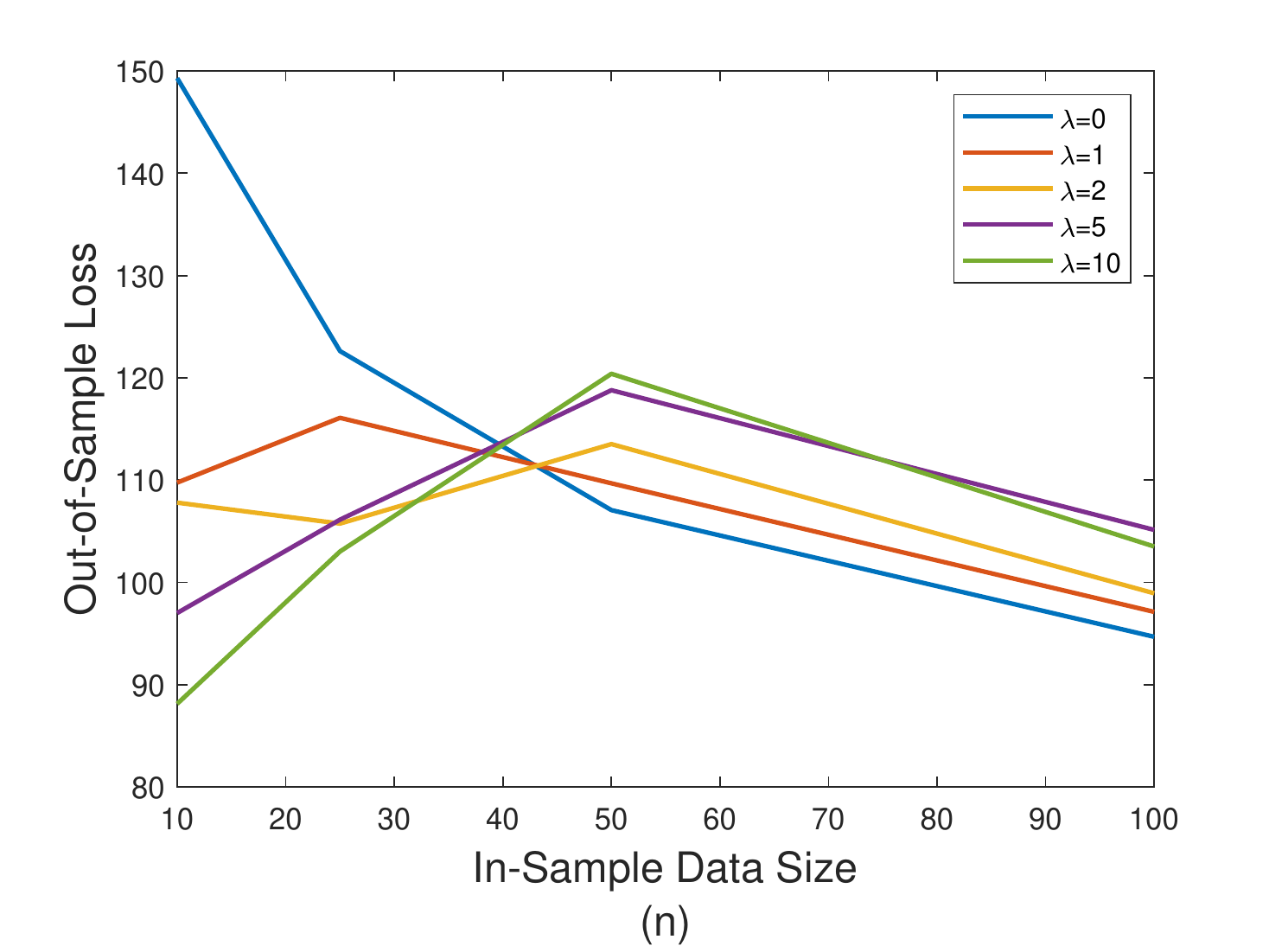}
  \caption{${\bm\gamma}=(3,5)$}
  \label{fig:newsvendor5}
\end{subfigure}
\caption[Effect of regularization parameter $\lambda$ on out-of-sample loss for different side information covariates $\bm\gamma$]{Effect of regularization parameter $\lambda$ on the average out-of-sample loss for different sizes of training datasets at different side information covariates $\bm\gamma$.}
\label{fig:newsvendor}
\end{figure}

\Cref{fig:newsvendor} shows the results obtained. From the figure, we note that at points ${\bm\gamma}=(1.5,2.5)$  and ${\bm\gamma}=(3,2.5)$, where the density function values  $g(\bm \gamma)$ for the bivariate normal  are much smaller, the regularization scheme is quite effective and the average out-of-sample loss decreases significantly with the increase in regularization parameter $\lambda$. On the other hand, for points close to the mean $\bm \mu=[7.5,5]^\top$, for example, ${\bm\gamma}=(4.5,2.5)$ and ${\bm\gamma}=(4.5,5)$, the unregularized ($\lambda=0$) version of the formulation perform much better. This is quite intuitive since, in regions of high density, the NW estimator forms a good approximation to the true conditional expectation even for small sample sizes. By contrast, in the regions where the density values are smaller, the regularization term seeks to control the amount of overfitting to the limited data available. Another important observation that we make is that, in regions of moderate density values, for example, for points ${\bm\gamma}=(1.5,5)$ and ${\bm\gamma}=(3,5)$, the regularization helps in the setting where the sample sizes are smaller ($n=10$ and $n=25$). This is consistent with the generalization bound obtained in \Cref{cor:generalization_bound}. %{\color{blue} (the font is too small in the figure...)}     

\subsection{Wind energy commitment}

We next apply our DRO formulation to the wind energy problem considered in \cite{hannah2011approximate} and \cite{kim2011optimal}. At the beginning of day $t$, a wind energy producer determines the wind energy commitment levels $\bm x \in \RR^{24}$ for the next $24$ hours. The day-ahead prices $\bm \pi^t \in \RR^{24}$ are known to the decision maker. However, the hourly amounts of wind energy $\bm \xi^t \in \RR^{24}$ generated for the next $24$ hours are uncertain. If the actual production falls short of the commitment level, there is a penalty of twice the respective day-ahead price for each unit of unsatisfied demand. As the wind energy is generally highly correlated to the past data, we consider the side information vector ${\bm \gamma}={\bm \xi}^{t-1}$ in the implementation. Based on the derivation in~\eqref{eq:dualform-conic}, we arrive at the following DRO formulation for the wind energy commitment problem:
\begin{equation}\label{eq:wind-energy}
\begin{array}{cll}
\min& \alpha - \displaystyle \sum_{i=1}^n  \sqrt{\overline{w}_i} \beta_i + \frac{\lambda}{\sqrt{2}} \nu  \\
\text{s.t.} & \alpha \geq -\bm x^\top \bm \pi^i+\displaystyle 2 \sum_{j=1}^{24} \pi^i_j \max\{x_j-\xi^i_j,0\} +\frac{\beta_i}{\sqrt{\overline{w}_i}} &\forall i\in [n], \\
& \bm x \in \RR_+^{24}, \;\alpha\in\mb R, \; (\bm\beta,\nu) \in \SOC(n+1).
\end{array}
\end{equation}

In the experiment, we obtain the hourly wind energy data from North American Land Data Assimilation System\footnote{\url{https://climatedataguide.ucar.edu/climate-data/nldas-north-american-land-data-assimilation-system}} from $2002$ to $2011$ at the following locations: \textbf{Rhode Island} (41.8252N, 71.4188W) and \textbf{North Carolina} (33.9375N, 77.9375W).
%and \textbf{Ohio} (41.8125N, 81.5625W) 
The hourly day-ahead prices are downloaded from the publicly available PJM market dataset.\footnote{\url{http://dataminer2.pjm.com/feed/da_hrl_lmps/definition}} As the wind energy and day-ahead prices are closely related to seasons, we divide each year's data into four parts according to different seasons and conduct out-of-sample tests on each of them separately. In each season, we assume the decision maker has access to the first $n+1$ days of data, and plans for  the commitment levels for the next day. To incorporate side information, the historical data is then rearranged to $n$ samples of the form $\{(\bm \gamma^i, \bm \xi^{i+1})\}_{i=1}^{n}$, where we set $\bm \gamma^i=\bm \xi^i$ to be the covariate vector comprising of the  wind energy productions on day $i$. As $\bm \gamma^i$ is a $24$-dimensional vector with high correlations between its components, we adopt the dimensionality reduction procedure described in Section~\ref{sec:high_dim} to determine a $3$-dimensional subspace that explains more than $90\%$ of the variability of the historical observations. The NW kernel weights $\overline{w}_i$ are consequently computed using the projected data. We solve problem \eqref{eq:wind-energy} to obtain the optimal commitment levels and evaluate its true profit using the next day's data. We then drop the first day's data and include the data of the $(n+2)$th day, and move on to the planning for the $(n+3)$th day. We repeat this process $N$ times, and compute the total profit for these $N$ days as one trial's result. As there are 40 seasons in 10 years, we have $40$ trials in total.

 We then benchmark our Regularized Nadaraya-Watson (RNW) method with sample average approximation (SAA), the unregularized Nadaraya-Watson (NW) \citep{hannah2011approximate} and the residual-based distributionally robust optimization  (ERDRO) \citep{kannan2020residuals} methods in out-of-sample experiments. We also implemented the regularized linear decision rule (LDR) method \citep{baziergeneralization}; however, the method performs poorly and thus we do not  report the results. LDR fails in this experiment because the wind energy data is nonlinear and very complicated; such a parameterized regression model cannot fit it well and thus yields poor predictions. The ERDRO method  assumes $\xit$ can be modeled in terms of $\gammat$ as  $\xit = f(\tilde{\bm\gamma})+\tilde{\bm\epsilon}$, where $f(\gammat)=\mb E[\xit\lvert \tilde{\bm\gamma} ]$ is the regression function while $\tilde{\bm\epsilon}$ are mean zero errors. For the same reason with LDR, we adopt the nonparametric Nadaraya-Watson regression model to predict $\xit$ conditioned on the side information $\bm\gamma$, and we solve for the best commitment level in view of the worst-case distribution from within a modified $\chi^2$ ambiguity set. With these settings, we find that the ERDRO model  performs really well for this particular problem. 
 
 In the experiment, we set $n=14$ and $N=25$. The radius of the ambiguity set $\lambda$ and the bandwidth parameter $C_h$ are determined following a cross-validation procedure. 
In each trial, we split the first $2/3$ of the training set into a sub-training set and keep the remaining samples as a sub-validation set. Then we set the radius $\lambda$ to zero, and collect the total return of different bandwidth parameters $C_h\in [5\times 10^2, 5 \times 10^4]$ on a logarithm searching grid with 9 equidistant points. Next, we fix the best $C_h$ obtained in the previous procedure and tune for the best radius $\lambda \in [10^{-2},10^2]$ on a logarithm searching grid with 17 equidistant points.

\begin{table}[h]
    \centering
    \begin{tabular}{p{1.5cm} c  r  r  r}
    \hline
         Site & Statistic  & NW  & ERDRO& RNW \\
     \hline
        \multirow{3}{1.5em}{RI} & Mean  & 55.5 & 96.1 & 110.0\\
         &20th prct.  & -6.3  & 13.1 & 45.7\\
         & 80th prct. & 116.5  & 163.6 & 164.8\\
        %  &Worst case & - & - & - & - & -\\
         \hline
         \multirow{3}{1.5em}{NC} & Mean  & 64.2 & 69.1 & 79.1\\
         &20th prct.  & -7.9  & -2.8 & 7.0\\
         & 80th prct. & 179.1  & 192.5 & 189.4\\
        %  &Worst case & - & - & - & - & -\\
         \hline
    \end{tabular}
    \caption{Statistics of improvements over SAA $(\%)$}
    \label{tab:wind}
\end{table}

% Table \ref{tab:wind} suggests that these three methods performs favorably relative to its competitors: they all achieve promising improvement by investigating the underlying relationship between side information and uncertainties. Meanwhile, as the ERDRO and RNW method include the robustification consideration, they obtain much higher improvements in the RI dataset compared with the naive NW regression method. 
Table \ref{tab:wind} presents the statistics of improvement over the baseline sample average approximation (SAA) for the  unregularized Nadaraya-Watson (NW) method, the residual-based distributionally robust optimization method (ERDRO), and our Regularized Nadaraya-Watson (RNW) method. In each trial, the improvement over SAA is computed using the rule $d(x,y)=2(x-y)/\left(|x|+|y| \right)$, where $x$ is the $N$ days' total profit obtained by one of the three methods and $y$ is the total profit obtained by SAA.
The results indicate that our RNW method performs favorably relative to its competitors: it achieves the greatest mean improvements over SAA.
Meanwhile, we observe that the NW and ERDRO methods also attain significant improvements over SAA, which implies exploring side information indeed helps decision makers better estimate uncertainties. Moreover, with the benefit of the distributionally robust setting, the ERDRO and RNW methods are more robust in terms of 20th percentile compared with the vanilla NW regression method. 
And compared with the ERDRO model which applies regression to predict the high dimensional uncertain parameter $\xit$ conditioned on $\tilde{\bm\gamma} = \bm\gamma$, our method predicts the conditional expectation $\EE_{\bm \rho}[\ell(\bm x,\xit)]$ directly. Thus, we avoid the errors that come from high dimensional regression and obtain a better performance.

% And unlike our method which applies regression for predicting the conditional expectation $\EE_{\bm \rho}[\ell(\bm x,\xit)]$ directly, the ERDRO applies it to predict the high dimensional uncertain parameter $\xit$ conditioned on $\tilde{\bm\gamma} = \bm\gamma$. 
% Therefore, the optimality bound obtained by the ERDRO method suffers from the curse of dimensionality, whereas our method does not encounter this issue.  

%\subsection{Capital Budgeting Problem}
\section{Concluding remarks}
\label{sec:conclusion}
The NW approximation has recently garnered an increasing interest due to its significance in the context of decision-making under uncertainty with side information. The scheme, however, has so far resisted any sensible result about its out-of-sample performance.
% Due to our lack of understanding about the  estimator's tail behavior
In this paper, we established for the first time a complete, comprehensive theoretical result on the performance guarantees of the approximation. The new result inspired us to design a novel regularization scheme that can better mitigate the overfitting effects. In contrast to the popular $L_2$ regularization scheme which attempts to minimize the norm of the decision vector and may pointlessly encourage  an optimal solution that  is close to the origin, our proposed regularization scheme is directly constructed using the conditional standard deviation term appearing in the  theoretical bounds and can faithfully prioritize an optimal solution that generalizes well. In the future, it would be interesting to extend the model to the multi-stage setting, and devise a tractable solution procedure with similar performance guarantees for dynamic stochastic optimization problems.

%\input{PCA-attempts.tex}
%\paragraph{Acknowledgements.} Grani A. Hanasusanto is supported by the National Science Foundation grant no.~1752125.
	\end{onehalfspace}

\bibliographystyle{plainnat}
\bibliography{references}

\newpage
\begin{onehalfspace}
\appendix
\appendixpage
\addappheadtotoc
\numberwithin{equation}{section}
\section{Proof of Theorem~\ref{thm:MDP}}
\begin{proof}
To prove the desired result, we define a random variable $\tilde{L} = L(\bm\xit)$ and $\hat{f}(\bm\gamma, L)$ to be the joint density function of $(\gammat,\tilde{L})$. Then, we show that our setting is eligible to apply Theorem 2 in \cite{mokkadem:08}, which requires the following conditions to hold.
\begin{enumerate}
    \item[(A)] The kernel function $\mathcal K : \mathbb{R}^p \rightarrow \mathbb{R}$ is a bounded and integrable function that satisfies
    \begin{equation*}
    \int_{\mathbb{R}^p} \mathcal K(\bm\theta)\mathrm{d}\bm\theta= 1 \quad\text{and}\quad \lim_{\Vert \bm\theta \Vert \rightarrow \infty} \mathcal K(\bm\theta) = 0
    \end{equation*}
    
    \item[(B)] For any $u\in\mathbb{R}$, the function $\displaystyle \bm{t} \rightarrow \int_{\mathbb{R}} \exp( u L) \hat{f}(\bm t , L) \mathrm{d} L$ is continuous at $\bm{t}=\bm\gamma$ and bounded.
    
    \item[(C)] For any $u\in\mathbb{R}$, the functions $\displaystyle \bm{t} \rightarrow \int_{\mathbb{R}} u^2 L^2 \hat{f}(\bm t, L) \mathrm{d} L$ and $\displaystyle \bm{t} \rightarrow \int_{\mathbb{R}} u L \hat{f}(\bm t, L) \mathrm{d} L$ is continuous at $\bm{t}=\bm\gamma$, and the marginal density $\hat{f}_{\tilde{\bm\gm}}(\bm{\gamma}) \neq 0$, where $\bm{\gamma}$ is the fixed side information of interest in~\eqref{eq:conditional_expectation_problem}.
    
    \item[(D)] The sequence $\{a_n\}_{n\in\mathbb{N}}$ is chosen such that
    \begin{equation*}
    \lim_{n\rightarrow \infty} a_n=\infty \quad\text{and}\quad \lim_{n\rightarrow \infty} \frac{a_n^2}{n h_n^p}=0.
    \end{equation*}
    
    \item[(E)] There exists an integer $S\geq 2$ such that 
	\begin{itemize}
	\item[(i)] For $\forall s \in [S-1]$, $\forall j \in [p]$, 
    \begin{equation*}
    \int_{\mathbb{R}^p} \theta_j^s\mathcal{K}(\bm\theta)\mathrm{d}\bm\theta = 0 \quad\text{and}\quad  \int_{\mathbb{R}^p} \vert\theta_j^S\mathcal{K}(\bm\theta) \vert \mathrm{d}\bm\theta < \infty .
    \end{equation*}
    \item[(ii)] The chosen sequence $\{a_n\}_{n\in\mathbb{N}}$ satisfies $\lim_{n\rightarrow \infty} a_n h_n^S =0$.
    \item[(iii)] Both functions $\hat{f}_{\tilde{\bm \gamma}}(\bm\gamma)$ and $\int_0^1 L \cdot \hat{f}(\bm\gamma , L) \mathrm{d}L$ are $S$-times differentiable on $\mathbb{R}^p$, and their differentials of order $S$ are bounded and continuous at $\bm\gamma$.
\end{itemize}	    
\end{enumerate}

In what follows, we will show that Assumption \ref{as2} and all conditions in the theorem imply the conditions (A)-(E) above. We first notice that condition (A) holds for our choice of exponential kernel \eqref{eq:Gaussian_kernel}. To show that condition (B) holds, we note that
\begin{equation}
\label{eq:ConditionExp}
\EE[\exp( u L(\bm\xi)) \,|\,\gammat=\bm\gamma] 
= \frac{\int_{\mathbb{R}} \exp( u L) \hat{f}(\bm\gamma , L) \mathrm{d} L}{\hat{f}_{\tilde{\gamma}}(\bm\gamma)} 
= \frac{\int_{\RR^q}\exp(uL(\bm\xi))f(\bm\gamma,\bm\xi)\mathrm d\bm\xi}{f_{\tilde{\bm \gamma}}(\bm\gamma)}.
\end{equation}
Since $\hat{f}_{\tilde{\bm \gamma}}(\bm\gamma) = f_{\tilde{\bm \gamma}}(\bm\gamma)$ and $\int_{\mathbb{R}} \exp( u L) \hat{f}(\bm\gamma , L) \mathrm{d} L = \int_{\RR^q}\exp(uL(\bm\xi))f(\bm\gamma,\bm\xi)\mathrm d\bm\xi$, therefore by condition~2 stated in the theorem, condition (B) holds. Following the same argument as above in~\eqref{eq:ConditionExp} and the conditions 1 and 3 stated in the theorem, the first part of condition (C) holds. %{\color{blue} We may need extra assumption for the second part of condition (C), which I think it's fine to make this assumption.} 
Condition (D) holds as stated in the statement of the theorem. For condition (E), we show that it holds in our setting for $S = 2$.  Firstly, we note that $\int_{\mathbb{R}^p} \theta_j \mathcal{K}(\bm\theta)\mathrm{d}\bm\theta = 0 $, $\forall j \in [p]$, since the expectation of the distribution \eqref{eq:Gaussian_kernel} is $\mathbf{0}$ due to the symmetry of this distribution. Moreover,  
\begin{equation*}
\int_{\mathbb{R}^p} \vert\theta_j^2\mathcal{K}(\bm\theta) \vert \mathrm{d}\bm\theta = \frac{1}{Z} \int_{\mathbb{R}^p}  \Big\vert\theta_j^2\exp\left(-\|\bm \theta\|_2\right) \Big\vert \mathrm{d}\bm\theta \leq \frac{1}{Z} \int_{\mathbb{R}^p}  \Big\vert\theta_j^2\exp\left(- \|\bm \theta\|_1/p\right) \Big\vert \mathrm{d}\bm\theta < \infty,
\end{equation*}
where the last inequality holds because
\begin{equation*}
\int_{\mathbb{R}}  \Big\vert \exp\left(- \vert\theta_i \vert/p\right) \Big\vert \mathrm{d} \theta_i < \infty \quad\text{and}\quad  \int_{\mathbb{R}}  \Big\vert \theta_j^2 \exp\left(- \vert\theta_j \vert/p\right) \Big\vert \mathrm{d} \theta_j <\infty.
\end{equation*}
Thus, part (i) of condition (E) holds. Part (ii) of condition (E) also holds with $S=2$ as stated in the statement of the theorem. Following the same argument that we used in \eqref{eq:ConditionExp} to show that condition (B) is satisfied, part (iii) of condition (E) holds from condition 3 in the statement of the theorem. %{\color{blue} We may need extra assumption about $f_{\tilde{\bm\gamma}}(\bm\gamma)$ as well, eg $f_{\tilde{\bm\gamma}}(\bm\gamma)>0$.} 

Therefore, we can apply Theorem 2 in \cite{mokkadem:08} which implies that the sequence $\{a_n(\EErho[L(\xit)]-\EErhohat[L(\xit)])\}_{n\in\mathbb N}$ obeys a large deviation principle with speed $nh_n^p/a_n^2$ and rate function 
\begin{equation}
\label{eq:generic_rate}
	I_{\bm\gamma}(y)=\frac{y^2 f_{\tilde{\bm\gamma}}(\bm\gamma)}{2\VVrho[L(\xit)]\int_{\RR^p}\mathcal K^2(\bm \theta)\mathrm d\bm \theta}.
	\end{equation}
\end{proof}

\section{Proof of \Cref{thm:generalization_bound}}
\label{app:generalization_bound}
\begin{proof}
We set the function in \eqref{eq:MDP} to $L(\bm\xi)=\ell(\bm x,\bm\xi)$ and verify that the conditions in Theorem~\ref{thm:MDP} are satisfied. To establish continuity of $\int_{\Xi}\ell(\bm x,\bm\xi)^2f(\bm\gamma,\bm\xi)\mathrm d\bm\xi$ at $\bm\gamma$, we fix  $\epsilon>0$ and show that there exists $\delta>0$ such that
	\begin{equation}
	\label{eq:continuity_in_rho}
	    \|\bm\gamma-\bm\gamma'\|\leq\delta \Longrightarrow \left|\int_{\Xi}\ell(\bm x,\bm\xi)^2f(\bm\gamma,\bm\xi)\mathrm d\bm\xi-\int_{\Xi}\ell(\bm x,\bm\xi)^2f(\bm\gamma',\bm\xi)\mathrm d\bm\xi\right|\leq \epsilon. 
	\end{equation}
Let $\mu(\Xi)$ be the Lebesgue measure of the support set $\Xi$. By assumption \ref{as1}, the following chain of inequalities hold:
	\begin{equation*}
	\begin{array}{rl}
	    \displaystyle\left|\int_{\Xi}\ell(\bm x,\bm\xi)^2f(\bm\gamma,\bm\xi)\mathrm d\bm\xi-\int_{\Xi}\ell(\bm x,\bm\xi)^2f(\bm\gamma',\bm\xi)\mathrm d\bm\xi\right|&\displaystyle\leq \sup_{\bm\xi\in\Xi}\left|\ell(\bm x,\bm\xi)^2\right|\sup_{\bm\xi\in\Xi}|f(\bm\gamma,\bm\xi)-f(\bm\gamma',\bm\xi)|\mu(\Xi)\\
	    &\displaystyle\leq \sup_{\bm\xi\in\Xi}|f(\bm\gamma,\bm\xi)-f(\bm\gamma',\bm\xi)|\mu(\Xi).
	   \end{array}
	\end{equation*}
	We now show that there exists $\delta>0$ such that 
	\begin{equation}
	\label{eq:eps-delta}
		\|\bm\gamma-\bm\gamma'\|\leq\delta\;\Longrightarrow \;\sup_{\bm\xi\in\Xi}|f(\bm\gamma,\bm\xi)-f(\bm\gamma',\bm\xi)|\leq\epsilon/\mu(\Xi),
	\end{equation}
which is sufficient to prove the claim. Suppose for the sake of contradiction the  implication \eqref{eq:eps-delta} does not hold. That is, for any $\delta>0$, there exist $\bm\gamma_\delta'$ with $\|\bm\gamma-\bm\gamma_\delta'\|\leq\delta$ and  $\bm\xi_\delta\in\Xi$ such that $     |f(\bm\gamma,\bm\xi_\delta)-f(\bm\gamma_\delta',\bm\xi_\delta)|>\epsilon/\mu(\Xi)$. %We  set $\delta\downarrow 0$. 
By construction, we have $\lim_{\delta\rightarrow 0}\bm\gamma_\delta'=\bm\gamma$. Let $\bm\xi^\star$ be a limit point of the sequence $\{\bm\xi_\delta\}$ as $\delta\rightarrow 0$. By the compactness of the support set in assumption \ref{as1} we have $\bm\xi^\star\in\Xi$. The continuity of the density function in assumption \ref{as2} then implies that
    \begin{equation*}
        \epsilon/\mu(\Xi)\leq\lim_{\delta\rightarrow 0}|f(\bm\gamma,\bm\xi_\delta)-f(\bm\gamma_\delta',\bm\xi_\delta)|=|f(\bm\gamma,\bm\xi^\star)-f(\bm\gamma,\bm\xi^\star)|=0,
    \end{equation*} 
which is a contradiction because $\epsilon/\mu(\Xi)>0$. We may thus conclude that the first condition in Theorem \ref{thm:MDP} is indeed satisfied. 
	
By following the same argument, one can show that   $\int_{\RR^q}\exp(u\ell(\bm x,\bm\xi))f(\bm\gamma,\bm\xi)\mathrm d\bm\xi$ is continuous at $\bm\gamma$. The boundedness of the expression holds because $\sup_{\bm\xi\in\Xi}\exp(u\ell(\bm x,\bm\xi))\leq \exp(u)$ for every $u\in\RR$. Thus, the second condition in Theorem \ref{thm:MDP} is also satisfied. Finally, by the Leibniz's rule we have
\begin{equation*}
\begin{array}{ll}
\displaystyle\frac{\partial}{\partial\gamma_i}\int_{\Xi}\ell(\bm x,\bm\xi)f(\bm\gamma,\bm\xi)\mathrm d\bm\xi=\int_{\Xi}\ell(\bm x,\bm\xi)\frac{\partial}{\partial\gamma_i}f(\bm\gamma,\bm\xi)\mathrm d\bm\xi&\forall i\in[p]\qquad\textup{and}\\[4mm]
\displaystyle\frac{\partial^2}{\partial\gamma_i\partial\gamma_j}\int_{\Xi}\ell(\bm x,\bm\xi)f(\bm\gamma,\bm\xi)\mathrm d\bm\xi=\int_{\Xi}\ell(\bm x,\bm\xi)\frac{\partial^2}{\partial\gamma_i\partial\gamma_j}f(\bm\gamma,\bm\xi)\mathrm d\bm\xi&\forall i,j\in[p].
\end{array}
\end{equation*}
%where the interchange of the differentiation and the integration operators is valid by the dominated convergence theorem. {\color{blue} (why? I thought Leibniz's rule is enough, and dominated convergence theorem is for exchanging limit and integration operator?)}
Thus, in view of our assumption that $\partial f(\bm\gamma,\bm\xi)/ \partial\gamma_i$  and $\partial^2 f(\bm\gamma,\bm\xi)/ (\partial\gamma_i\partial\gamma_j)$ are continuous and bounded, we may  apply  the  same  argument to conclude that   the third condition in Theorem \ref{thm:MDP} is also satisfied. 
	
Next, let the closed and the open sets  in \eqref{eq:MDP} be defined as $\mathcal C=(-\infty,-\epsilon']\cup[\epsilon',\infty)$ and  $\mathcal O=(-\infty,-\epsilon')\cup(\epsilon',\infty)$, respectively. The function $I_{\bm\gamma}(y)$ is a convex quadratic function centered at $0$, which implies that	    $\inf_{y\in\mathcal C} I_{\bm\gamma}(y)=\inf_{y\in\mathcal O} I_{\bm\gamma}(y)=I_{\bm\gamma}(\epsilon')$.	Thus, we obtain
		\begin{equation}
		\begin{array}{rl}
\displaystyle	- I_{\bm\gamma}(\epsilon')&\displaystyle\leq \;\liminf_{n\rightarrow \infty} \frac{1}{\nu_n}\log \mathbb P\left(a_n\left(\EErho[\ell(\bm x,\xit)]-\EErhohat[\ell(\bm x,\xit)]\right)\in\mathcal O\right)\\[2mm]
&\displaystyle\leq\; 	\limsup_{n\rightarrow \infty} \frac{1}{\nu_n}\log \mathbb P\left(a_n\left(\EErho[\ell(\bm x,\xit)]-\EErhohat[\ell(\bm x,\xit)]\right)\in\mathcal C\right)\leq \;- I_{\bm\gamma}(\epsilon'),
\end{array}
\end{equation}
which gives rise to the stronger result
\begin{equation}
\frac{1}{\nu_n}\log \mathbb P\left(\left|a_n\left(\EErho[\ell(\bm x,\xit)]-\EErhohat[\ell(\bm x,\xit)]\right)\right|\geq\epsilon' \right)= - I_{\bm\gamma}(\epsilon') + o(1).
\end{equation}
Multiplying both sides of the inequality with $\nu_n$, taking exponential, and substituting the definition of $G_{\bm\gamma}(\epsilon)$ yield
\begin{equation*}
\mathbb P\left(\left|a_n\left(\EErho[\ell(\bm x,\xit)]-\EErhohat[\ell(\bm x,\xit)]\right)\right|\geq\epsilon' \right)= \exp\left(- \frac{(\epsilon')^2 \nu_n g(\bm\gamma)}{\VVrho[\ell(\bm x,\xit)]}+o(\nu_n)\right). 
\end{equation*}
Since $\nu_n=n h_n^p/a_n^2$, we have
\begin{equation*}
\mathbb P\left(\left|a_n\left(\EErho[\ell(\bm x,\xit)]-\EErhohat[\ell(\bm x,\xit)]\right)\right|\geq\epsilon'\right)= \exp\left(- \frac{(\epsilon')^2 (n h_n^p/a_n^2) g(\bm\gamma)}{\VVrho[\ell(\bm x,\xit)]}+o(n h_n^p/a_n^2)\right). 
\end{equation*}	
We consider $a_n$'s that are strictly positive (there exists such $a_n$'s. For example, $a_n = \log n$) and denote $\epsilon' = \epsilon a_n$ for some constant $\epsilon$, we obtain
\begin{equation*}
\mathbb P\left(\left|\EErho[\ell(\bm x,\xit)]-\EErhohat[\ell(\bm x,\xit)]\right|\geq\epsilon\right)= \exp\left(- \frac{\epsilon^2 (n h_n^p) g(\bm\gamma)}{\VVrho[\ell(\bm x,\xit)]}+o(n h_n^p/a_n^2)\right). 
\end{equation*}	
Since $g(\bm\gamma)$ and $\VVrho[\ell(\bm x,\xit)]$ are constants for fixed $\bm{x}$ and $\bm\gamma$, the above is equivalent to 
\begin{equation*}
\mathbb P\left(\left|\EErho[\ell(\bm x,\xit)]-\EErhohat[\ell(\bm x,\xit)]\right|\geq\epsilon\right)= \exp\left(- n h_n^p \frac{\epsilon^2 g(\bm\gamma) (1 + o(1/a_n^2))}{\VVrho[\ell(\bm x,\xit)]}\right). 
\end{equation*}	
Since $\lim_{a_n \rightarrow \infty} a_n^2 = \infty$, we complete the proof.
\end{proof}

\section{Proof of \Cref{thm:generalization_bound_continuousX}}
\noindent Before we prove \Cref{thm:generalization_bound_continuousX}, we first obtain some useful results for Lipschitz continuous loss functions in the following lemma.
\label{sec:generalization_bound_continuousX}
\begin{lemma}
\label{lem:lipschitz_loss}
Assume that the loss function $\ell(\bs x,\tilde{\bxi})$  is M-Lipschitz continuous in $\bs{x}$, i.e., there exists a constant $M>0$ such that 
\begin{equation}
\label{eq:lem_obj_fun_lips_assum}
\left| \ell(\bs{x},\bs{\xi}) - \ell(\bs{x}',\bs{\xi}) \right| \leq M \Vert \bs{x} - \bs{x}' \Vert \quad \forall \bs{x}, \bs{x}' \in \mathcal{X}, \ \bs{\xi}\in \Xi.
\end{equation}
Then, for any $\bs \gamma \in \mb R^p$, we have
\begin{equation*}
\begin{array}{rll}
\left\lvert \EErho[\ell(\bs  x,\xit)]  - \EErho[\ell(\bs x',\xit)] \right\vert &\leq \;\; M\Vert \bs{x} - \bs{x}' \Vert & \forall \bs{x}, \bs{x}' \in \mathcal{X},\\
\left\vert  \EErhohat[\ell(\bs  x,\xit)] - \EErhohat[\ell(\bs{x}',\xit)] \right\vert &\leq \;\; M\Vert \bs{x} - \bs{x}' \Vert & \forall \bs{x}, \bs{x}' \in \mathcal{X},\\
\left\vert  \sqrt{\VVrho[\ell(\bs x,\xit)]} - \sqrt{\VVrho[\ell(\bs x',\xit)]} \right\vert &\leq \;\; M\Vert \bs{x} - \bs{x}' \Vert & \forall \bs{x}, \bs{x}' \in \mathcal{X},\\
\left\vert  \sqrt{\VVrhohat[\ell(\bs x,\xit)]} - \sqrt{\VVrhohat[\ell(\bs x',\xit)]} \right\vert &\leq \;\; M\Vert \bs{x} - \bs{x}' \Vert & \forall \bs{x}, \bs{x}' \in \mathcal{X}.
\end{array}
\end{equation*}
\end{lemma}

\begin{proof}
The first two inequalities above can be verified by directly applying \eqref{eq:lem_obj_fun_lips_assum}. To show that the third inequality holds, we use the observation used in the proof of Proposition \ref{prop:bound_sample_stddev} that for any $\bs{x} \in \mathcal{X}$,
\begin{equation*}
\sqrt{\VVrho[\ell(\bs x,\xit)]}=\min_{t\in[0,1]}\sqrt{ \EErho[(\ell(\bs x,\xit)-t)^2]}.
\end{equation*}
Without any loss of generality, we assume $\sqrt{\VVrho[\ell(\bs x,\xit)]} \geq \sqrt{\VVrho[\ell(\bs x',\xit)]}$, and we obtain
\begin{equation*}
\begin{array}{rll}
\sqrt{\VVrho[\ell(\bs x,\xit)]} - \sqrt{\VVrho[\ell(\bs x',\xit)]}  &\displaystyle = \;\; \min_{t\in[0,1]}\sqrt{ \EErho[(\ell(\bs x,\xit)-t)^2]} - \min_{t\in[0,1]}\sqrt{ \EErho[(\ell(\bs{x}',\xit)-t)^2]} \\
&\displaystyle \leq \;\; \sqrt{ \EErho[(\ell(\bs x,\xit)-t')^2]} - \sqrt{ \EErho[(\ell(\bs{x}',\xit)-t')^2]} ,
\end{array}
\end{equation*}
where $t' = \underset{t\in[0,1]}{\arg\min}\sqrt{ \EErho[(\ell(\bs{x}',\xit)-t')^2]}$. 

Next, we note that the function $\sqrt{ \EErho[(\cdot)^2]}$ constitutes a semi-norm, which gives us that 
\begin{equation*}
\begin{array}{rll}
\sqrt{\VVrho[\ell(\bs x,\xit)]} - \sqrt{\VVrho[\ell(\bs x',\xit)]} &\displaystyle \leq \;\; \sqrt{ \EErho[(\ell(\bs x,\xit)-t')^2]} - \sqrt{ \EErho[(\ell(\bs{x}',\xit)-t')^2]}  \\
&\displaystyle \stackrel{(i)}{\leq} \;\; \sqrt{ \EErho[(\ell(\bs x,\xit)-t'-\ell(\bs{x}',\xit)+t')^2]} \\
&\displaystyle \leq \;\; \sqrt{ \EErho[(\ell(\bs x,\xit)-\ell(\bs{x}',\xit))^2]} \\
&\displaystyle \stackrel{(ii)}{\leq} \;\; M\Vert \bs{x} - \bs{x}' \Vert.
\end{array}
\end{equation*}
Here the inequality $(i)$ follows from the reverse triangle inequality while inequality $(ii)$ is obtained by noting that $\ell(\bs x,\bs \xi)$ is $M$-Lipschitz continuous for all $\bs \xi \in \Xi$. Thus, we have verified the third inequality. Using the same argument, we can show that the fourth inequality also holds.
\end{proof}

\begin{proof}[Proof of \Cref{thm:generalization_bound_continuousX}]
\noindent From \Cref{cor:generalization_bound}, we have that for a fixed $\bs x \in \mathcal{X}$, 
\begin{equation}
\begin{aligned}
\mb E_{\bs\gm}[\ell(\bs x,\tilde{\bxi})]\leq \hat{\mb E}_{\bs\gm}[\ell(\bs x,\tilde{\bxi})]+\sqrt{\frac{\mb V_{\bs\gm}[\ell(\bs x,\tilde{\bxi})]}{ n h_n^pg(\bs\gamma)(1+o(1))}\log\left(\frac{1}{\delta}\right)}
\end{aligned}
\end{equation}
with probability $1-\delta$.
% Applying union bound, we get
% \begin{equation}
% \begin{aligned}
%     \mb E_{\bs\gm}[\ell(\bs x,\tilde{\bxi})]\leq \hat{\mb E}_{\bs\gm}[\ell(\bs x,\tilde{\bxi})]+\sqrt{\frac{\mb V_{\bs\gm}[\ell(\bs x,\tilde{\bxi})]}{ n h_n^pg(\bs\gamma)(1+o(1))}\log\left(\frac{1}{\delta}\right)}
% \end{aligned}
% \end{equation}
Next, define a finite set of points $\mathcal{X}_\eta \subset \mathcal{X}$ such that for any $\bs x\in\mathcal{X}$, there exists some $\bs x'\in\mathcal{X}_\eta$ such that $\Vert \bs x - \bs x' \Vert \leq \eta.$ From \cite{SN2005}, we know that $\vert \mathcal{X}_{\eta} \vert = \mathcal{O}(1) (D/\eta)^d$. 
Since the loss function $\ell(\bs x,\tilde{\bxi})$  is $M$-Lipschitz continuous in $\bs x$, from \Cref{lem:lipschitz_loss}, we have that for any $\bs x \in \mathcal{X}$, there exists some $\bs x'\in \mathcal{X}_\eta$, such that $\lVert \bs x - \bs x'\rVert\leq \eta$ and the following condition holds:
\begin{equation}
\label{eq:lipschitz}
\begin{aligned}
\mb E_{\bs\gm}[\ell(\bs x,\tilde{\bxi})]&\leq \mb E_{\bs\gm}[\ell(\bs x',\tilde{\bxi})] + M\eta.
\end{aligned}
\end{equation}
In addition, from \Cref{cor:generalization_bound}, we note that for a fixed $\bs x' \in \mathcal{X}_\eta$
\begin{equation}
\begin{aligned}
\mb E_{\bs\gm}[\ell(\bs x',\tilde{\bxi})]&\leq \hat{\mb E}_{\bs\gm}[\ell(\bs x',\tilde{\bxi})]+\sqrt{\frac{\mb V_{\bs\gm}[\ell(\bs x',\tilde{\bxi})]}{ n h_n^pg(\bs\gamma)(1+o(1))}\log\left(\frac{1}{\delta}\right)}
\end{aligned}
\end{equation}
with probability at least $1-\delta$. Applying union bound, we get that for all $\bs x' \in \mc{X}_\eta$
\begin{equation}
\label{eq:lipschitz_union}
\begin{aligned}
\mb E_{\bs\gm}[\ell(\bs x',\tilde{\bxi})]&\leq \hat{\mb E}_{\bs\gm}[\ell(\bs x',\tilde{\bxi})]+\sqrt{\frac{\mb V_{\bs\gm}[\ell(\bs x',\tilde{\bxi})]}{ n h_n^pg(\bs\gamma)(1+o(1))}\log\left(\frac{\lvert \mathcal X_\eta \rvert}{\delta}\right)}
\end{aligned}
\end{equation}
with probability at least $1-\delta$. Combining the bounds in \eqref{eq:lipschitz} and \eqref{eq:lipschitz_union}, we get that for any $\bs x \in \mathcal{X}$, there exists some $\bs x'\in \mathcal{X}_\eta$, such that $\lVert \bs x - \bs x'\rVert\leq \eta$ and  
\begin{equation*}
\begin{aligned}
\mb E_{\bs\gm}[\ell(\bs x,\tilde{\bxi})]&\leq \hat{\mb E}_{\bs\gm}[\ell(\bs x',\tilde{\bxi})]+\sqrt{\frac{\mb V_{\bs\gm}[\ell(\bs x',\tilde{\bxi})]}{ n h_n^pg(\bs\gamma)(1+o(1))}\log\left(\frac{\lvert \mathcal X_\eta \rvert}{\delta}\right)}+M\eta
\end{aligned}
\end{equation*}
with probability $1-\delta$. Again, using the Lipschitz continuity of $\ell(\bs x,\bs\xi)$, from \Cref{lem:lipschitz_loss}, we get
\begin{equation*}
\begin{aligned}
\mb E_{\bs\gm}[\ell(\bs x,\tilde{\bxi})]&\leq \hat{\mb E}_{\bs\gm}[\ell(\bs x,\tilde{\bxi})]+\sqrt{\frac{\mb V_{\bs\gm}[\ell(\bs x,\tilde{\bxi})]}{ n h_n^pg(\bs\gamma)(1+o(1))}\log\left(\frac{\lvert \mathcal X_\eta \rvert}{\delta}\right)}+  M\eta \left(1 + \sqrt{\frac{\log\left(\frac{\lvert \mathcal X_\eta \rvert}{\delta}\right)}{ n h_n^pg(\bs\gamma)(1+o(1))}} \right)
\end{aligned}
\end{equation*}
with probability at least $1-\delta$. 
\end{proof}

\section{Proofs of  \Cref{prop:generalizationbound-highdim} and \Cref{prop:generalizationbound-highdim2}} \label{sec:proof_generalizationbound-highdim}

Before we prove the result in \Cref{prop:generalizationbound-highdim}, we first present below the statement of the Davis-Kahan Theorem and some useful results about sub-gaussian random vectors.

\begin{theorem}[Davis-Kahan Theorem (Theorem 2 in \citep{yu2014useful})] \label{thm:Davis-Kahan}
Let $\bm \Sigma,\hat{\bm \Sigma}\in\mathbb{R}^{p\times p}$ be symmetric with eigenvalues $\lambda_1\geq\ldots\geq\lambda_p$ and $\hat{\lambda}_1\geq\ldots\geq\hat{\lambda}_p$, respectively. Fix $1\leq s \leq r \leq p$ and assume that min$(\lambda_{s-1} - \lambda_s, \lambda_{r}-\lambda_{r+1})>0$, where $\lambda_0:=\infty$ and $\lambda_{p+1}:=-\infty$. Let $d=r-s+1$, and let $\bmath U=[\bmath u_s,\bmath u_{s+1},\ldots,\bmath u_r] \in\mathbb{R}^{p\times d}$ and $ \hat{\bmath U}=[\hat{\bmath u}_s,\hat{\bmath u}_{s+1},\ldots,\hat{\bmath u}_r]\in\mathbb{R}^{p\times d}$ have orthonormal columns satisfying $\bm \Sigma \bmath u_j=\lambda_j \bmath u_j$ and $\hat{\bm \Sigma} \hat{\bmath u}_j=\hat{\lambda}_j \hat{\bmath u}_j$ for $j=s,s+1,\ldots,r$. Then, there exists an orthogonal matrix ${\bmath O}\in\mathbb{R}^{d\times d}$ such that
\begin{equation}
\lVert \bmath U{\bmath O} -\hat{\bmath U} \rVert_{\fr} \leq \frac{2^{3/2}d^{1/2}\lVert\hat{\bm \Sigma}-\bm \Sigma\rVert_{2}}{\min(\lambda_{s-1}-\lambda_{s},\lambda_r-\lambda_{r+1})}.
\end{equation}
\end{theorem}

\begin{lemma}[Covariance Estimation for Sub-Gaussian distributions (Corollary 5.50 in \citep{vershynin2010introduction})]
\label{lem:CovarianceMatrixSubGaussian}
Consider a sub-gaussian probability distribution in $\mb R^p$ with true covariance matrix $\bm \Sigma$ and sample covariance matrix $\hat{\bm \Sigma}$ constructed from $n$ i.i.d.~observations. Let~$\epsilon \in (0,1)$ and $t\geq 1$. Then, with probability at least $1-2\exp(-t^2 p)$, we have
\begin{equation}
\lVert \hat{\bm \Sigma} - \bm \Sigma\rVert_2 \leq \epsilon \ \text{\rm provided} \ n\geq C\bigg(\frac{t}{\epsilon}\bigg)^2 p.
\end{equation}
Here, $C$ is a constant that depends only on the sub-gaussian parameter $\sigma$ for the distribution.
\end{lemma}

\begin{lemma}[Theorem 1.19 in \citep{rigollet2015high}]\label{lem:SubGaussianLength}
Let $\tilde{\bm \gamma}\in \mb R^d$ be a sub-gaussian random vector with sub-gaussian parameter $\sigma$. Then, with probability $1-\delta$, we have 
\begin{equation*}
\| \tilde{\bm \gamma} \| \leq 4\sigma\sqrt{d}+2\sigma\sqrt{2\log\bigg(\frac{1}{\delta}\bigg)}
\end{equation*}
for some $\delta\in(0,1)$.  
\end{lemma}

\begin{lemma}\label{lem:rho_bound}
Consider a sub-gaussian random vector~$\tilde{\bm \gamma} \in \mb R^p$ with sub-gaussian parameter $\sigma$. Let $\bm \Sigma,\hat{\bm \Sigma}\in\mb R^{p\times p}$ denote respectively the true covariance matrix for $\tilde{\bm \gamma}$ and the sample covariance matrix estimated from $n$ i.i.d.~observations of $\tilde{\bm \gamma}$. Let $\bmath U$ and $\hat{\bmath U}$ be the matrices whose columns comprise the top $p'$ eigenvectors of these covariance matrices. Suppose $\tilde{\bm\rho},\hat{\bm\rho} \in\mb R^{p'}$ denote respectively the true and estimated projections of $\tilde{\bm \gamma}\in \mb R^p$ onto the subspace spanned by the columns of $\bmath U$ and $\hat{\bmath U}$. Then, with probability at least $1- 2\delta $, we have under some basis coordinate system
\begin{equation}\label{eq:lemma_rho_bound}
\lVert \tilde{\bs\rho}-\hat{\bs\rho} \rVert_2 \leq  \frac{C}{\lambda_{p'} - \lambda_{p'+1}} \sqrt{\frac{p'}{n} \log \left(\frac{2}{\delta}\right)} \gamma_{\max},
\end{equation}
where $C>0$ is a constant that depends on the sub-gaussian parameter $\sigma$, $\lambda_{p'}$ is the $p'$-th largest eigenvalue of the true covariance matrix $\bs \Sigma$, and $\gm_{\max}:=4\sigma\left(\sqrt{p}+\sqrt{\frac{1}{2}\log \left(\frac{1}{\delta}\right)}\right)$.
\end{lemma}
\begin{proof}
Based on the definition of $\tilde{\bm\rho}$ and $\hat{\bm \rho}$, we have
\begin{equation}
\label{eq:rho_bound}
\begin{aligned}
        \lVert \tilde{\bm\rho}-\hat{\bm\rho} \rVert_2 &= \lVert (\bmath U{\bmath O}-\hat{\bmath U})\tilde{\bm \gamma}\rVert_2\\
        &\leq \| \bmath U{\bmath O}-\hat{\bmath U}\|_2 \|\tilde{\bm \gamma}\|_2\\
        &\leq \| \bmath U{\bmath O}-\hat{\bmath U}\|_{\fr} \|\tilde{\bm \gamma}\|_2 \\
        &\stackrel{(i)}{\leq} \frac{2^{3/2}p'{}^{1/2}}{\lambda_{p'} - \lambda_{p'+1}} \lVert\hat{\bm \Sigma}-\bm \Sigma\rVert_{2} \|\tilde{\bm \gamma}\|_2 \\
        &\stackrel{(ii)}{\leq}  \frac{2^{3/2} p'{}^{1/2}}{\lambda_{p'} - \lambda_{p'+1}}  \sqrt{\frac{C'}{n} \log\left(\frac{2}{\delta}\right)} \|\tilde{\bm \gamma}\|_2  ,
\end{aligned}
\end{equation}
 with probability at least $1-\delta$. Here, ${\bmath O}$ is an orthogonal (change-of-basis) matrix and $C'>0$ is a constant that depends on the sub-gaussian parameter~$\sigma$. Inequality $(i)$ follows from the application of the Davis-Kahan Theorem whose statement is detailed in \Cref{thm:Davis-Kahan}. We obtain inequality $(ii)$ by noting that $\tilde{\bs \gamma}$ is a sub-gaussian random vector and, therefore,~\Cref{lem:CovarianceMatrixSubGaussian} applies to our setting. Putting~ $\epsilon=\sqrt{\frac{C'}{n} \log\left(\frac{2}{\delta}\right)}$ and $t=\sqrt{ \log\left(\frac{2}{\delta}\right)\frac{1}{p}} $ in~\Cref{lem:CovarianceMatrixSubGaussian}, we get that 
$\lVert\bs\Sigma-\hat{\bs\Sigma}\rVert_{2}\leq \sqrt{\frac{C'}{n} \log\left(\frac{2}{\delta}\right)}$ with probability at least $1-\delta$.

Next, we define $\gm_{\max}:=4\sigma\left(\sqrt{p}+\sqrt{\frac{1}{2}\log \left(\frac{1}{\delta}\right)}\right)$. From~\Cref{lem:SubGaussianLength}, we get $\|\tilde{\bs \gamma}\|\leq \gm_{\max}$ with probability at least $1-\delta$. Therefore, by applying union bound to \eqref{eq:rho_bound}, we obtain that with probability at least 
 $1-2\delta$,
\begin{equation*}
\lVert \bs\rho-\hat{\bs\rho} \rVert_2 \leq  \frac{C}{\lambda_{p'} - \lambda_{p'+1}} \sqrt{\frac{p'}{n} \log \left(\frac{2}{\delta}\right)} \gamma_{\max}.
\end{equation*}
Substituting $C=2^{3/2}\sqrt{C'}$, we obtain the desired result.
\end{proof}

\begin{proof}[Proof of \Cref{prop:generalizationbound-highdim}]
\noindent To make the dependence of the bandwidth $h$ explicit, we define the kernel function as %$k_h(x)$ for any scalar value $x>0$ as
\begin{equation*}
 k_h(x)=\frac{1}{Z} \exp\left(-\frac{x}{h}\right),
\end{equation*}
where $Z$ is the normalization constant. As before, for any generic $\bs \gamma \in \mb R^{p}$, we denote its projection in the low-dimensional space as $\bs\rho \in \mb R^{p'}$. Similarly, for each $\bs\gm^i \in \mb R^{p}$, its projection is expressed as $\bs\rho^i \in \mb R^{p'}$. Since the true subspace is not known, we estimate the projection matrix by $\hat{\bmath U}$ obtained using principal component analysis (PCA), and denote the estimated projections of $\bs\rho$ and $\bs \rho^i$  by $\hat{\bs\rho}$ and $ \hat{\bs\rho}^i$ respectively. 

Let $\eta_i=\|\bs\rho-\bs\rho^i\|$ and $\eta_i'=\|\hat{\bs\rho}-\hat{\bs\rho}^i\|$. Using the result from~\Cref{lem:rho_bound} and applying union bound, we get that with probability at least $1-2(n_1+1)\delta$, $\|\bs\rho-\hat{\bs\rho}\|\leq \tau$ and $\|\bs\rho^i-\hat{\bs\rho}^i\|\leq \tau$ for all $i\in 
\mc I_1$, where we set $\tau=\frac{C}{\lambda_{p'} - \lambda_{p'+1}} \sqrt{\frac{p'}{n_2} \log \left(\frac{2}{\delta}\right)} \gamma_{\max}$ and $\gm_{\max}=4\sigma\left(\sqrt{p}+\sqrt{\frac{1}{2}\log \left(\frac{1}{\delta}\right)}\right)$ in the bound obtained from \Cref{eq:lemma_rho_bound}. Therefore, by using reverse triangle and triangle inequalities, we get that $|\eta_i-\eta_i'|\leq \| (\bs\rho-\bs\rho^i)- (\hat{\bs\rho}-\hat{\bs\rho}^i)\| \leq \| \bs\rho-\hat{\bs\rho}\|+\| \bs\rho^i-\hat{\bs\rho}^i\|  \leq 2\tau$ with probability at least $1-2 (n_1+1) \delta$. 

We are now in a position to obtain a bound for $\lvert \EE_{\bs\rho}[\ell(\bs x,\xit)]-\hat{\EE}_{\hat{\bs\rho}}[\ell(\bs x,\xit)] \rvert$ for a fixed $\bs x \in \mc X$. From triangle inequality, we first note that 
\begin{equation}\label{eq:rho_generalization_bound}
\lvert \EE_{\bs\rho}[\ell(\bs x,\xit)]-\hat{\EE}_{\hat{\bs\rho}}[\ell(\bs x,\xit)] \rvert \leq    \lvert \EE_{\bs\rho}[\ell(\bs x,\xit)]-\hat{\EE}_{\bs\rho}[\ell(\bs x,\xit)] \rvert + \lvert \hat{\EE}_{\bs\rho}[\ell(\bs x,\xit)]-\hat{\EE}_{\hat{\bs\rho}}[\ell(\bs x,\xit)] \rvert.
\end{equation}
Next, we obtain high probability bounds for each term in the right hand side of the above expression. By applying \Cref{cor:generalization_bound}, we get that with probability at least $1-\delta$, the first term is upper-bounded as
\begin{equation*}
\lvert \EE_{\bs\rho}[\ell(\bs x,\xit)]-\hat{\EE}_{\hat{\bs\rho}}[\ell(\bs x,\xit)] \rvert \leq \sqrt{ \frac{ \mathbb{V}_{\bm\rho}[\ell(\bs x,\xit)] }{ n_1 h_{n_1}^{p'}g(\bs\rho)(1+o(1))}\log \left(\frac{1}{\delta} \right)}% =\sqrt{ \frac{ \VVrho[\ell(\bs x,\xit)] }{ n_1 h_{n_1}^{p'}g(\bs\gamma)(1+o(1))}\log \left(\frac{1}{\delta} \right)} 
\end{equation*}
%{\color{red} (Not 100\% sure about the above equality...but it seems correct?)} 
Next, we consider the second term in \Cref{eq:rho_generalization_bound}, which gives us
\begin{equation}\label{eq:PCA_bound}
\begin{array}{rl}
\lvert \hat{\EE}_{\bs\rho}[\ell(\bs x,\xit)]-\hat{\EE}_{\hat{\bs\rho}}[\ell(\bs x,\xit)] \rvert =&\displaystyle \left\vert\sum_{i \in \mc I_1}\frac{  k_{h}(\eta_i)}{\sum_{j \in \mc I_1} k_{h}(\eta_j)} \ell(\bs x,\bs\xi^i) - \sum_{i \in \mc I_1}\frac{  k_{h}(\eta_i')}{\sum_{j \in \mc I_1} k_{h}(\eta_j')} \ell(\bs x,\bs\xi^i)\right\vert\\
\leq &\displaystyle \left( \max_i \left\vert \ell(\bs x,\bs\xi^i)\right\vert\right) \sum_{i \in \mc I_1}\left\vert\frac{ k_{h}(\eta_i)}{\sum_{j \in \mc I_1} k_{h}(\eta_j)} - \frac{ k_{h}(\eta_i')}{\sum_{j \in \mc I_1} k_{h}(\eta_j')} \right\vert\\
\leq &\displaystyle\sum_{i \in \mc I_1}\left\vert\frac{ k_{h}(\eta_i)}{\sum_{j \in \mc I_1} k_{h}(\eta_j)} - \frac{ k_{h}(\eta_i')}{\sum_{j \in \mc I_1} k_{h}(\eta_j')} \right\vert .
\end{array}
\end{equation}
Here, the last inequality follows from the Assumption \ref{as3} that $\ell(\bs x,\bs\xi)$ takes values between 0~and~1 for all $\bs x \in \mc X$ and $\bs \xi \in \Xi$. Next, we obtain a bound for each term within the summation below.
\begin{equation}\label{eq:bound_summand}
\begin{aligned}
&\displaystyle\left\vert\frac{ k_{h}(\eta_i)}{\sum_{j \in \mc I_1} k_{h}(\eta_j)} - \frac{ k_{h}(\eta_i')}{\sum_{j \in \mc I_1} k_{h}(\eta_j')} \right\vert 
\\
=&\max\bigg\{ \frac{ k_{h}(\eta_i)}{\sum_{j \in \mc I_1} k_{h}(\eta_j)}- \frac{ k_{h}(\eta_i')}{\sum_{j \in \mc I_1} k_{h}(\eta_j')},  \frac{ k_{h}(\eta_i')}{\sum_{j \in \mc I_1} k_{h}(\eta_j')}- \frac{ k_{h}(\eta_i)}{\sum_{j \in \mc I_1} k_{h}(\eta_j)} \bigg\} 
\\
\leq &\max\bigg\{ \frac{ k_{h}(\eta_i)}{\sum_{j \in \mc I_1} k_{h}(\eta_j)}- \frac{ k_{h}(\eta_i+2\tau)}{\sum_{j \in \mc I_1} k_{h}(\eta_j-2\tau)},  \frac{ k_{h}(\eta_i-2\tau)}{\sum_{j \in \mc I_1} k_{h}(\eta_j+2\tau)}- \frac{ k_{h}(\eta_i)}{\sum_{j \in \mc I_1} k_{h}(\eta_j)} \bigg\} 
\\ 
= &\displaystyle\max\bigg\{ \frac{\exp(-\eta_i/h)}{\sum_{j \in \mc I_1}\mathcal\exp(-\eta_j/h)}- \frac{\exp(-(\eta_i+2\tau)/h)}{\sum_{j \in \mc I_1}\exp(-(\eta_j-2\tau)/h)},
\\ 
& \quad \quad \quad \quad \quad \quad \quad  \quad \quad \quad \ \frac{\exp(-(\eta_i-2\tau)/h)}{\sum_{j \in \mc I_1} \exp(-(\eta_j+2\tau)/h)}- \frac{\exp(-\eta_i/h)}{\sum_{j \in \mc I_1}\mathcal\exp(-\eta_j/h)} \bigg\} \\ 
\stackrel{(i)}{=} &\frac{\exp(-\eta_i/h)}{\sum_{j \in \mc I_1}\mathcal\exp(-\eta_j/h)} \max\big\{1-\exp(-4\tau/h), \exp(4\tau/h)-1 \big\}
\\
\leq &\frac{\exp(-\eta_i/h)}{\sum_{j \in \mc I_1}\mathcal\exp(-\eta_j/h)} (\exp(4\tau/h)-1) .
\end{aligned}
\end{equation}
Here, %{\color{red} $k_{h}(\eta_i-2\tau)$ is well defined as $\eta_i-2\tau > 0$ for sufficiently large $n_2$}, and 
we obtain equality $(i)$ by considering each of the two terms within the max operator in the previous expression separately. We first evaluate the first term
\begin{equation*}
\frac{\exp(-\eta_i/h)}{\sum_{j \in \mc I_1}\mathcal\exp(-\eta_j/h)}- \frac{\exp(-(\eta_i+2\tau)/h)}{\sum_{j \in \mc I_1}\exp(-(\eta_j-2\tau)/h)} = \frac{\exp(-\eta_i/h)}{\sum_{j \in \mc I_1}\mathcal\exp(-\eta_j/h)} (1-\exp(-4\tau/h)).
\end{equation*}
Next, we consider the second term
\begin{equation*}
\frac{\exp(-(\eta_i-2\tau)/h)}{\sum_{j \in \mc I_1} \exp(-(\eta_j+2\tau)/h)}- \frac{\exp(-\eta_i/h)}{\sum_{j \in \mc I_1}\mathcal\exp(-\eta_j/h)} =  \frac{\exp(-\eta_i/h)}{\sum_{j \in \mc I_1} \exp(-\eta_j/h)}  (\exp(4\tau/h)-1).
\end{equation*}
Combining the results from \eqref{eq:PCA_bound} and \eqref{eq:bound_summand}, we get 
\begin{equation*}
\begin{aligned}
\displaystyle \left\vert\sum_{i \in \mc I_1}\frac{ k_{h}(\eta_i)}{\sum_{j \in \mc I_1} k_{h}(\eta_j)} \ell(\bs x,\bs\xi^i) - \sum_{i \in \mc I_1}\frac{ k_{h}(\eta_i')}{\sum_{j \in \mc I_1} k_{h}(\eta_j')} \ell(\bs x,\bs\xi^i)\right\vert&\leq \exp(4\tau/h)-1 \\
&\stackrel{(i)}{\leq}\frac{4\tau}{h} + \bigg(\frac{4\tau}{h}\bigg)^2 \\ 
&\stackrel{(ii)}{\leq}\frac{8\tau}{h} .
\end{aligned}
\end{equation*}
Note that $h$ is scaled with $n_1$ such that $n_2^{-1/2} / h < 1$. Thus,  $4\tau/h < 1$ for sufficiently large $n_1$ and $n_2$, and inequalities $(i)$ and $(ii)$ then follow from the fact that $e^x \leq 1+ x +x^2$ and $x^2\leq x$ for $x \leq 1$. Therefore, we have that with probability at least $1-2(n_1+1)\delta-\delta \geq 1-5 n_1\delta $, for a fixed $\bs x \in \mc X$, we have
\begin{equation*}
\begin{aligned}
\lvert \EE_{\bs\rho}[\ell(\bs x,\xit)]-\hat{\EE}_{\hat{\bs\rho}}[\ell(\bs x,\xit)] \rvert
&  \leq \sqrt{ \frac{ \mathbb{V}_{\bm\rho}[\ell(\bs x,\xit)] }{ n_1 h_{n_1}^{p'}g(\bs\rho)(1+o(1))}\log \left(\frac{1}{\delta} \right)}+\frac{8\tau}{h}\\
&  = \sqrt{ \frac{ \mathbb{V}_{\bm\rho}[\ell(\bs x,\xit)] }{ n_1 h_{n_1}^{p'}g(\bs\rho)(1+o(1))}\log \left(\frac{1}{\delta} \right)}+\frac{8}{h} \frac{C}{\lambda_{p'} - \lambda_{p'+1}} \sqrt{\frac{p'}{n_2} \log \left(\frac{2}{\delta}\right)} \gamma_{\max}.
\end{aligned}
\end{equation*}
\noindent Therefore, by applying union bound, we get that with probability at least $1-5 n_1\delta$, %, for any $\bs x \in \mc X$
\begin{equation*}
\begin{aligned}
    \lvert \EE_{\bs\rho}[\ell(\bs x,\xit)]-\hat{\EE}_{\hat{\bs\rho}}[\ell(\bs x,\xit)] \rvert
    &  \leq  \sqrt{ \frac{ \mathbb{V}_{\bm\rho}[\ell(\bs x,\xit)] }{ n_1 h_{n_1}^{p'}g(\bs\rho)(1+o(1))}\log \left(\frac{\lvert \mc X \rvert}{\delta} \right)}\\
    &\quad+\frac{8}{h} \frac{4\sigma C}{\lambda_{p'} - \lambda_{p'+1}} \sqrt{\frac{p'}{n_2} \log \left(\frac{2\lvert \mc X \rvert}{\delta}\right)} \left(\sqrt{p}+\sqrt{\frac{1}{2}\log \left(\frac{\lvert \mc X \rvert}{\delta}\right)}\right)\qquad\forall \bs x \in \mc X.
\end{aligned}
\end{equation*}
The result then follows by performing the change of variable $\delta\leftarrow 5n_1\delta$, and by noting that  $\EE_{\bs\gamma}[\ell(\bs x,\xit)]=\EE_{\bs\rho}[\ell(\bs x,\xit)]$ and $\mathbb{V}_{\bm\gm}[\ell(\bs x,\xit)] =\mathbb{V}_{\bm\rho}[\ell(\bs x,\xit)]$.
\end{proof}

\begin{proof}[Proof of \Cref{prop:generalizationbound-highdim2} {[With bounded $\bm\gamma$]}]
\noindent We consider the same setup as in the proof of \Cref{prop:generalizationbound-highdim}, where we define the kernel function as
\begin{equation*}
 k_h(x)=\frac{1}{Z} \exp\left(-\frac{x}{h}\right),
\end{equation*}
with $Z$ a normalization constant. As before, for any generic $\bs \gamma \in \mb R^{p}$, we denote its projection on the low-dimensional space as $\bs\rho \in \mb R^{p'}$. Similarly, for each $\bs\gm^i \in \mb R^{p}$, its projection is expressed as $\bs\rho^i \in \mb R^{p'}$. Since the true subspace is not known, we estimate the projection matrix by $\hat{\bmath U}$ obtained using PCA, and denote the estimated projections of $\bs\rho$ and $\bs \rho^i$  by $\hat{\bs\rho}$ and $ \hat{\bs\rho}^i$, respectively. 

Let $\eta_i=\|\bs\rho-\bs\rho^i\|$ and $\eta_i'=\|\hat{\bs\rho}-\hat{\bs\rho}^i\|$. Using the result from the proof of~\Cref{lem:rho_bound}, if $\Vert\tilde{\bm{\gamma}}\Vert \leq \gamma_{\max}$ almost surely, then we get that with probability at least $1-\delta$, $\|\bs\rho-\hat{\bs\rho}\|\leq \tau$ and $\|\bs\rho^i-\hat{\bs\rho}^i\|\leq \tau$ for all $i\in \mc I_1$, where we set $\tau=\frac{C}{\lambda_{p'} - \lambda_{p'+1}} \sqrt{\frac{p'}{n_2} \log \left(\frac{2}{\delta}\right)} \gamma_{\max}$ in the bound obtained from \Cref{eq:lemma_rho_bound}. 

The rest of this proof proceeds as that of \Cref{prop:generalizationbound-highdim}. In particular, from triangle inequality, we first note that 
\begin{equation*}
\lvert \EE_{\bs\rho}[\ell(\bs x,\xit)]-\hat{\EE}_{\hat{\bs\rho}}[\ell(\bs x,\xit)] \rvert \leq    \lvert \EE_{\bs\rho}[\ell(\bs x,\xit)]-\hat{\EE}_{\bs\rho}[\ell(\bs x,\xit)] \rvert + \lvert \hat{\EE}_{\bs\rho}[\ell(\bs x,\xit)]-\hat{\EE}_{\hat{\bs\rho}}[\ell(\bs x,\xit)] \rvert,
\end{equation*}
where the first term in the right hand side is upper-bounded as
\begin{equation*}
\lvert \EE_{\bs\rho}[\ell(\bs x,\xit)]-\hat{\EE}_{\hat{\bs\rho}}[\ell(\bs x,\xit)] \rvert \leq \sqrt{ \frac{ \mathbb{V}_{\bm\rho}[\ell(\bs x,\xit)] }{ n_1 h_{n_1}^{p'}g(\bs\rho)(1+o(1))}\log \left(\frac{1}{\delta} \right)}% =\sqrt{ \frac{ \VVrho[\ell(\bs x,\xit)] }{ n_1 h_{n_1}^{p'}g(\bs\gamma)(1+o(1))}\log \left(\frac{1}{\delta} \right)} 
\end{equation*}
with probability at least $1-\delta$. %{\color{red} (Not 100\% sure about the above equality...but it seems correct?)} 
The second term gives us
\begin{equation*}
\lvert \hat{\EE}_{\bs\rho}[\ell(\bs x,\xit)]-\hat{\EE}_{\hat{\bs\rho}}[\ell(\bs x,\xit)] \rvert \leq \frac{8\tau}{h} 
\end{equation*}
Therefore, we have that with probability at least $1-\delta-\delta = 1-2\delta $, for a fixed $\bs x \in \mc X$, we have
\begin{equation*}
\begin{aligned}
\lvert \EE_{\bs\rho}[\ell(\bs x,\xit)]-\hat{\EE}_{\hat{\bs\rho}}[\ell(\bs x,\xit)] \rvert
&  \leq \sqrt{ \frac{ \mathbb{V}_{\bm\rho}[\ell(\bs x,\xit)] }{ n_1 h_{n_1}^{p'}g(\bs\rho)(1+o(1))}\log \left(\frac{1}{\delta} \right)}+\frac{8\tau}{h}\\
&  = \sqrt{ \frac{ \mathbb{V}_{\bm\rho}[\ell(\bs x,\xit)] }{ n_1 h_{n_1}^{p'}g(\bs\rho)(1+o(1))}\log \left(\frac{1}{\delta} \right)}+\frac{8}{h} \frac{C}{\lambda_{p'} - \lambda_{p'+1}} \sqrt{\frac{p'}{n_2} \log \left(\frac{2}{\delta}\right)}
\gamma_{\max}.
\end{aligned}
\end{equation*}
\noindent By applying union bound, we get that with probability at least $1-\delta$, 
\begin{equation*}
\begin{aligned}
\lvert \EE_{\bs\rho}[\ell(\bs x,\xit)]-\hat{\EE}_{\hat{\bs\rho}}[\ell(\bs x,\xit)] \rvert&\leq  \sqrt{ \frac{ \mathbb{V}_{\bm\rho}[\ell(\bs x,\xit)] }{ n_1 h_{n_1}^{p'}g(\bs\rho)(1+o(1))}\log \left(\frac{2\lvert \mc X \rvert}{\delta} \right)} \\
&\quad+ \frac{8}{h} \frac{C}{\lambda_{p'} - \lambda_{p'+1}} \sqrt{\frac{p'}{n_2} \log \left(\frac{4\lvert \mc X \rvert}{\delta}\right)} \gamma_{\max}\qquad  \forall\bs x \in \mc X.
\end{aligned}
\end{equation*}
The result then follows  by noting that  $\EE_{\bs\gamma}[\ell(\bs x,\xit)]=\EE_{\bs\rho}[\ell(\bs x,\xit)]$ and $\mathbb{V}_{\bm\gm}[\ell(\bs x,\xit)] =\mathbb{V}_{\bm\rho}[\ell(\bs x,\xit)]$.
\end{proof}

\section{Proof of Proposition \ref{prop:bound_sample_stddev}}
\label{app:bound_sample_stddev}
To prove Proposition \ref{prop:bound_sample_stddev}, we rely on the following lemma. 
\begin{lemma}\label{lem:diff_expect_emp}
%Consider a function $L(\xit):\Xi\rightarrow[0,1]$ that satisfies all conditions in Theorem \ref{thm:MDP}. 
For any fixed $\bm x\in\mathcal X$ and $t\in[0,1]$, we have
\begin{equation}
\label{eq:bound_sample_stddev_t}
\displaystyle\left|\sqrt{ \EErho[(\ell (\bm x,\xit)-t)^2]}-\sqrt{\EErhohat[(\ell (\bm x,\xit)-t)^2]}\right|\leq 
\sqrt{ \frac{\log\left(\frac{1}{\delta}\right)}{ n h_n^pg(\bm\gamma)(1+o(1))}},
\end{equation}
with probability at least $1-\delta$.
\end{lemma}
\begin{proof}
By applying Theorem \ref{thm:generalization_bound} to the input function $(\ell (\bm x,\xit)-t)^2$, which also satisfies all conditions in the theorem, we obtain that with a probability at least $1-\delta$
\begin{equation*}
\begin{array}{rl}
\displaystyle\left|\EErho[(\ell (\bm x,\xit)-t)^2]-\EErhohat[(\ell (\bm x,\xit)-t)^2]\right|&\displaystyle\leq \sqrt{\frac{\VVrho[(\ell (\bm x,\xit)-t)^2]}{ n h_n^pg(\bm\gamma)(1+o(1))}\log\left(\frac{1}{\delta}\right)}\\
&\displaystyle\leq \sqrt{\frac{\EErho[(\ell (\bm x,\xit)-t)^2]}{ n h_n^pg(\bm\gamma)(1+o(1))}\log\left(\frac{1}{\delta}\right)}.
\end{array}
\end{equation*}
%where $s_1(n)$ is a function in $n$ that is of $o(1)$. 
Here, the last inequality follows from 
\begin{equation*}
\begin{array}{rl}
\VVrho[(\ell (\xit)-t)^2]&= \EErho[(\ell (\xit)-t)^4]-\EErho[(\ell (\xit)-t)^2]^2\\
&\leq \EErho[(\ell (\xit)-t)^4]\\
&\leq \EErho[(\ell (\xit)-t)^2],
\end{array}
\end{equation*}
where the final inequality holds because the random variable $(\ell(\bm x,\xit)-t)^2$ is supported on a subset of $[0,1]$. Next, expanding the absolute value term yields the following two cases:
\begin{equation}
\label{eq:two_sided_sample_stddev_t}
\begin{array}{l}
\displaystyle \EErho[(\ell(\bm x,\xit)-t)^2]-\EErhohat[(\ell(\bm x,\xit)-t)^2]\leq \sqrt{\frac{\EErho[(\ell(\bm x,\xit)-t)^2]}{ n h_n^pg(\bm\gamma)(1+o(1))}\log\left(\frac{1}{\delta}\right)}\quad\textup{ and }\\
\displaystyle\EErhohat[(\ell(\bm x,\xit)-t)^2]-\EErho[(\ell(\bm x,\xit)-t)^2]\leq \sqrt{\frac{\EErho[(\ell(\bm x,\xit)-t)^2]}{ n h_n^pg(\bm\gamma)(1+o(1))}\log\left(\frac{1}{\delta}\right)}. 
\end{array}
\end{equation}
From the first case, we obtain
\begin{equation*}
\begin{array}{ll}
&\displaystyle \EErho[(\ell(\bm x,\xit)-t)^2]-\sqrt{\frac{\EErho[(\ell(\bm x,\xit)-t)^2]}{ n h_n^pg(\bm\gamma)(1+o(1))}\log\left(\frac{1}{\delta}\right)}\leq \EErhohat[(\ell(\bm x,\xit)-t)^2],\\
\end{array}
\end{equation*}
which is equivalent to
\begin{equation*}
\displaystyle\left(\sqrt{ \EErho[(\ell(\bm x,\xit)-t)^2]}-\frac{1}{2}\sqrt{\frac{\log\left(\frac{1}{\delta}\right)}{ n h_n^pg(\bm\gamma)(1+o(1))}}\right)^2\leq \frac{1}{4}\frac{\log\left(\frac{1}{\delta}\right)}{ n h_n^pg(\bm\gamma)(1+o(1))}+ \EErhohat[(\ell(\bm x,\xit)-t)^2].
\end{equation*}
Taking square root on both sides then yields
\begin{equation}
\label{eq:side_1_of_sample_stddev_t}
\begin{array}{ll}
\displaystyle\sqrt{ \EErho[(\ell(\bm x,\xit)-t)^2]}-\frac{1}{2}\sqrt{\frac{\log\left(\frac{1}{\delta}\right)}{ n h_n^pg(\bm\gamma) (1+o(1))}}&\displaystyle\leq \sqrt{ \frac{1}{4}\frac{\log\left(\frac{1}{\delta}\right)}{ n h_n^pg(\bm\gamma)(1+o(1))}+ \EErhohat[(\ell(\bm x,\xit)-t)^2]}\\
&\displaystyle\leq \frac{1}{2}\sqrt{ \frac{\log\left(\frac{1}{\delta}\right)}{ n h_n^pg(\bm\gamma)(1+o(1))}}+ \sqrt{\EErhohat[(\ell(\bm x,\xit)-t)^2]},
\end{array}
\end{equation}
where the last inequality follows from the relation $\sqrt{a_1+a_2}\leq \sqrt{a_1}+\sqrt{a_2}$.  Next, the second case in~\eqref{eq:two_sided_sample_stddev_t} yields
\begin{equation*}
\begin{array}{rl}
\displaystyle \EErhohat[(\ell(\bm x,\xit)-t)^2]&\leq\displaystyle \EErho[(\ell(\bm x,\xit)-t)^2]+\sqrt{\frac{\EErho[(\ell(\bm x,\xit)-t)^2]}{ n h_n^pg(\bm\gamma)(1+o(1)))}\log\left(\frac{1}{\delta}\right)}\\
&\displaystyle=\left(\sqrt{ \EErho[(\ell(\bm x,\xit)-t)^2]}+\frac{1}{2}\sqrt{\frac{\log\left(\frac{1}{\delta}\right)}{ n h_n^pg(\bm\gamma)(1+o(1))}}\right)^2-  \frac{1}{4}\frac{\log\left(\frac{1}{\delta}\right)}{ n h_n^pg(\bm\gamma)(1+o(1))}\\
&\displaystyle\leq\left(\sqrt{ \EErho[(\ell(\bm x,\xit)-t)^2]}+\frac{1}{2}\sqrt{\frac{\log\left(\frac{1}{\delta}\right)}{ n h_n^pg(\bm\gamma)(1+o(1))}}\right)^2. \\
\end{array}
\end{equation*}
Finally, taking square root on both sides and combining with the inequality in \eqref{eq:side_1_of_sample_stddev_t}, we   conclude that  the bound in \eqref{eq:bound_sample_stddev_t} indeed holds. This completes the proof. 
\end{proof}
\noindent Using this lemma, we prove the bound of the error introduced by the empirical conditional standard deviation.

\begin{proof}[Proof of Proposition \ref{prop:bound_sample_stddev}]
We first show that the function $\sqrt{ \EErho[(\ell(\bm x,\xit)-t)^2]}$ is Lipschitz continuous in $t$ with  constant~$1$. Indeed, by the reverse triangle inequality, we have
\begin{equation}
\begin{array}{rl}
\displaystyle \left|\sqrt{ \EErho[(\ell(\bm x,\xit)-t)^2]}-\sqrt{ \EErho[(\ell(\bm x,\xit)-t')^2]}\right|&\leq\displaystyle\sqrt{ \EErho[(\ell(\bm x,\xit)-t-\ell(\bm x,\xit)+t')^2]}=|t-t'|,
\end{array}
\end{equation}
where the inequality holds because the function $\sqrt{ \EErho[(\cdot)^2]}$ constitutes a semi-norm. One can similarly show that the function $\sqrt{ \EErhohat[(\ell(\bm x,\xit)-t)^2]}$ is Lipschitz continuous in $t$ with constant $1$. We next observe that 
\begin{equation*}\label{eq:optimization_for_std_dev}
\sqrt{\VVrho[\ell(\bm x,\xit)]}=\min_{t\in[0,1]}\sqrt{ \EErho[(\ell(\bm x,\xit)-t)^2]}\quad\textup{ and }\quad\sqrt{\VVrhohat[\ell(\bm x,\xit)]}=\min_{t\in[0,1]}\sqrt{ \EErhohat[(\ell(\bm x,\xit)-t)^2]}, 
\end{equation*} 
which follows from the fact that the minimizers of these scalar optimization problems are respectively given by the mean $\EErho[(\ell(\bm x,\xit)]$ and the empirical mean $\EErhohat[(\ell(\bm x,\xit)]$. Consider now a finite subset $\mathcal T=\{0,\tau /2,\tau, 3\tau /2 , \dots ,1\}$ of $[0,1]$ with cardinality $|\mathcal T|=1+1/(\tau/2) = 1 + 2/\tau$. Let $t^\star$ and $\hat t^\star$, respectively, be the minimizers of the above optimization problems over the subset~$\mathcal T$ instead of $[0,1]$. By the Lipschitz continuity of the objective functions, we can guarantee that 
\begin{equation*}
\left|\sqrt{\VVrho[\ell(\bm x,\xit)]}-\sqrt{ \EErho[(\ell(\bm x,\xit)-t^\star)^2]}\right|\leq \tau /2
\;\;\textup{and}\;\;
\left|\sqrt{\VVrhohat[\ell(\bm x,\xit)]}-\sqrt{ \EErhohat[(\ell(\bm x,\xit)-\hat t^\star)^2]}\right|\leq \tau /2. 
\end{equation*}
Thus, to ensure that the bound $\left|\sqrt{\VVrho[\ell(\bm x,\xit)]}-\sqrt{\VVrhohat[\ell(\bm x,\xit)]}\right|\leq\epsilon$ holds, we require the sufficient condition 
\begin{equation*}
\left|\sqrt{ \EErho[(\ell(\bm x,\xit)-t^\star)^2]}-\sqrt{ \EErhohat[(\ell(\bm x,\xit)-\hat t^\star)^2]}\right|\leq \epsilon-\tau. 
\end{equation*}
Note that the left-hand side expression is upper bounded by the largest error
\begin{equation*}
\begin{array}{rl}
\displaystyle\max_{t\in\mathcal T}\left|\sqrt{ \EErho[(\ell(\bm x,\xit)- t)^2]}-\sqrt{ \EErhohat[(\ell(\bm x,\xit)- t)^2]}\right|. 
\end{array}
\end{equation*}
Thus, applying the union bound to \eqref{eq:bound_sample_stddev_t} over $t\in\mathcal T$ yields an upper bound on left-hand side expression, as follows
\begin{equation*}
\displaystyle\left|\sqrt{ \EErho[(\ell(\bm x,\xit)-t^\star)^2]}-\sqrt{ \EErhohat[(\ell(\bm x,\xit)-\hat t^\star)^2]}\right|\leq 
\sqrt{ \frac{\log\left(\frac{|\mathcal T|}{\delta}\right)}{ n h_n^pg(\bm\gamma)(1+o(1))}}.
\end{equation*}
The result then follows by equating the right hand side with $\epsilon-\tau$. 
\end{proof}

\section{Proof of \Cref{thm:optimality_bound}}
\label{sec:thm_optimality_bound}
Using the result in Proposition \ref{prop:bound_sample_stddev}, we first obtain a new generalization bound in view of the empirical conditional standard deviation. 
\begin{lemma}
Fix a tolerance level $\tau>0$. Then, for any $\bm x\in\mathcal X$, we have
\begin{equation}
\label{eq:error_bound_with_sample_variance}
\left|\EErho[\ell(\bm x,\xit)]-\EErhohat[\ell(\bm x,\xit)]\right|\leq\left(\sqrt{\VVrhohat[\ell(\bm x,\xit)]}+\tau\right) \sqrt{\frac{\log\left(\frac{2}{\delta}\right)}{ n h_n^pg(\bm\gamma)(1+o(1))}}
+{ \frac{\sqrt{\log\left(\frac{2(1+2/\tau)}{\delta}\right)\log\left(\frac{2}{\delta}\right)}}{ n h_n^pg(\bm\gamma)(1+o(1))}},
\end{equation}
with probability at least  $1-\delta$. 
\end{lemma}

\begin{proof}
The bounds in \eqref{eq:error_bound} and~\eqref{eq:bound_sample_stddev} yield
\begin{equation*}
\begin{array}{rl}
\displaystyle\left|\EErho[\ell(\bm x,\xit)]-\EErhohat[\ell(\bm x,\xit)]\right|&\displaystyle\leq \left(\sqrt{\VVrho[\ell(\bm x,\xit)]}-\sqrt{\VVrhohat[\ell(\bm x,\xit)]}+\sqrt{\VVrhohat[\ell(\bm x,\xit)]}\right)\sqrt{\frac{\log\left(\frac{1}{\delta}\right)}{ n h_n^pg(\bm\gamma)(1+o(1))}}\\
&\displaystyle\leq \left(\tau
		+\sqrt{ \frac{\log\left(\frac{1+2/\tau}{\delta}\right)}{ n h_n^pg(\bm\gamma)(1+o(1))}}+\sqrt{\VVrhohat[\ell(\bm x,\xit)]}\right)\sqrt{\frac{\log\left(\frac{1}{\delta}\right)}{ n h_n^pg(\bm\gamma)(1+o(1))}}.
\end{array}
\end{equation*}
The above inequality holds with probability at least $1-2\delta$, which completes the proof. 
\end{proof}
\noindent The above lemma shows that the errors introduced by replacing the conditional variance term with its empirical estimates diminish at the faster rate of $ O({1}/( n h_n^p))$, and become negligible when the sample size is large. 

\begin{proof}[Proof of \Cref{thm:optimality_bound}]
\label{sec:thm:optimality_bound}
Applying the union bound to \eqref{eq:error_bound_with_sample_variance} over $\bm x\in\mathcal X$, we find that with probability at least $1-\delta$, 
\begin{multline*}
\displaystyle\left|\EErho[\ell(\bm x,\xit)]-\EErhohat[\ell(\bm x,\xit)]\right|\leq\left(\sqrt{\VVrhohat[\ell(\bm x,\xit)]}+\tau\right) \displaystyle\sqrt{\frac{\log\left(\frac{2|\mathcal X|}{\delta}\right)}{ n h_n^pg(\bm\gamma)(1+o(1))}}\displaystyle+{ \frac{\log\left(\frac{2|\mathcal X|(1+2/\tau)}{\delta}\right)}{ n h_n^p g(\bm\gamma)(1+o(1))}}, \\ \quad\forall\bm x\in\mathcal X.
\end{multline*}
Thus, for $\bm x = \hat{\bm  x}$ we get
\begin{equation*}
\begin{array}{rl}
\displaystyle\EErho[\ell(\hat{\bm  x},\xit)]&\displaystyle\leq\;\EErhohat[\ell(\hat{\bm  x},\xit)]+\left(\sqrt{\VVrhohat[\ell(\hat{\bm  x},\xit)]}+\tau\right) \sqrt{\frac{\log\left(\frac{2|\mathcal X|}{\delta}\right)}{ n h_n^pg(\bm\gamma)(1+o(1))}}
+{ \frac{\log\left(\frac{2|\mathcal X|(1+2/\tau)}{\delta}\right)}{ n h_n^pg(\bm\gamma)(1+o(1))}} \\
&\displaystyle\leq\;\EErhohat[\ell({\bm  x}^\star,\xit)]+\left(\sqrt{\VVrhohat[\ell({\bm  x}^\star,\xit)]}+\tau\right) \sqrt{\frac{\log\left(\frac{2|\mathcal X|}{\delta}\right)}{ n h_n^pg(\bm\gamma)(1+o(1))}}
+{ \frac{\log\left(\frac{2|\mathcal X|(1+2/\tau)}{\delta}\right)}{ n h_n^pg(\bm\gamma)(1+o(1))}},
\end{array}
\end{equation*}
where the second inequality holds because $\bm x^\star$ is suboptimal for the regularized problem \eqref{eq:regularized_problem}. 
Next, applying the bound~\eqref{eq:error_bound} for $\EErhohat[\ell({\bm  x}^\star,\xit)]$ and the bound \eqref{eq:bound_sample_stddev} for $\sqrt{\VVrhohat[\ell({\bm  x}^\star,\xit)]}$, we obtain 
\begin{equation*}
\begin{array}{rl}
\displaystyle\EErho[\ell(\hat{\bm  x},\xit)]
&\displaystyle\leq\;\EErho[\ell({\bm  x}^\star,\xit)]+\sqrt{\frac{\VVrho[\ell(\bm x^\star,\xit)]}{ n h_n^pg(\bm\gamma)(1+o(1))}\log\left(\frac{3}{\delta}\right)}\\
&\qquad\displaystyle+\left(\sqrt{\VVrho[\ell({\bm  x}^\star,\xit)]} + 2\tau
+\sqrt{ \frac{\log\left(\frac{3+6/\tau}{\delta}\right)}{ n h_n^pg(\bm\gamma)(1+o(1))}}\right) \sqrt{\frac{\log\left(\frac{6|\mathcal X|}{\delta}\right)}{ n h_n^pg(\bm\gamma)(1+o(1))}}\\
& \displaystyle 
\qquad+{ \frac{\log\left(\frac{6|\mathcal X|(1+2/\tau)}{\delta}\right)}{ n h_n^pg(\bm\gamma)(1+o(1))}}.
\end{array}
\end{equation*}
Finally, after performing further algebraic simplifications, we arrive at the desired bound. This completes the proof. 
\end{proof}

% \noindent Theorem \ref{thm:optimality_bound} asserts that if there is an optimal solution~$\bm x^\star$ of the stochastic  problem~\eqref{eq:conditional_expectation_problem} which yields a cost with negligible empirical conditional variance then the regularized solution $\hat{\bm  x}$ will converge to this optimal solution at the faster rate of $O({1}/{ (n h_n^p)})$. Note that the bound  \eqref{eq:suboptimality_bound} grows only logarithmically in the cardinality of the feasible set $\mathcal X$ and, hence, at most linearly in the dimension of the decision vector~$\bm x$. Improved bounds with similar guarantees can be derived for other classes of cost function  and feasible set by taking into account the class' Rademacher complexity and VC dimension \citep{bartlett2002rademacher,vapnik1998statistical,shalev2014understanding}. 

% In our analysis for Theorem \ref{thm:optimality_bound}, we assumed that the feasible set $\mathcal{X}$ is finite. In what follows, we show that under additional mild assumptions, the result can be extended to the setting where $\mathcal{X}$ is a continuous set. 

\section{Proof of \Cref{cor:continuousX}}
\label{sec:thm_optimality_bound_continuosX}
\begin{proof}
Recall that $\x^\star\in \mc{\X}$ minimize the true conditional expectation $\EErho[\ell(\x,\xit)]$ over all $\x \in \X$. Next, consider a fixed parameter $\eta >0$. As before, similar to the proof of \Cref{thm:generalization_bound_continuousX}, we define a finite set of points $\mathcal{X}_\eta \subset \mathcal{X}$ such that $\bm x^\star \in \X_\eta$ and for any $\bm{x}\in\mathcal{X}$, there exists $\bm{x}'\in\mathcal{X}_\eta$ such that
$\Vert \bm{x} - \bm{x}' \Vert \leq \eta.$ From \cite{SN2005}, we know that the cardinality for the set $\vert \mathcal{X}_{\eta} \vert = \mathcal{O}(1) (D/\eta)^d$. 
% Next, we note that since $\ell$ is Lipschitz continuous in $\bm{x}$ {\color{red} for all $ \bm \xi$?}, we have that 
% \begin{equation*}
% \left| \hat{\mb E}_\gamma[\ell(\bm{x},\xit)] - \hat{\mb E}_\gamma\ell(\bm{x}',\xit)] \right| \leq M \Vert \bm{x} - \bm{x}' \Vert, \quad \forall \bm{x}, \bm{x}'\in\mathcal{X}.
% \end{equation*}
% Thus, for any $\bm{x}\in\mathcal{X}$, there exists $\bm{x}'\in\mathcal{X}_\eta$ such that
% \begin{equation*}
% \EErho[\ell(\bm{x},\xit)]  \leq \EErho[\ell(\bm{x}',\xit)] + M\eta.
% \end{equation*}

Let $\hat{\x}\in \X$ and $\hat{\x}'\in \X_\eta$ denote the minimizers of $\hat{\mb E}_\gamma[\ell(\bm{x},\xit)]+\lambda \sqrt{\VVrhohat[\ell(\x,\xit)]}$ over $\X$ and $\X_\eta$ respectively. Next, consider a solution $\hat{\x}'' \in \X_\eta$ such that  $\lVert \hat{\x}-\hat{\x}''\rVert \leq \eta$. Using the result obained in \Cref{lem:lipschitz_loss}, by Lipschitz continuity of $\hat{\mb E}_\gamma[\ell(\hat{\x}'',\xit)]+\lambda \sqrt{\VVrhohat[\ell(\hat{\x}'',\xit)]}$, we have 
\begin{equation}\label{eq:lipschitz-ineq} 
\begin{aligned}
 &\hat{\mb E}_\gamma[\ell(\hat{\x}'',\xit)]+\lambda \sqrt{\VVrhohat[\ell(\hat{\x}'',\xit)]}
 &&\leq \hat{\mb E}_\gamma[\ell(\hat{\x},\xit)]+\lambda \sqrt{\VVrhohat[\ell(\hat{\x},\xit)]} + M'\lVert \hat{\x}'' - \hat{\x}\rVert \\
 &&&\leq \hat{\mb E}_\gamma[\ell(\hat{\x}',\xit)]+\lambda \sqrt{\VVrhohat[\ell(\hat{\x}',\xit)]} + M'\eta \\
 &&&\leq \hat{\mb E}_\gamma[\ell(\x^*,\xit)]+\lambda \sqrt{\VVrhohat[\ell(\x^*,\xit)]} + M'\eta ,
\end{aligned}
\end{equation}
where $M' = (1+\lambda)M$. Furthermore, since $\mathcal{X}_\eta$ is finite, we can apply the same approach as used in the proof of Theorem~\ref{thm:optimality_bound} to obtain  the following result:
\begin{multline*}
\displaystyle\left|\EErho[\ell(\bm x,\xit)]-\EErhohat[\ell(\bm x,\xit)]\right|\leq\left(\sqrt{\VVrhohat[\ell(\bm x,\xit)]}+\tau\right) \displaystyle\sqrt{\frac{\log\left(\frac{2|\mathcal X_\eta|}{\delta}\right)}{ n h_n^pg(\bm\gamma)(1+o(1))}}\displaystyle+{ \frac{\log\left(\frac{2|\mathcal X_\eta|(1+2/\tau)}{\delta}\right)}{ n h_n^p g(\bm\gamma)(1+o(1))}}, \\ \quad\forall\bm x\in \mathcal{X}_\eta
\end{multline*}
with probability at least $1-\delta$.
% We denote $\hat{\bm{x}}'\in\mathcal{X}_\eta$ such that $\EErho[\ell(\hat{\bm{x}},\xit)]  \leq \EErho[\ell(\hat{\bm{x}}',\xit)] + M\eta$. Using the same technique in the proof of Theorem \ref{thm:optimality_bound}, we obtain 
Substituting $\bm x = \hat{\bm  x}''$ and using the same approach as discussed in the proof of Theorem \ref{thm:optimality_bound}, we obtain 
\begin{equation*}
\begin{array}{rl}
\displaystyle\EErho[\ell(\hat{\bm  x}'',\xit)]&\displaystyle\leq\;\EErhohat[\ell(\hat{\bm  x}'',\xit)]+\left(\sqrt{\VVrhohat[\ell(\hat{\bm  x}'',\xit)]}+\tau\right) \displaystyle\sqrt{\frac{\log\left(\frac{2|\mathcal X_\eta|}{\delta}\right)}{ n h_n^pg(\bm\gamma)(1+o(1))}}\displaystyle+{ \frac{\log\left(\frac{2|\mathcal X_\eta|(1+2/\tau)}{\delta}\right)}{ n h_n^p g(\bm\gamma)(1+o(1))}} \\
&\displaystyle\leq\;\EErhohat[\ell({\bm  x}^\star,\xit)]+M'\eta+\left(\sqrt{\VVrhohat[\ell({\bm  x}^\star,\xit)]}+\tau\right) \displaystyle\sqrt{\frac{\log\left(\frac{2|\mathcal X_\eta|}{\delta}\right)}{ n h_n^pg(\bm\gamma)(1+o(1))}}\displaystyle+{ \frac{\log\left(\frac{2|\mathcal X_\eta|(1+2/\tau)}{\delta}\right)}{ n h_n^p g(\bm\gamma)(1+o(1))}},
\end{array}
\end{equation*}
where the second inequality follows from~\eqref{eq:lipschitz-ineq}. Following the steps in Theorem~\ref{thm:optimality_bound}, we obtain
\begin{equation*}
\displaystyle\EErho[\ell(\hat{\bm  x}'',\xit)]
\displaystyle\leq\displaystyle\EErho[\ell({\bm  x}^\star,\xit)]+M'\eta+\left(\sqrt{\VVrho[\ell({\bm  x}^\star,\xit)]}+\tau\right) \sqrt{\frac{4\log\left(\frac{6|\mathcal X_\eta|}{\delta}\right)}{ n h_n^pg(\bm\gamma) (1+o(1))}}+{ \frac{2\log\left(\frac{6|\mathcal X_\eta|(1+2/\tau)}{\delta}\right)}{ n h_n^pg(\bm\gamma)(1+o(1))}},
\end{equation*}
with probability at least $1-\delta$. Furthermore, from Lipschitz continuity of $\ell$,  we get that 
\begin{equation*}
\EErho[\ell(\hat{\bm  x},\xit)]  \leq \EErho[\ell(\hat{\x}'',\xit)] + M\eta.
\end{equation*}
This gives us the final bound below
\begin{equation*}
\displaystyle\EErho[\ell(\hat{\bm  x},\xit)]
\displaystyle\leq\displaystyle\EErho[\ell({\bm  x}^\star,\xit)]+(M+M')\eta+\left(\sqrt{\VVrho[\ell({\bm  x}^\star,\xit)]}+\tau\right) \sqrt{\frac{4\log\left(\frac{6|\mathcal X_\eta|}{\delta}\right)}{ n h_n^pg(\bm\gamma) (1+o(1))}}+{ \frac{2\log\left(\frac{6|\mathcal X_\eta|(1+2/\tau)}{\delta}\right)}{ n h_n^pg(\bm\gamma)(1+o(1))}},
\end{equation*}
\end{proof}

\section{Proof of Proposition~\ref{prop:DRO}}
\begin{proof}
The proof of this proposition follows and generalizes the approach discussed in  \cite{duchi2019variance}. To simplify the notation, we define a random variable $\tilde{z} = \ell(\bm x,\xit)$ and a vector $\bm{z} \in \mathbb{R}^n$ where $z_i = \ell(\bm x,\bm{\xi}^i)$. We denote
\begin{equation*}
\overline{z} =  \hat{\mathbb{E}}_\gamma [ \ell(\bm x,\xit) ] = \sum_{i=1}^n \overline{w}_i \cdot \ell(\bm x,\bm{\xi}^i) 
\quad\text{and}\quad 
s = \hat{\mathbb{V}}_\gamma [\ell(\bm x,\xit) ]  = \sum_{i=1}^n \overline{w}_i \left(\ell(\bm x,\bm{\xi}^i) -\overline{z}\right)^2.
\end{equation*}
%Since $\bm{q}$ depends on $\bm{\gamma}$, $\overline{z}$ and $s$ should also be a function of $\bm{\gamma}$, but to simplify notation, we do not specify it here. Thus, 
The DRO problem $\max_{\mb P\in\mathcal P_\lambda(\hat{\mb P}_{\bm \gamma})}  \mb E_{\mb P} [\ell(\bm x,\xit)]$ can be equivalently written as
\begin{equation*}
\max_{\bm{w}} \left\{ \bm{w}^\top\bm{z} : \sum_{i=1}^n  \frac{1}{2\overline{w}_i} \left(w_i- \overline{w}_i \right)^2 \leq \rho , \; \bm{w}^\top\mathbf{e} = 1, \; \bm{w}\in\mathbb{R}^n_+ \right\},
\end{equation*}
where $\rho = \frac{\lambda^2}{2}$. By change of variable $\bm{u} = \bm{w} -\overline{\bm{w}}$, the above problem is equivalent to
%We perform change of variable. Define $\bm{u} = \bm{p} -\bm{q}$, we have
%\begin{itemize}
%\item $\bm{p}^\top\bm{z} = (\bm{p}+\bm{q}-\bm{q})^\top\bm{z} = \bm{q}^\top\bm{z}  + (\bm{p}-\bm{q})^\top\bm{z} = \overline{z} + \bm{u}^\top\bm{z} =\overline{z} + \bm{u}^\top(\bm{z} - \overline{z}\cdot \mathbf{e})$
%\item $\sum_{i=1}^n  \frac{1}{2q_i} \left(p_i- q_i \right)^2 = \sum_{i=1}^n  \frac{1}{2q_i} \left(u_i \right)^2 := \Vert \bm{u}\Vert_Q^2$ [Since $q_i>0$, this is a norm]
%\item Since $\bm{q}^\top\mathbf{e} = 1$, $\bm{p}^\top\mathbf{e} = 1 \;\;\Longleftrightarrow\;\; \bm{u}^\top\mathbf{e} = 0$
%\item $\bm{p}\in\mathbb{R}^n_+ \;\;\Longleftrightarrow\;\; \bm{u} + \bm{q} \geq \mathbf{0}$
%\end{itemize}
%The problem becomes
\begin{equation*}
\max_{\bm{u}} \left\{ \overline{z} + \bm{u}^\top(\bm{z} - \overline{z}\cdot \mathbf{e}) : \Vert \bm{u}\Vert_Q^2 \leq \rho , \; \bm{u}^\top\mathbf{e} = 0, \; \bm{u} + \overline{\bm{w}} \geq \mathbf{0} \right\},
\end{equation*}
where $\Vert \bm{u}\Vert_W := \sqrt{ \sum_{i=1}^n  \frac{1}{2\overline{w}_i} \left(u_i \right)^2}$ is defined to be a weighted norm. We further define its dual norm $\Vert \bm{u}\Vert_{W^{-1}}^2 := \sqrt{ \sum_{i=1}^n  2\overline{w}_i \left(u_i \right)^2}$, and the upper bound of the above optimization problem is
\begin{equation*}
\overline{z} + \bm{u}^\top(\bm{z} - \overline{z}\cdot \mathbf{e}) \leq  \overline{z} + \Vert \bm{u}\Vert_W \; \Vert \bm{z} - \overline{z}\cdot \mathbf{e} \Vert_{W^{-1}} \leq \overline{z} + \sqrt{\rho}\; \Vert \bm{z} - \overline{z}\cdot \mathbf{e} \Vert_{W^{-1}} = \overline{z} + \sqrt{2\rho s},
\end{equation*}
where the last equality holds because
\begin{equation*}
\Vert \bm{z} - \overline{z}\cdot \mathbf{e} \Vert_{W^{-1}} = \sqrt{\sum_{i=1}^n 2\overline{w}_i (z_i-\overline{z})^2}  = \sqrt{2\hat{\mathbb{V}}_\gamma [\tilde{z}]}.
\end{equation*}
The above upper bound can be achieved by selecting
\begin{equation*}
u_i = \frac{\sqrt{2\rho} \overline{w}_i (z_i - \overline{z})}{\sqrt{s}}.
\end{equation*}
The above choice of $\bm{u}$ satisfies the constraints $\Vert \bm{u}\Vert_W^2 \leq \rho$ and $\bm{u}^\top\mathbf{e} = 0$.
%Note that for such $\bm{u}$,
%\begin{itemize}
%\item $\displaystyle \Vert \bm{u}\Vert_Q^2 = \sum_{i=1}^n \frac{1}{2q_i} \left(\frac{\sqrt{2\rho} q_i (z_i - \overline{z})}{\sqrt{s}}\right)^2 = \sum_{i=1}^n \frac{1}{2q_i} \frac{2\rho q_i^2 (z_i - \overline{z})^2}{s} = \frac{\rho}{s} \sum_{i=1}^n q_i(z_i - \overline{z})^2 = 
%\frac{\rho}{s} s = \rho $
%\item $\displaystyle \bm{u}^\top\mathbf{e} = \sum_{i=1}^n  \frac{\sqrt{2\rho} q_i (z_i - \overline{z})}{\sqrt{s}} = \frac{\sqrt{2\rho}}{\sqrt{s}} \sum_{i=1}^n q_i (z_i - \overline{z}) = 0$.
%\end{itemize}
Therefore, such $\bm{u}$ is feasible as long as 
\begin{equation*}
u_i = \frac{\sqrt{2\rho} \overline{w}_i (z_i - \overline{z})}{\sqrt{s}} \geq - \overline{w}_i  \;\Longleftrightarrow\;  \frac{\sqrt{2\rho} (z_i - \overline{z})}{\sqrt{s}} \geq - 1 .
\end{equation*}
Since $\vert z_i - \overline{z} \vert = \left\vert \ell(\bm x,\bm{\xi}^i) -  \sum_{i=1}^n \overline{w}_i \cdot \ell(\bm x,\bm{\xi}^i)  \right\vert \leq 1$, then a sufficient condition of the above is 
\begin{equation*}
\frac{ 2\rho 1^2}{s} \leq 1 \;\;\Longleftrightarrow\;\; s\geq 2\rho  \;\;\Longleftrightarrow\;\; \sqrt{2\rho s} \geq 2\rho.
\end{equation*}
Thus, if $s - 2\rho  \geq 0$, $\bm{u}$ is a feasible solution. On the other hand, $\bm{u} = \mathbf{0}$ is another feasible solution for this problem. Thus
\begin{equation*}
\overline{z} + \left(\sqrt{2\rho\hat{\mathbb{V}}_\gamma [\tilde{z}]} - 2\rho \right)_+ 
\leq 
\max_{\bm{w}\in\Delta^n} \left\{ \bm{w}^\top\bm{z} : \sum_{i=1}^n  \frac{1}{2\overline{w}_i} \left(w_i- \overline{w}_i \right)^2 \leq \rho  \right\} 
\leq \overline{z} + \sqrt{2\rho\hat{\mathbb{V}}_\gamma [\tilde{z}]} .
\end{equation*}
By letting $\lambda = \sqrt{2\rho}$, we complete this proof.
\end{proof}

\section{Proof of \Cref{prop:empirical_var_large} and its Corollary}\label{app:empirical_var_large}
\begin{proof}
We show that for any fixed tolerance level $\tau>0$, we have
\begin{equation}\label{eq:var_bd_union}
\sqrt{\VVrhohat[\ell(\bm x,\xit)]}
\geq \sqrt{\VVrho[\ell(\bm x,\xit)]} - \tau 
-\sqrt{ \frac{\log\left(\frac{\lvert \mc X_\eta \rvert}{\delta}\right)+\log(1+2/\tau)}{ n h_n^pg(\bm\gamma)(1+o(1))}}-2M\eta \qquad \forall\bm x \in \mc X
\end{equation}
with probability at least $1-\delta$. Here, $\lvert \mathcal X_\eta \rvert=O(1) (D/\eta)^d$. Then, the claim of this proposition should immediately follow as both \eqref{eq:DROequivlent_var_cond} and \eqref{eq:var_bd_union} together imply that $\hat{\mathbb{V}}_\gamma [\ell(\bm x,\xit) ] \geq \lambda^2$, which is sufficient to establish that the regularization scheme is equivalent to the DRO model (cf.~\Cref{prop:DRO}). 
From~\eqref{eq:bound_sample_stddev}, we have that for any fixed $\bm x \in \mc X$, with probability at least $1-\delta$, the following lower bound holds:
\begin{equation}\label{eq:var-lb}
\sqrt{\VVrhohat[\ell(\bm x,\xit)]}\geq \sqrt{\VVrho[\ell(\bm x,\xit)]} - \tau 
-\sqrt{ \frac{\log\left(\frac{1+2/\tau}{\delta}\right)}{ n h_n^pg(\bm\gamma)(1+o(1))}}.
\end{equation}
% Setting $\tau=\sqrt{\VVrho[\ell(\bm x,\xit)]}/2$ and $\delta=1/n$, we get that with probability at least $1-n^{-1}$, we have
% \begin{equation}
% \begin{aligned}
%     \sqrt{\VVrhohat[\ell(\bm x,\xit)]}&\geq \frac{\sqrt{\VVrho[\ell(\bm x,\xit)]}}{2}  
% -\sqrt{ \frac{\log\left(n)+\log(1+2/\tau\right)}{ n h_n^pg(\bm\gamma)(1+o(1))}}\\
% &\geq \frac{\sqrt{\VVrho[\ell(\bm x,\xit)]}}{2} 
% -C\sqrt{ \frac{\log\left(n\right)}{ n h_n^pg(\bm\gamma)(1+o(1))}}.
% \end{aligned}
% \end{equation}
% For a finite solution set $\mc X$, by applying union bound, we get that for any $\bm x\in\mc X$, the above lower bound holds with probability $1-\frac{\lvert \mc X \rvert}{n}$.

% \noindent {\color{blue} Setting $\tau=\sqrt{\VVrho[\ell(\bm x,\xit)]}/2$?}, we get that with probability at least $1-\delta$, we have
% \begin{equation}
% \begin{aligned}
%     \sqrt{\VVrhohat[\ell(\bm x,\xit)]}&\geq \sqrt{\VVrho[\ell(\bm x,\xit)]} - \tau 
% -\sqrt{ \frac{\log\left(\frac{1}{\delta})+\log(1+2/\tau\right)}{ n h_n^pg(\bm\gamma)(1+o(1))}}.
% \end{aligned}
% \end{equation}
\noindent  Next, we define a finite solution set $\mathcal{X}_\eta \subset \mathcal{X}$ such that, for any $\bm{x}\in\mathcal{X}$, there exists some $\bm{x}'\in\mathcal{X}_\eta$ such that
$\Vert \bm{x} - \bm{x}' \Vert \leq \eta.$ From \cite{SN2005}, we know that $\vert \mathcal{X}_{\eta} \vert = O(1) (D/\eta)^d$.
% Since $\mc X_\eta$ is finite, by applying union bound, we get that for any $\bm x\in\mc X$, the following lower bound holds with probability $1-\delta$
% \begin{equation}
% \label{eq:var-lb}
% \begin{aligned}
%     \sqrt{\VVrhohat[\ell(\bm x,\xit)]}&\geq \sqrt{\VVrho[\ell(\bm x,\xit)]} - \tau 
% -\sqrt{ \frac{\log\left(\frac{\lvert \mc X \rvert}{\delta}\right)+\log(1+2/\tau)}{ n h_n^pg(\bm\gamma)(1+o(1))}}.
% \end{aligned}
% \end{equation}
Since the loss function $\ell(\bm x,\tilde{\bxi})$  is $M$-Lipschitz continuous in $\bm{x}$, we have that for any $\bm x \in \mathcal{X}$, there exists some $\bm{x}'\in \mathcal{X}_\eta$, such that $\lVert \bm x - \bm{x}'\rVert\leq \eta$ and the following condition holds:
\begin{equation}\label{eq:var-lipschitz}
\begin{aligned}
 \left\lvert\sqrt{\VVrhohat[\ell(\bm x,\xit)]}-\sqrt{\VVrhohat[\ell(\bm x',\xit)]}\right\rvert &\leq   M\eta.
\end{aligned}
\end{equation}
Next, consider fixed $\bm x \in \mc X$ and $\bm x' \in \mc X_\eta$ such that $\lVert\bm x-\bm x'\rVert\leq \eta $. Combining \eqref{eq:var-lb} and \eqref{eq:var-lipschitz}, we have that 
\begin{equation*}
\begin{aligned}
\sqrt{\VVrhohat[\ell(\bm x,\xit)]} &\geq \sqrt{\VVrhohat[\ell(\bm x',\xit)]} - M\eta \geq \sqrt{\VVrho[\ell(\bm x',\xit)]} - \tau 
-\sqrt{ \frac{\log\left(\frac{1}{\delta}\right)+\log(1+2/\tau)}{ n h_n^pg(\bm\gamma)(1+o(1))}}-M\eta
\end{aligned}
\end{equation*}
with probability at least $1-\delta$. Next, by applying union bound, we get that for all $\bm x' \in \mc X_\eta$,
\begin{equation*}
\sqrt{\VVrhohat[\ell(\bm x',\xit)]} \geq \sqrt{\VVrho[\ell(\bm x',\xit)]} - \tau
-\sqrt{ \frac{\log\left(\frac{\lvert \mc X_\eta \rvert}{\delta}\right)+\log(1+2/\tau)}{ n h_n^pg(\bm\gamma)(1+o(1))}}-2M\eta
\end{equation*}
with probability at least $1-\delta$. Thus, from above, we get that for all $\bm x \in \mc X$, the following bound holds with a probability of at least $1-\delta$:
\begin{equation*}
\begin{aligned}
\sqrt{\VVrhohat[\ell(\bm x,\xit)]}
&\geq \sqrt{\VVrho[\ell(\bm x',\xit)]} - \tau 
-\sqrt{ \frac{\log\left(\frac{\lvert \mc X_\eta \rvert}{\delta}\right)+\log(1+2/\tau)}{ n h_n^pg(\bm\gamma)(1+o(1))}}-M\eta\\
&\geq \sqrt{\VVrho[\ell(\bm x,\xit)]} - \tau 
-\sqrt{ \frac{\log\left(\frac{\lvert \mc X_\eta \rvert}{\delta}\right)+\log(1+2/\tau)}{ n h_n^pg(\bm\gamma)(1+o(1))}}-2M\eta.
\end{aligned}
\end{equation*}
Thus, we proved \eqref{eq:var_bd_union}.
\end{proof}
\begin{corollary}
\label{cor:DRO}
Fix a parameter $\omega>0$ such that $n^{1-\omega}h_n^p$ is increasing in $n$. Then, for all $\bm x \in \mc X$, we have
\begin{equation*}
\sqrt{\VVrhohat[\ell(\bm x,\xit)]} \geq \sqrt{\VVrho[\ell(\bm x,\xit)]}  -\sqrt{ \frac{2}{n^{\omega}g(\bm\gamma)(1+o(1))}} - \frac{2}{\exp\left( n^{1-\omega}h_n^p\right) - 1} - \frac{2MD}{\exp\left( \frac{n^{1-\omega}h_n^p}{2d}\right)}
\end{equation*}
with probability at least $1-C_{\mathcal X}\exp\left( - \frac{n^{1-\omega}h_n^p}{2}\right)$ for some constant $C_{\mathcal X}$. In particular, if the bandwidth $h_n = C_h/n^{1/(p+4)}$ is used with some constant $C_h$, then by setting $\omega = 2/(p+4)$, we obtain $n^{1-\omega}h_n^p = C_h^p n^{2/(p+4)}$, and so
\begin{equation*}
\sqrt{\VVrhohat[\ell(\bm x,\xit)]} \geq \sqrt{\VVrho[\ell(\bm x,\xit)]}  -\sqrt{ \frac{2}{n^{2/(p+4)}g(\bm\gamma)(1+o(1))}} - \frac{2}{\exp\left( C_h^p n^{2/(p+4)} \right) - 1} - \frac{2MD}{\exp\left( \frac{C_h^p n^{2/(p+4)}}{2d}\right)}
\end{equation*}
with probability at least $1-C_{\mathcal X}\exp\left( - \frac{C_h^p n^{2/(p+4)}}{2}\right)$ for some constant $C_{\mathcal X}$.
\end{corollary}

\begin{proof}
From Proposition~\ref{prop:empirical_var_large}, for any $\tau>0$, we have
\begin{equation*}
\sqrt{\VVrhohat[\ell(\bm x,\xit)]} \geq \sqrt{\VVrho[\ell(\bm x,\xit)]} - \tau -\sqrt{ \frac{\log\left(\frac{\lvert \mc X_\eta \rvert}{\delta}\right)+\log(1+2/\tau)}{ n h_n^pg(\bm\gamma)(1+o(1))}}-2M\eta,
\end{equation*}
with probability at least $1-\delta$. Here, $\lvert \mathcal X_\eta \rvert=O(1) (D/\eta)^d$. Suppose we let
\begin{equation*}
\eta = \frac{D}{\exp\left( \frac{n^{1-\omega}h_n^p}{2d}\right)},
\end{equation*}
for some $\omega > 0$. Then for sufficiently large $n$ (that is, when $\eta$ is sufficiently small), we have $\lvert \mathcal X_\eta \rvert= C_{\mathcal X} (D/\eta)^d$, where $C_{\mathcal X}$ is some constant. We then let
\begin{equation*}
\delta = \frac{C_{\mathcal X}}{\exp\left( \frac{n^{1-\omega}h_n^p}{2}\right)}, \;\;\; \text{and} \;\;\; \tau = \frac{2\delta}{C_{\mathcal X} \left(\frac{D}{\eta}\right)^d - \delta} = \frac{2}{\exp\left( n^{1-\omega}h_n^p\right) - 1}.
\end{equation*}
With the above specified parameters, we have 
\begin{enumerate}
\item 
\begin{equation*}
\displaystyle - \tau - 2M\eta = - \frac{2}{\exp\left( n^{1-\omega}h_n^p\right) - 1} - \frac{2MD}{\exp\left( \frac{n^{1-\omega}h_n^p}{2d}\right)}  .
\end{equation*}
\item
\begin{equation*}
\log\left(\frac{\lvert \mc X_\eta \rvert}{\delta}\right) + \log(1+2/\tau) = 2\log\left(\frac{\lvert \mc X_\eta \rvert}{\delta}\right).
\end{equation*}
\item 
\begin{equation*}
\frac{\log\left(\frac{\lvert \mc X_\eta \rvert}{\delta}\right)}{nh_n^p} = \frac{\log\left(\frac{C_{\mathcal X} (D/\eta)^d}{\delta}\right)}{nh_n^p} = \frac{n^{1-\omega}h_n^p}{nh_n^p} = n^{-\omega}.
\end{equation*}
\end{enumerate}
By combining the above three results, we obtain the desired result.
\end{proof}

\section{Proof of \Cref{rem:socp}}
\begin{proof}
We first consider the inner maximization problem, which given a feasible solution~$\x\in\mc X$, yields the distribution with the worst-case expected loss as given below  
\begin{equation*}
\begin{aligned}
\max_{\mb P\in\mathcal P_\lambda(\hat{\mb P}_{\bm \gamma})}  \mb E_{\mb P} [\ell(\bm x,\xit)].
\end{aligned}
\end{equation*}
To simplify the notation, we first let $\eta=\frac{\lambda^2}{2}$. In addition, for a fixed $\bm x \in \mc X$, we define % a random variable $\tilde{z}=\ell(\bm x, \xit)$ and 
the vector $\bm z \in \mb R^n$ whose $i$-th component $z_i = \ell(\bm x,\bxi^i)$ denotes the loss function evaluated for the $i$-th data point. We then have the following formulation for the worst-case expected loss: 
\begin{equation}
\begin{aligned}
&\max &&\sum_{i=1}^n z_i w_i\\
&\;\;\text{s.t.} &&\bm e^\top \bm w = 1, \\
&&& \sum_{i=1}^n \frac{(w_i-\overline{w}_i)^2}{\overline{w}_i} \leq \eta, \\
&&& \bm w \in\mb R^n_+.
\end{aligned}
\end{equation}
We note that the above formulation can be equivalently expressed in terms of the second-order cone constraints below:
\begin{equation}
\begin{aligned}
&\max
&&\bm z^\top \bm w \\
&\;\;\text{s.t.} &&\bm e^\top \bm w = 1, \\
&&& \bigg[\frac{(w_1-\overline{w}_1)}{\sqrt{\overline{w}_1}},\ldots, \frac{(w_n-\overline{w}_n)}{\sqrt{\overline{w}_n}},\sqrt{\eta}\bigg]^\top \in  \SOC(n+1), \\
&&& \bm w \in\mb R^n_+. \\
\end{aligned}
\end{equation}
From the theory of conic programming duality, we know that the second-order cone is self-dual.  We then introduce the dual variables $\alpha\in\mb R$ and $(\bm\beta,\nu) \in \SOC(n+1)$ corresponding to the first two constraints, and write out the Lagrange function in terms of the dual variables, as follows: 
\begin{equation}
\begin{aligned}
\mc L(\bm w,\alpha, \bm \beta, \nu) &= \bm z^\top \bm w + \alpha ( 1 -\bm e^\top \bm w ) +\sum_{i=1}^n \frac{(w_i-\overline{w}_i)}{\sqrt{\overline{w}_i}} \beta_i + \nu \sqrt{\eta}  \\
&= \alpha  + \sum_{i=1}^n \bigg(z_i - \alpha +\frac{\beta_i}{\sqrt{\overline{w}_i}}\bigg) w_i -  \sum_{i=1}^n  \sqrt{\overline{w}_i} \beta_i + \sqrt{\eta} \nu .   
\end{aligned}
\end{equation}

\noindent Thus, we have that the associated Lagrangian dual function is given by
\begin{equation}%\label{eq:dual-func}
\begin{aligned}
g(\alpha, \bm \beta,\nu) &= \max_{\bm w \geq 0} \alpha + \sum_{i=1}^n \bigg(z_i - \alpha +\frac{\beta_i}{\sqrt{\overline{w}_i}}\bigg) w_i -  \sum_{i=1}^n  \sqrt{\overline{w}_i} \beta_i + \sqrt{\eta} \nu  \\
&= \alpha -  \sum_{i=1}^n  \sqrt{\overline{w}_i} \beta_i + \sqrt{\eta} \nu +  \sum_{i=1}^n \max_{w_i \geq 0} \bigg(z_i - \alpha +\frac{\beta_i}{\sqrt{\overline{w}_i}}\bigg) w_i. \\
\end{aligned}
\end{equation}
From above, we have 
\begin{equation}
g(\alpha, \bm \beta,\nu) =
\begin{cases}
\alpha -  \sum_{i=1}^n  \sqrt{\overline{w}_i} \beta_i + \sqrt{\eta} \nu &\text{if }  z_i, +\frac{\beta_i}{\sqrt{\overline{w}_i}} \leq \alpha \\
\infty &\text{otherwise}.
\end{cases}
\end{equation}

\noindent Therefore, the dual problem can be written as:
\begin{equation}
\begin{aligned}
&\min && \alpha -  \sum_{i=1}^n  \sqrt{\overline{w}_i} \beta_i + \sqrt{\eta} \nu   \\
&\;\text{s.t.} && \alpha \geq \ell(\bm x, \bm \xi^i) +\frac{\beta_i}{\sqrt{\overline{w}_i}} & \forall i\in [n],  \\
&&& \bm x \in \mc{\X},\;  \alpha\in\mb R,\;  (\bm\beta,\nu) \in \SOC(n+1).
\end{aligned}
\end{equation}
\end{proof}

\section{Details of \Cref{ex:portOpt_compare}}\label{app:portOptExample}
The proposed regularized NW approximation of the portfolio optimization problem is given by
\begin{equation*}
\max_{\bm x\in\mathcal X}\; \EErhohat[\xit^\top\bm x]-\lambda\sqrt{\VVrhohat\left[\xit^\top\bm x\right]}.
\end{equation*}
By applying \Cref{coro:SOCP}, the above problem can equivalently be reformulated as the  second-order cone program
\begin{equation*}
\begin{array}{cl}
\textup{max} &\displaystyle \left(\sum_{i=1}^n \overline{w}_{i}{(\bm\xi^i)}^\top\bm x\right)-\lambda\rho\\
\textup{s.t.} &\displaystyle \bm x\in\mathcal X, \;t\in\RR,\\
& \left(\sqrt{\overline{w}_{1}}({(\bm\xi^1)}^\top\bm x-t),\ldots,\sqrt{\overline{w}_{n}}({(\bm\xi^n)}^\top\bm x-t),\rho\right)\in\SOC(n+1),
\end{array}
\end{equation*}
where $\overline{w}_{i}=\frac{\mathcal K_{h}(\bm\gamma-\bm\gamma^i)}{\sum_{j=1}^n\mathcal K_{h}(\bm\gamma-\bm\gamma^j)}$, $i\in[n]$. In this example, we select $\lambda$ such that the model provides the best out-of-sample performance.

On the other hand, the LDR approach seeks for the best parameters $x_1,x_2,y \in \mathbb R$ so that the decision of investing $\$(x_i + \gamma \cdot y)$ in asset $i$, for $i=1,2$, and $\$(1-x_1-x_2 - 2\gamma \cdot y)$ in asset~$3$ generates the highest empirical return. The optimal portfolio allocation thus constitutes an  affine function in~$\gamma$. To find these parameters, we solve the following regularized empirical maximization problem:
\begin{equation}
\begin{array}{rll}
   \displaystyle \max%_{(\bm x , y)\in\mathcal X_{\text{LDR}}}
   &\displaystyle \frac{1}{n}\sum_{i=1}^n \left( \bm\xi_i^\top \left[ \bm x +  (\gamma_i \cdot y)\cdot \mathbf e \right]   \right) - \lambda (x_1^2 + x_2^2 + y^2)\\
   \st&  x_1,x_2,y\in\RR,\\
   &x_1 + \gamma_i \cdot y  \geq  0 , \; x_2 + \gamma_i \cdot y  \geq  0 , \; x_1 + x_2 + 2 \gamma_i \cdot y  \leq  1 \qquad \forall i \in [n].
\end{array}
\end{equation}
The constraints of this problem prohibit short selling and  ensure that the total allocation does not exceed $\$1$. In this example, similar to the NW approximation, we select $\lambda$ such that the this model provides the best out-of-sample performance.

\end{onehalfspace}

\end{document}